\documentclass{article}[12pt]

\usepackage{
	amsmath,
	amssymb,
	amsthm,
	bbm,
	enumerate,
	nicefrac,
	thm-restate,
	tikz
}
\usepackage[paper=letterpaper,margin=1in,twoside=false,includehead,headheight=18pt]{geometry}

\usetikzlibrary{calc}
\usetikzlibrary{positioning}
\usetikzlibrary{decorations.markings}

\newtheoremstyle{plain}{}{}{\itshape}{0pt}{\bfseries}{.}{5pt plus 1pt minus 1pt}{}
\theoremstyle{plain}
\newtheorem{theorem}{\newline Theorem}[section]
\newtheorem{lemma}[theorem]{\newline Lemma}
\newtheorem{proposition}[theorem]{\newline Proposition}
\newtheorem{corollary}[theorem]{\newline Corollary}

\theoremstyle{definition}
\newtheorem{definition}[theorem]{Definition}
\newtheorem{remark}[theorem]{Remark}

\newcommand{\pr}{\mathbf{Pr}}
\newcommand{\ex}{\mathbf{E}}
\newcommand{\AND}{\quad \text{and} \quad}
\newcommand{\ba}{{\bf a}}
\newcommand{\bb}{{\bf b}}

\newcommand{\mbf}[1] {\text{\boldmath$#1$}}
\newcommand{\bah}{{\bf \hat{a}}}
\newcommand{\bbh}{{\bf \hat{b}}}
\newcommand{\chio}{{\chi_o}}
\DeclareMathOperator{\poly}{poly}
\DeclareMathOperator{\Ker}{Ker}
\DeclareMathOperator{\rank}{rank}
\DeclareMathOperator{\Span}{span}

\newcommand*{\vcenteredhbox}[1]{\begingroup\setbox0=\hbox{#1}\parbox{\wd0}{\box0}\endgroup}

\begin{document}

\title{The oriented chromatic number of random graphs of bounded degree}
\date{\today}
\author{Karen Gunderson\\
\small{University of Manitoba}\\
\small{\texttt{karen.gunderson@umanitoba.ca}}
 \and
JD Nir\\
\small{Oakland University}\\
\small{\texttt{jdnir@oakland.edu}}}
\maketitle

\begin{abstract}
The chromatic number of the random graph $\mathcal{G}(n,p)$ has long been studied and has inspired several landmark results. In the case where $p = d/n$, Achlioptas and Naor showed the chromatic number is asymptotically two-point concentrated. Kemkes et al.~later proved the same result holds for $\mathcal{G}(n,d)$, the random $d$-regular graph.  

We consider the oriented chromatic number of the directed models $\vec{\mathcal{G}}(n,p)$ and $\vec{\mathcal{G}}(n,d)$. Previous extremal results can be used to bound the oriented chromatic number of a random $d$-regular digraph between $\Omega(\sqrt{2}^d)$ and $O(d^2 2^d)$. Using colourings by doubly regular tournaments, we improve the upper bound to $O(\sqrt{e}^d)$. As part of our proof, we extend an optimization result of Achlioptas and Naor for functions over doubly stochastic matrices which may be of independent interest. 
\end{abstract}

{\small KEYWORDS: random graph, bounded degree, oriented colouring}

\section{Introduction}

The oriented chromatic number of an oriented graph is a natural generalization of the usual chromatic number.  The chromatic number of an undirected graph $G=(V,E)$, denoted $\chi(G)$, is the smallest $k$ for which there is a function $c:V \to [k]$ such that for every $uv \in E,\ c(u) \ne c(v)$.  The chromatic number of $G$ can equivalently be defined as the smallest $k$ for which there is a graph homomorphism from $G$ to a complete graph $K_k$.  The oriented chromatic number of an oriented graph $\vec{G}$, denoted $\chio(\vec{G})$, is the smallest $k$ for which there is a homomorphism (of oriented graphs) from $\vec{G}$ to some tournament of order $k$.  An equivalent formulation is that $\chio(\vec{G})$ is the least $k$ for which there is a partition of $V$ into $k$ sets, called colour classes, so that each colour class contains no arcs and so that between every pair of colour classes, all arcs have the same direction.  The notion of such colourings of oriented graphs was first introduced by Courcelle~\cite{Co94} in a series of papers studying graph properties described in monadic second order logic.

While $\chio(\vec{G})$ is always bounded from below by the usual chromatic number of the underlying graph, these parameters can, in general, be far.  It was noted by Kostochka, Sopena, and Zhu~\cite{KSZ97} that there is an orientation of the complete bipartite graph $K_{n, n}$ with oriented chromatic number $2n$.  Oriented graphs, $\vec{G}$, that have $\chio(\vec{G})$ equal to the number of vertices in $\vec{G}$ are sometimes called \emph{oriented cliques} and are precisely the oriented graphs with diameter at most $2$.  The problem of determining the minimum number of arcs in a directed graph with diameter at most $2$ was proposed by Erd\H{o}s, S\'{o}s, and R\'{e}nyi~\cite{ERS66} (see also~\cite{Zn70, DM87}) and F\"{u}redi, Horak, Pareek, and Zhu~\cite{FHPZ98} showed that this minimum is $(1+o(1))n \log n$ (see also~\cite{KLSS99}).  The maximum order of outerplanar and planar oriented cliques was studied by Klostermeyer and MacGillivray~\cite{kM04b} and Sen~\cite{Se12}.

Another closely-related concept is the \emph{acyclic chromatic number} of an oriented graph, denoted $\chi_a(\vec{G})$, which is the smallest number of colours needs to properly colour the vertices of an oriented graph so that every pair of colour classes induce an acyclic graph.  Raspaud and Sopena~\cite{RS94} showed that $\chio(\vec{G})$ can be bounded in terms of $\chi_a(\vec{G})$ and Kostochka, Sopena and Zhu~\cite{KSZ97} gave bounds on $\chi_a(\vec{G})$ in terms of $\chio(\vec{G})$.

There has been significant prior investigation into $\chio(\vec{G})$ focusing on extremal results for graphs with bounded degree properties. In a recent series of papers~\cite{Du20, DMS19}, Duffy, with others, prove that the oriented chromatic number of connected cubic graphs is at most eight.  Oriented graphs with maximum degree $4$ have an oriented chromatic number at most $69$ (see~\cite{DMS19}) although it is believed that this upper bound can be improved. In~\cite{KSZ97}, Kostochka, Sopena, and Zhu consider deterministic bounds on $\chio(\vec{G})$ in terms of its maximum degree. They show that for any oriented $\vec{G}$ with maximum degree $d$, $\chio(\vec{G}) \le 2d^22^d$. They also prove that when $n$ is sufficiently large, for every $d$-regular graph $G$ on $n$ vertices, there is an orientation $\vec{G}$ with  $\chio(\vec{G}) \ge \lceil 2^{d/2} \rceil$. 

In this paper, we consider the question of bounding the oriented chromatic number of a random directed graph, extending results on the chromatic number of undirected random graphs.
The random graph $\mathcal{G}(n,p)$ was introduced in 1960 by Erd\H{o}s and R\'enyi~\cite{ER60} and independently by Gilbert~\cite{Gi59}.
In this model, each possible edge of a graph on $n$ vertices is chosen independently with probability $p$. This model marks the beginning of the field of random graph theory, the study of random graph models and their properties. We direct the reader to the books~\cite{AS16, Bo01, JLR00} for results on random graphs and their connections to other fields. In addition to $\mathcal{G}(n,p)$, we consider the model $\mathcal{G}(n,d)$ in which a $d$-regular graph on $n$ vertices is selected uniformly at random.

The two models of random directed graphs that we consider in this paper are $\vec{\mathcal{G}}(n,p)$ and $\vec{\mathcal{G}}(n,d)$.  In  each, after selecting an appropriate undirected graph according to $\mathcal{G}(n, p)$ or $\mathcal{G}(n, d)$, respectively, each edge is given an orientation uniformly at random. In particular, $\vec{\mathcal{G}}(n,p)$ and $\vec{\mathcal{G}}(n,d)$ contain no directed loops or two-cycles as they are orientations of simple graphs.

The question of the chromatic number of a random graph was mentioned by Erd\H{o}s and R\'enyi~\cite{ER60} and investigated as early as 1974 by Grimmet and McDiarmid~\cite{GM75}, but the topic has received considerable attention since 1987 following the breakthrough result of Shamir and Spencer~\cite{SS87} in which martingales were used to show $\chi(\mathcal{G}(n,p))$ is concentrated in a small set of values with probability tending to $1$ as $n$ tends to infinity (which we refer to as \emph{asymptotically almost surely} or a.a.s.). Bollob\'as~\cite{Bo88} and later \L uczak~\cite{Lu91a} prove that for $pn \to \infty$, a.a.s. $\chi(\mathcal{G}(n,p))$ is $(\frac{1}{2}+o(1))\log(1/(1-p))\frac{n}{\log n}$. \L uczak also proved~\cite{Lu91b} that when $p = d/n$, so that on average each vertex has degree $d$, there is a value $k_d$ such that the chromatic number of $G(n,p)$ is a.a.s. $k_d$ or $k_d+1$. In 2005, Achlioptas and Naor~\cite{AN05} determined $k_d$ to be the smallest integer such that $d < 2k_d\log k_d$. Building on that work, Achloptas and Moore~\cite{AM04} showed that the $\chi(\mathcal{G}(n,d))$ is a.a.s. $k_d, k_d+1$, or $k_d+2$ for the same $k_d$. Kemkes et al.~\cite{KPW10} improved that result to $k_d$ or $k_d+1$, and Coja-Oghlan et al.~\cite{CEH13} proved that the chromatic number is concentrated at $k_d$ for sufficiently large $d$.

Directed random graphs have received significantly less attention than undirected models. The aforementioned Kostochka, Sopena, and Zhu bound, $\chio(\vec{G}) \le 2d^22^d$ for oriented graphs with maximum degree $d$, also serves as an upper bound for $\chio(\vec{\mathcal{G}}(n,d))$. Chebolu and Frieze considered Hamilton cycles in random lifts of directed graphs~\cite{CF08}. Bensmail, Duffy, and Sen~\cite{BDS17} noted that a.a.s. $\chio(\vec{\mathcal{G}}(n,p=\frac{1}{2})) = n$. (We expand on this observation in Proposition~\ref{prop:gnp_large_p}.) For $d=1$ or $2$, $\vec{\mathcal{G}}(n,d)$ can be readily calculated exactly, and we do so in Section~\ref{ssec:vsparse_vdense}. Unlike in any known undirected case, $\chio(\vec{\mathcal{G}}(n,2))$ is not concentrated in a single value; instead, a.a.s. $\chio(\vec{\mathcal{G}}(n,2)) \in \{4,5\}$ and each value occurs with positive probability. We believe this difference suggests the oriented chromatic number of directed graphs deserves further study.

\textbf{Our contribution.} We extend the literature on undirected graph models to the directed case by analyzing the chromatic number of the random graph models $\vec{\mathcal{G}}(n,p=\frac{d}{n})$ and $\vec{\mathcal{G}}(n,d)$.

A `doubly regular' tournament $T$ is a tournament in which every vertex has
out-degree $\frac{n-1}{2}$ and for every pair of vertices $v,w$ there are
$\frac{n-3}{4}$ vertices $u$ such that $v \to u$ and $w \to u$.

Now set
\begin{equation}\label{eq:uk}
u_k = \frac{2\log k}{\log k - \log(\frac{k-1}{2})} < \frac{2}{\log 2}\log k
\end{equation}
and
\begin{equation}\label{eq:lk}
\ell_k = \frac{2(k-1)^3}{k(k+1)(k-2)}\log(k-1) \ge 2\left(1-\frac{2}{k+1}\right)\log(k-1).
\end{equation}

The main results of this paper are
\begin{theorem}\label{thm:gnp_main}
Let $d > 0$ be a real number and let $k_1, k_2 \in \mathbb{Z}$ such that there exists a doubly regular tournament of order $k_2$ and $d \in (u_{k_1-1}, \ell_{k_2}]$. Then a.a.s. $\chio(\vec{\mathcal{G}}(n,p=\frac{d}{n})) \in [k_1,2k_2+11]$.
\end{theorem}
and
\begin{theorem}\label{thm:gnd_main}
Let $d\ge 2$ be an integer and let $k_1, k_2 \in \mathbb{Z}$ such that there exists a doubly regular tournament of order $k_2$ and $d \in (u_{k_1-1}, \ell_{k_2}]$. Then a.a.s. $\chio(\vec{\mathcal{G}}(n,d)) \in [k_1,2k_2+11]$.
\end{theorem}

We also use Theorems~\ref{thm:gnp_main} and \ref{thm:gnd_main} to prove the following corollary:
\begin{corollary}\label{cor:main_cor}
Let $d > 1$ be given. Then a.a.s.
\[ \chio(\vec{\mathcal{G}}(n,p=\tfrac{d}{n})) \in (2^{d/2}, 4e^{d/2}+4d+15]. \]
Furthermore, if $d$ is an integer, a.a.s.
\[ \chio(\vec{\mathcal{G}}(n,d)) \in (2^{d/2}, 4e^{d/2}+4d+15] \]
as well.
\end{corollary}

Corollary~\ref{cor:main_cor} allows us to contrast our results with the undirected case. For all $d$, we know $\mathcal{G}(n,p=\frac{d}{n})$ and $\mathcal{G}(n,d)$ are concentrated in at most two values. The gap in our directed result is exponential in $d$. While the bounds of Corollary~\ref{cor:main_cor} may be improvable, the case $d=2$ suggests the oriented chromatic number may not be concentrated in as small a set of values as the unoriented cases.

The lower bound on $\chio$ given by Kostochka, Sopena and Zhu~\cite{KSZ97} in fact showed that for $d$ fixed and $n$ large, every graph $G$ with $n$ vertices and average degree $d$ has some orientation with $\chio(\vec{G}) > 2^{d/2} -1$, roughly corresponding to both lower bounds on $\chio$ in Corollary~\ref{cor:main_cor}.  Neither result follows directly from the other, but the underlying proofs are similar as the lower bound for a worst-case orientation of a fixed graph roughly comes from counting proper oriented colourings.

As part of our argument, we maximize a specific function over doubly stochastic matrices. This is common in investigations of the chromatic number of regular graphs; see~\cite{AM04, KPW10, NP21}, each of which rely on a result of Achlioptas and Naor~\cite{AN05}. However, unlike previous investigations, we need to optimize a function not covered by the original theorem of Achlioptas and Naor. We therefore extended their result in the following proposition which we believe may be of independent interest.

\begin{restatable}{proposition}{optimization} \label{prop:generalized_AN}
Let $k \ge 3$, let $G = (V, E)$ be a connected regular graph on vertex set $[k]^2$, and let $\lambda \in \mathbb{R}$ be the second-largest eigenvalue of the adjacency matrix for $G$. Then, for every $d < \frac{4|E|}{k^3 \lambda} \cdot \frac{(k-1)}{(k-2)} \log(k-1)$, the function $\gamma_d$ defined on doubly stochastic matrices by
\[
\gamma_d(\mbf A) = -\frac{1}{k} \sum_{v \in V(G)} a_v \log a_v + \frac{d}{2}\log\left(\frac{2\sum_{uv \in E(G)} a_u a_v}{k^2}\right)
\]
has a unique maximum value at $\mbf A = \frac{1}{k} J_k$ of $\gamma_d(\frac{1}{k}J_k) = \log k + \frac{d}{2}\log\left(\frac{2|E|}{k^4}\right)$.
\end{restatable}

\textbf{Outline of the argument.} It will be convenient to work with slightly different random graph models. Let $\vec{\mathcal{M}}(n,m)$ be a random multigraph model on graphs of order $n$ generated by selecting one of the $n^2$ ordered pairs of vertices (chosen with replacement) and adding a directed edge (or loop) from the first vertex to the second and then repeating that process to add a total of $m$ arcs. The configuration model $\mathcal{C}(n,d)$, introduced by Bollob\'as~\cite{Bo80}, is a random $d$-regular multigraph model in which a random perfect matching is chosen for a set of $nd$ vertices partitioned into $n$ groups of $d$ points uniformly among all $\frac{(dn)!}{(dn/2)!2^{dn/2}}$ perfect matchings. By contracting each part of the partition to a vertex and adding edges between two vertices if their corresponding parts held adjacent vertices (or a loop if two vertices of the same part were adjacent), one generates a $d$-regular multigraph with loops on $n$ vertices. By orienting each edge before contraction we get the directed configuration model $\vec{\mathcal{C}}(n,d)$.

Note that $\vec{\mathcal{M}}(n,m)$ and $\vec{\mathcal{C}}(n,d)$ are models of \emph{directed} graphs, not \emph{oriented} graphs, as they may contain loops or directed two-cycles. As such, we do not define the oriented chromatic number of these models. Instead, we use random variables to count the number of proper oriented $k$-colourings of $\vec{\mathcal{M}}(n,m)$ and $\vec{\mathcal{C}}(n,d)$. If $\vec{M}\sim \vec{\mathcal{M}}(n,m)$ or $\vec{C} \sim \vec{\mathcal{C}}(n,d)$ contains a loop or directed two-cycle, then these variables equal zero for all $k$. We call such directed graphs \emph{uncolourable}. If a directed graph is not uncolourable, we say it is \emph{colourable}. Note that a colourable directed graph may contain multiple edges as long as those edges are oriented in the same direction. When appropriate, we distinguish those directed graphs which are orientations of simple graphs, which are colourable, from colourable directed graphs in general. In Section~\ref{sec:main}, we relate $\vec{\mathcal{G}}(n,p)$ to $\vec{\mathcal{M}}(n,m)$ and $\vec{\mathcal{G}}(n,d)$ to $\vec{\mathcal{C}}(n,d)$ by conditioning on the event that $\vec{M} \sim \vec{\mathcal{M}}(n,m)$ or $\vec{C} \sim \vec{\mathcal{C}}(n,d)$ is an oritentation of a simple graph and use the results regarding the number of proper oriented $k$-colourings of $\vec{\mathcal{M}}(n,m)$ and $\vec{\mathcal{C}}(n,d)$ to bound $\chio(\vec{\mathcal{G}}(n,p))$ and $\chio(\vec{\mathcal{G}}(n,d))$. 

In Section~\ref{sec:optimization} we prove
Proposition~\ref{prop:generalized_AN}. We develop some results regarding
doubly regular tournaments in Section~\ref{sec:tour_prods}. In Section~\ref{sec:lb} we use first moment arguments to find values of $k$ for which there are a.a.s.~zero proper oriented $k$-colourings of $\vec{\mathcal{M}}(n,m)$ and $\vec{\mathcal{C}}(n,d)$. Determining the smallest values for which such colourings exist with positive probability relies on more delicate second moment arguments. First, we narrow our
consideration to certain nice colourings, simplifying the second
moment calculations. We conduct the second moment calculations in
Section~\ref{sec:ub} which give a value $k$ for which there exist
$k$-oriented-colourings of $\vec{\mathcal{M}}(n,m)$ and $\vec{\mathcal{C}}(n,d)$
with positive probability. In Section~\ref{sec:window} we prove that, conditioned on $\vec{C} \sim \vec{\mathcal{C}}(n,d)$ being an orientation of a simple graph, a.a.s.~there is $k=k(d)$ such that $\chio(\vec{C})$ is contained in an interval that is linear in $k$, as well as an analgous result for $\vec{\mathcal{M}}(n,m)$. We prove
Theorems~\ref{thm:gnp_main} and \ref{thm:gnd_main} in Section~\ref{sec:main} by
relating $\vec{\mathcal{G}}(n,p)$ to $\vec{\mathcal{M}}(n,m)$ and
$\vec{\mathcal{G}}(n,d)$ to $\vec{\mathcal{C}}(n,d)$ before concluding with
additional thoughts in Section~\ref{sec:further_disc}.

\textbf{Notation and conventions.} As previously mentioned, we say that a family of events dependent on $n$ happens asymptotically almost surely, abbreviated a.a.s., if the probability of that event approaches 1 as $n$ approaches infinity. For any integers $m, n$, $I_n$ is an $n \times n$ square identity matrix with entries $1$ along the main diagonal and zero elsewhere, $J_{m,n}$ is an $m \times n$ matrix with every entry $1$, and $0_{m,n}$ is an $m \times n$ matrix with every entry 0. If $m=n$, we write $J_n$ or $0_n$. We take $0^0 = 1$ and $0 \log 0 = 0$. There are a few occasions where a $k \times k$ matrix is treated as interchangeable with a vector of length $k^2$ in a tensor product space. Such instances will be highlighted when it may cause confusion. Notation for tensor products of vector spaces and matrices follows, for example, that used in~\cite{sR08}.  Use of the tensor product (sometimes called Kronecker product of matrices in the context of graph products) follows usage in~\cite{GR01}. We frequently use Stirling's formula $n! = \xi(n)(n/e)^n$ where $\xi(n)$ is a function satisfying $1 \le \xi(n)$ and $\lim_{n \to \infty} \xi(n) = \sqrt{2\pi n}$.

\section{Optimization Over Doubly Stochastic Matrices} \label{sec:optimization}

In this section we prove Proposition~\ref{prop:generalized_AN} which describes how to maximize certain functions over doubly stochastic matrices. This optimization result generalizes that given by Achlioptas and Naor~\cite{AN05} in their proof of the concentration of the usual chromatic number in random graphs.

For every $k \ge 3$ and constant $c$, Achlioptas and Naor define a function $g_c$ on $k \times k$ doubly stochastic matrices $\mbf A = (a_{ij})$ by
\begin{equation}\label{eq:AN-function}
g_c(\mbf A) = - \frac{1}{k} \sum_{i = 1}^k \sum_{j = 1}^k a_{ij}\log a_{ij} + c \log\left(1 - \frac{2}{k} + \frac{1}{k^2}\sum_{i = 1}^k\sum_{j = 1}^k a_{ij}^2\right)
\end{equation}
and show that for $c \le \frac{(k-1)^3}{k(k-2)}\log(k-1)$, $g_c$ is maximized in its domain by the matrix $\frac{1}{k} J_k$ that has all entries equal to $\frac{1}{k}$.  The first sum in the expression for $g_c$ is viewed as the `entropy part' of the function, while the second term is viewed as a function of the 2-norm of $\mbf A$: $||\mbf A||_2^2 = \sum_{i = 1}^k \sum_{j = 1}^k a_{ij}^2$. In their paper, Achlioptas and Naor, in fact, give a much more general result, replacing the first sum with one of the form $\frac{1}{k} \sum_{i = 1}^k \sum_{j = 1}^k h(a_{ij})$ where $h$ satisfies a number of conditions involving its derivatives.

We use their techniques to show a generalization involving the second term,
replacing the 2-norm with an expression involving the `Lagrangian' of a regular
graph, a function over its edges. The proof uses the key tools developed by Achlioptas and Naor in their
optimization result, which are stated here for reference with some of their
notation.

For any $k > 1$, define a function $f: [1/k, 1] \to \mathbb{R}$ as follows. For any $r \in [1/k, 1]$, set 
\[
x(r) = \frac{1 + \sqrt{(k-1)(kr-1)}}{k}, \qquad \text{and} \qquad y(r) = \frac{1-x(r)}{k-1} =  \frac{k-1 -\sqrt{(k-1)(kr-1)}}{k(k-1)}
\]
defined to be the unique values $x, y \ge 0$ satisfying $x + (k-1)y = 1$ and $x^2 + (k-1)y^2 = r$.  Set
\begin{equation}\label{eq:single-row-opt-fn}
f(r) = -x(r)\log x(r) - (k-1) y(r)\log y(r).
\end{equation}
\begin{theorem}[see {\cite[Theorem 9]{AN05}}]\label{thm:AN-entropy-ub}
Let $k > 1$ and $1 \le \rho  \le k$.  For every row stochastic matrix $\mbf A$ satisfying $||\mbf A||_2^2 = \rho$,
\begin{equation}
-\sum_{i = 1}^k \sum_{j = 1}^k a_{ij} \log a_{ij}
\le \max\left\{ m \log k + (k-m) f\left(\frac{k\rho-m}{k(k-m)}\right) :\ 0 \le m \le \frac{k(k-\rho)}{k-1}  \right\}.
\end{equation}
\end{theorem}
Note that, in their proof, Achlioptas and Noar~\cite{AN05} show that the function $-\sum_{i = 1}^k\sum_{j = 1}^k a_{ij}\log a_{ij}$, subject to the condition $||\mbf A||_2^2 = \rho$, achieves a unique maximum that is of the form given in Theorem~\ref{thm:AN-entropy-ub}, but the corresponding value of $m$, determined by $k$ and $\rho$, is not determined.
\begin{lemma}[see {\cite[Proof of Theorem 7]{AN05}}]\label{lem:eta-min}
For any $k \ge 3$, define a function $\eta: [0, 1-1/k] \to \mathbb{R}$ by 
\[
\eta(x) = \begin{cases}
	k/2	&\text{if $x = 0$}\\
	\frac{f\left(\frac{1}{k}\right) - f\left(\frac{1}{k} + x\right)}{x}	&\text{if $x > 0$}
\end{cases}.
\]
Then, $\eta$ achieves its unique global minimum at $x = \frac{(k-2)^2}{k(k-1)}$ with a value of $\eta\left( \frac{(k-2)^2}{k(k-1)}\right) = \frac{k-1}{k-2} \log(k-1)$.
\end{lemma}
In~\cite{AN05}, the uniqueness of the global minimum in Lemma~\ref{lem:eta-min} is not stated, but is shown in the proof.

\begin{definition}
For any graph $G = (V, E)$, define the \emph{Lagrangian of $G$} to be a function $\mathcal{L}_G$ defined on stochastic vectors in $[0, 1]^V$ by
\[
\mathcal{L}_G(\mbf a) = \sum_{uv \in E(G)} a_u a_v.
\]
\end{definition}
Lagrangians of graphs were studied by Motzkin and Straus~\cite{MS65} who showed that $\mathcal{L}_G$ is maximized by vectors whose support is a maximum clique in $G$, with equal entries on those coordinates.

We are now ready to prove Proposition~\ref{prop:generalized_AN}, which we restate here:

\optimization*

\begin{proof}
Let $\mbf A$ be a doubly stochastic matrix.  In this proof, $\mbf A$ is treated as a vector in the tensor product space $\mathbb{R}^k \otimes \mathbb{R}^k$.  Similarly, the constant $1$ matrix $J_k$ is treated as interchangeable with the vector $1_k \otimes 1_k$.  
Let $M$ be the adjacency matrix for $G$.  Then,
\begin{align}
2 \mathcal{L}_G(\mbf A)
	&=2 \sum_{uv \in E(G)} a_u \cdot a_v  \notag\\
	&=\sum_{u \in V(G)} \sum_{v \in N(u)} a_u \cdot a_v  \notag\\ 
	&= {\mbf A}^T M {\mbf A}.
\end{align}
Let $r$ be the degree of vertices in $G$.  Since $G$ is connected and $r$-regular, the largest eigenvalue for $M$ is $r$, of multiplicity 1, spanned by the vector $J_k$.  The second-largest eigenvalue is $\lambda < r$.  Since $\mbf A$ is doubly stochastic, $J_k^T \bullet \mbf A = k$ and so $J_k^T M \mbf A = r J_k^T \mbf A = r k$.  Thus,
\begin{align}
\left({\mbf A} - \frac{1}{k}J_k\right)^T M \left({\mbf A} - \frac{1}{k} J_k\right)
	&= {\mbf A}^T M \mbf A - \left(\frac{1}{k}J_k\right)^T M \mbf A - {\mbf A}^T M \left(\frac{1}{k}J_k\right) + \frac{1}{k^2} J_k^T M J_k \notag \\
	&={\mbf A}^T M \mbf A  - 2r + \frac{1}{k^2} \cdot 2|E| \notag\\
	&={\mbf A}^T M \mbf A  - 2\cdot \frac{2|E|}{k^2} + \frac{2|E|}{k^2} \notag\\
	&={\mbf A}^T M \mbf A - \frac{2|E|}{k^2}. \label{eq:lagrangian-AminusJ}
\end{align}
Note that since $\mbf A$ is double-stochastic, $(\mbf A - \frac{1}{k}J_k) \bullet J_k = 0$ and so
\begin{equation}\label{eq:norm}
|| \mbf A - \tfrac{1}{k} J_k ||_2^2 = ||\mbf A||_2^2 - 2 \mbf A \bullet \left(\frac{1}{k} J_k \right) +  \left(\frac{1}{k} J_k \right) \bullet  \left(\frac{1}{k} J_k \right) = ||\mbf A||_2^2 - \mbf A \bullet \left(\frac{1}{k} J_k\right) = ||\mbf A||_2^2 - 1.
\end{equation}
By the Rayleigh-Ritz theorem from linear algebra, since the vector $\mbf A - \frac{1}{k} J_k$ is orthogonal to the eigenspace of the largest eigenvector for $M$, 
\begin{equation}\label{eq:RR-2nd-eval}
\left({\mbf A} - \frac{1}{k}J_k\right)^T M \left({\mbf A} - \frac{1}{k} J_k \right) \le \lambda ||\mbf A - \tfrac{1}{k}J_k||_2^2.
\end{equation}
Thus,
\begin{align}
2 \mathcal{L}_G({\mbf A}) &= {\mbf A}^T M \mbf A \notag\\
	&=\left({\mbf A} - \tfrac{1}{k}J_k\right)^T M ({\mbf A} - \tfrac{1}{k} J_k) + \frac{2|E|}{k^2}	&&\text{(by eqn.~\eqref{eq:lagrangian-AminusJ})} \notag\\
	&\le \lambda ||\mbf A - \frac{1}{k}J_k||_2^2 + \frac{2|E|}{k^2} &&\text{(by eqn.~\eqref{eq:RR-2nd-eval})} \notag\\
	&= \lambda\left(||\mbf A||_2^2 - 1\right) + \frac{2|E|}{k^2}	&&\text{(by eqn.~\eqref{eq:norm})} \label{eq:ub-lagrangian}
\end{align}
Furthermore, 
\[
2\mathcal{L}_G\left(\tfrac{1}{k}J_k\right) = \frac{2|E|}{k^2} = \lambda\left(||\tfrac{1}{k}J_k||_2^2-1\right) + \frac{2|E|}{k^2}.
\]
Thus, if the function $\nu_d$ defined by on doubly stochastic $k \times k$ matrices by 
\begin{equation}\label{eq:new-fn-optimization}
\nu_d(\mbf A) = -\frac{1}{k} \sum_{v \in V(G)} a_v \log a_v + \frac{d}{2}\log\left(\frac{\lambda}{k^2} \left(||\mbf A||_2^2 - 1 \right) + \frac{2|E|}{k^4}\right)
\end{equation}
is maximized at $\mbf A = \frac{1}{k} J_k$ for a given value of $d$, then the same is true for the function $\gamma_d$.  From here, the optimization tools developed in~\cite{AN05} are used to optimize the sum in the function $\nu_d$, subject to a fixed value of $||\mbf A||_2^2$.
In order to show that for all doubly stochastic $\mbf A$, $\nu_d(\mbf A) \le \nu_d \left(\frac{1}{k} J_k \right)$, fix the value of $||\mbf A||_2^2 = \rho \in [1, k]$.   By Theorem~\ref{thm:AN-entropy-ub}, it suffices to show that for every choice of $\rho \in [1, k]$ and $m \in \left[0, \frac{k(k - \rho)}{k-1}\right]$,
\begin{equation*}
\frac{m \log k}{k} + \frac{(k-m)}{k} f\left(\frac{k\rho-m}{k(k-m)}\right) + \frac{d}{2}\log\left(\frac{\lambda (\rho-1)}{k^2} + \frac{2|E|}{k^4}\right)
	\le \nu_d\left(\tfrac{1}{k} J_k \right)
	=\log k + \frac{d}{2}\log\left(\frac{2|E|}{k^4}\right).
\end{equation*}
Rearranging gives the condition
\begin{align*}
\frac{d}{2} \log\left(1 + \frac{\lambda(\rho - 1)}{2|E|/k^2}\right) &\le \left(1 - \frac{m}{k}\right)\left[\log k - f\left(\frac{k\rho-m}{k(k-m)}\right) \right]\\
	&=\left(1 - \frac{m}{k}\right)\left[f\left(\frac{1}{k}\right) - f\left(\frac{k\rho-m}{k(k-m)}\right) \right].
\end{align*}
Using the inequality $\log(1+x) \le x$, it thus suffices to show that for every $\rho \in [1, k]$ and $m \in \left[0, \frac{k(k - \rho)}{k-1}\right]$,
\[
d \le \frac{2 \left(1 - \frac{m}{k}\right)}{\frac{\lambda(\rho - 1)}{2|E|/k^2}}\left[f\left(\frac{1}{k}\right) - f\left(\frac{k\rho-m}{k(k-m)}\right) \right]
\]
Consider now the right hand side of the above equation.
\begin{align}
\frac{2 \left(1 - \frac{m}{k}\right)}{\frac{\lambda(\rho - 1)}{2|E|/k^2}}\left[f\left(\frac{1}{k}\right) - f\left(\frac{k\rho-m}{k(k-m)}\right) \right]
	&=\frac{2 \left(1 - \frac{m}{k}\right)}{\frac{\lambda(\rho - 1)}{2|E|/k^2}}\left[f\left(\frac{1}{k}\right) - f\left(\frac{1}{k} + \frac{k\rho-k}{k(k-m)}\right) \right] \notag\\
	&=\frac{4|E| \left(1 - \frac{m}{k}\right)}{k^2\lambda(\rho - 1)}\left[f\left(\frac{1}{k}\right) - f\left(\frac{1}{k} + \frac{\rho-1}{k(1-m/k)}\right) \right] \notag\\
	&= \frac{4|E|}{k^3\lambda} \left[\frac{f\left(\frac{1}{k}\right) - f\left(\frac{1}{k} + \frac{\rho-1}{k(1-m/k)}\right) }{\frac{\rho-1}{k(1-m/k)}}\right] \notag\\
	&=\frac{4|E|}{k^3\lambda} \eta\left(\frac{\rho-1}{k(1-m/k)}\right) \notag\\
	&\ge \frac{4|E|}{k^3\lambda} \cdot \frac{k-1}{k-2} \log(k-1). \label{eq:dbound-general}
\end{align}
The final inequality in Equation~\eqref{eq:dbound-general} follows from Lemma~\ref{lem:eta-min} since $0 \le \frac{\rho-1}{1-m/k} \le \frac{\rho - 1}{(\rho-1)/(k-1)} = k-1$ for $0 \le m \le \frac{k(k-\rho)}{k-1}$ and hence $0 \le \frac{\rho-1}{k(1-m/k)} \le 1-\frac{1}{k}$ is in the domain of $\eta$.
Thus, $d \le \frac{4|E|}{k^3\lambda} \cdot \frac{k-1}{k-2} \log(k-1)$ suffices to guarantee that $\nu_d$ is maximized at $\mbf A = \frac{1}{k} J_k$ and hence that the same holds for the function $\gamma_d$.

Since the maximum in Theorem~\ref{thm:AN-entropy-ub} and the minimum in Lemma~\ref{lem:eta-min} are unique, then $d >  \frac{4|E|}{k^3\lambda} \cdot \frac{k-1}{k-2} \log(k-1)$, the function $\nu_d$ attains a unique maximum at $\frac{1}{k}J_k$, and hence the same is true for the function $\gamma_d$.
\end{proof}
To see that Proposition~\ref{prop:generalized_AN} generalizes the optimization result for the function in Equation~\eqref{eq:AN-function}, consider the graph $G = K_k \otimes K_k$, the graph on vertex set $[k]^2$ with $(u,v) \sim (a, b)$ if{f} $u \ne a$ and $v \ne b$.  The graph $G$ is a connected $(k-1)^2$ regular graph.  Since the adjacency matrix of $G$ is $(J_k - I_k) \otimes (J_k - I_k)$ it has as eigenvalues $(k-1)^2$, $1$, and $-(k-1)$.  The eigenspace for $-(k-1)$ is the set of all vectors of the form $1_k \otimes \mbf v$ or $\mbf v \otimes 1_k$, where $\mbf v \bullet 1_k = 0$.  For any such eigenvector and any doubly stochastic matrix $\mbf A$, $(1_k \otimes \mbf v) \bullet \mbf A = (\mbf v \otimes 1_k) \bullet A = 0$ and so, decomposing $\mbf A$ into the basis of eigenvectors for the adjacency matrix of $G$ gives
\[
2 \mathcal{L}_G(\mbf A) = (k-1)^2\cdot ||\tfrac{1}{k} J_k||_2^2 + 1 \cdot ||\mbf A - \tfrac{1}{k}J_k||_2^2 = (k-1)^2 + ||\mbf A||_2^2 - 1 = k^2 - 2k + ||\mbf A||_2^2.
\]
Since $G$ has $\frac{k^2(k-1)^2}{2}$ edges, Proposition~\ref{prop:generalized_AN} says that for $k \ge 3$, if
\[
c = \frac{d}{2} \le \frac{2|E|}{k^3 \cdot 1} \cdot \frac{(k-1)}{(k-2)} \log(k-1) = \frac{(k-1)^3}{k(k-2)} \log(k-1),
\]
then the function
\[
\gamma_d(\mbf A) =  -\frac{1}{k} \sum_{v \in V(G)} a_v \log a_v + \frac{d}{2}\log\left(\frac{2\mathcal{L}_G(\mbf A)}{k^2}\right) =  -\frac{1}{k} \sum_{v \in V(G)} a_v \log a_v + \frac{d}{2}\log\left(\frac{k^2 - 2k + ||\mbf A||_2^2}{k^2}\right)
\]
is maximized at $\mbf A = \frac{1}{k}J_k$.
For the result here, we shall take $G$ to be the Kronecker product of a doubly regular tournament with itself, as defined in Section~\ref{sec:tour_prods}.

\section{Properties of Doubly Regular Tournaments} \label{sec:tour_prods}

In this section we develop several results regarding Kronecker products of tournaments with nice properties. 

Doubly regular tournaments were introduced by Brown and Reid~\cite{RB72}.
\begin{definition}
A tournament $T$ is said to be \emph{doubly regular} if{f} there are integers $k, \ell$ so that for every $v \in V(T)$, $\deg^+(v) = k$ and for every pair of vertices $v, w$, $|N^+(v) \cap N^+(w)| = \ell$.
\end{definition}
Brown and Reid~\cite{RB72} note that if $T$ is a doubly regular tournament of order $n$, then $n \equiv 3 \pmod{4}$, $k = \frac{n-1}{2}$ and $\ell = \frac{n-3}{4}$.  They further relate the existence of a 
doubly regular tournament 
of order $n$ to the existence of a `skew Hadamard matrix' of order $n+1$. A recent survey on the existence of skew-Hadamard matrices was given by Koukouvinos and Stylianou~\cite{KS08}.
A special case of a doubly regular tournament are the `Paley tournaments' (see e.g.~\cite{Ha86}). For each prime $q \equiv 3 \mod 4$, the \emph{Paley} tournament $T_q$ is defined on vertex set $\mathbb{F}_q$ by directing the edge $\{u,v\}$ from $u$ to $v$ if{f} $(v-u)$ is a square modulo $q$ and otherwise directing it from $v$ to $u$. Standard number theoretic arguments give that exactly one of $v-u, u-v \in (\mathbb{F}_q^\times)^2$, so the direction is well-defined, and as
\[ |(\mathbb{F}_q^\times)^2| = \frac{q-1}{2} = \frac{1}{2}|\mathbb{F}_q^\times| \]
we see that $T_q$ satisfies $d^+(v) = d^-(v) = \frac{q-1}{2}$ for every $v \in V(T_q)$.
As shown in~\cite{Ha86}, for each pair of distinct vertices $u,v \in V(T_q)$, there are exactly 
$\frac{q-3}{4}$ 
vertices $w$ that satisfy $u \to w$ and $v \to w$.

\begin{lemma} \label{lem:T_props}
Let $T$ be doubly regular tournament with $k$ vertices.
Then if $u \to v$, the number of oriented triangles containing the edge $u$--$v$
paths of length two in each possible orientation is given in the table below:
\begin{center}
$\begin{array}{c|c|c|c|c}
\text{adjacency} &
\vcenteredhbox{\begin{tikzpicture}
\node[circle,draw,label=u,minimum size=2mm,inner sep=0] (u) at (-0.5,0.866) {};
\node[circle,draw,label=v,minimum size=2mm,inner sep=0] (v) at (0.5,0.866) {};
\node[circle,draw,label=below:w,anchor=south,minimum size=2mm,inner sep=0] (w) at (0,0) {};
\draw[->] (u) -> (v);
\draw[->] (u) -> (w);
\draw[->] (v) -> (w);
\end{tikzpicture}} &
\vcenteredhbox{\begin{tikzpicture}
\node[circle,draw,label=u,minimum size=2mm,inner sep=0] (u) at (-0.5,0.866) {};
\node[circle,draw,label=v,minimum size=2mm,inner sep=0] (v) at (0.5,0.866) {};
\node[circle,draw,label=below:w,anchor=south,minimum size=2mm,inner sep=0] (w) at (0,0) {};
\draw[->] (u) -> (v);
\draw[->] (u) -> (w);
\draw[->] (w) -> (v);
\end{tikzpicture}} &
\vcenteredhbox{\begin{tikzpicture}
\node[circle,draw,label=u,minimum size=2mm,inner sep=0] (u) at (-0.5,0.866) {};
\node[circle,draw,label=v,minimum size=2mm,inner sep=0] (v) at (0.5,0.866) {};
\node[circle,draw,label=below:w,anchor=south,minimum size=2mm,inner sep=0] (w) at (0,0) {};
\draw[->] (u) -> (v);
\draw[->] (w) -> (u);
\draw[->] (v) -> (w);
\end{tikzpicture}} &
\vcenteredhbox{\begin{tikzpicture}
\node[circle,draw,label=u,minimum size=2mm,inner sep=0] (u) at (-0.5,0.866) {};
\node[circle,draw,label=v,minimum size=2mm,inner sep=0] (v) at (0.5,0.866) {};
\node[circle,draw,label=below:w,anchor=south,minimum size=2mm,inner sep=0] (w) at (0,0) {};
\draw[->] (u) -> (v);
\draw[->] (w) -> (u);
\draw[->] (w) -> (v);
\end{tikzpicture}}\\
\hline
&&&&\\
\# w & \displaystyle\frac{k-3}{4} & \displaystyle\frac{k-3}{4} & \displaystyle\frac{k+1}{4} & \displaystyle\frac{k-3}{4} 
\end{array}$
\end{center}
\end{lemma}

\begin{proof}
As noted above, the first value is shown in~\cite{Ha86}. For the second count, there are $d^{+}(u) - 1 = \frac{k-3}{2}$ out-neighbours of $u$ other than $v$. Of those, $v$ is directed towards $\frac{k-3}{4}$ of them, leaving $\frac{k-3}{2}-\frac{k-3}{4} = \frac{k-3}{4}$ of them satisfying $u \to w \to v$. The third and fourth counts follow from similar arguments.
\end{proof}

Recall that the signed adjacency matrix of an oriented graph $D$ on $n$ vertices, $\{v_1, v_2, \ldots, v_n\}$, is an $n \times n$ matrix $A$ with
\[
A_{i,j} = \begin{cases}
	$1$ &\text{if  } v_i \to v_j \in E(G)\\
	$-1$ &\text{if } v_j \to v_i \in E(G),  \text{ and}\\
	$0$	&\text{otherwise}.
\end{cases}
\]

\begin{lemma}\label{lem:eval-dr-tourn}
If $T$ is a doubly regular tournament on $k$ vertices with signed adjacency matrix $M$, then $M$ is diagonalizable over $\mathbb{C}$ and has eigenvalues:
\begin{itemize}
	\item $0$, with the eigenspace equal to the span of $\mathbf{1}$.
	\item $i\sqrt{k}$ with multiplicity $\frac{k-1}{2}$, and
	\item $-i\sqrt{k}$ with multiplicity $\frac{k-1}{2}$
\end{itemize}
\end{lemma}

\begin{proof}
Since the signed adjacency matrix satisfies $M^T = -M$, it is skew symmetric and real skew symmetric matrices are diagonalizable over $\mathbb{C}$.

For the eigenvalues of $M$, consider the product, $M^2$.  If $M = (m_{ij})_{k \times k}$, the $(i,j)$ entry of $M^2$ is $\sum_{\ell = 1}^k m_{i \ell} m_{\ell j}$.  If $i= j$, then $\sum_{\ell = 1}^k m_{i \ell} m_{\ell i} = \sum_{\ell = 1}^k -m_{i \ell}^2 = -(k-1)$.  If $i \ne j$, then for any $\ell \notin \{i, j\}$, $m_{i \ell} m_{\ell j}$ is equal to $1$ if $(i, \ell, j)$ is a directed path (in either direction) and $-1$ otherwise.  By Lemma~\ref{lem:T_props}, this is 
\[
\sum_{\ell = 1}^k m_{i \ell} m_{\ell j} = -\frac{(k-3)}{4} + \frac{(k-3)}{4} + \frac{(k+1)}{4} - \frac{(k-3)}{4} = 1.
\]
Thus $M^2 = J_k  - kI_k$, which has eigenvalues $0$ with multiplicity $1$ and eigenvector $\mathbf{1}$ and $-k$ with multiplicity $k-1$.  Thus, the eigenvalues for $M$ are $0$ with multiplicity $1$ and eigenvector $\mathbf{1}$ and complex eigenvalues $i\sqrt{k}$ and $-i\sqrt{k}$.  Since $tr(M) = 0$, the eigenvalues $i\sqrt{k}$ and $-i\sqrt{k}$ occur with equal multiplicity.
\end{proof}

\begin{definition}\label{def:kroenecker_tournaments}
For any tournament $T$, define the Kronecker product of $T$ with itself, $T^{\otimes 2}$, as the graph with vertex set $V(T)^2$ and edges $(u,v) \sim (x,y)$ if{f} either $u \to x$ and $v \to y$ or $x \to u$ and $y \to v$. Note that while $T$ is a directed graph, we have defined $T^{\otimes 2}$ to be undirected. This is more convenient for our calculations, but it should be noted there is a natural orientation of $T^{\otimes 2}$ where $(u,v) \to (x,y)$ if $u \to x$ and $v \to y$.
\end{definition}	

\begin{proposition}\label{prop:reg_prod_connected}
Let $T$ be a regular tournament of order $k \ge 5$.  Then the Kroenecker product graph $T^{\otimes 2}$ is connected.
\end{proposition}

Note that we prove Proposition~\ref{prop:reg_prod_connected} for regular tournaments, not just doubly regular tournaments. In Section~\ref{ssec:col_not_doubly_regular} we consider colouring with a broader collection of tournaments and discuss the optimality of doubly regular tournaments in the context of our techniques.

\begin{proof}
First, consider two non-adjacent vertices $(u, v), (x, y) \in V(T^{\otimes 2})$ with $u \ne x$ and $v \ne y$.  Without loss of generality, assume that, in $T$, $u \to x$ and $y \to v$.  Define the following:
\begin{align*}
a_{++} &= |N^+(u) \cap N^+(x)|\\
a_{+-} &= |N^+(u) \cap N^-(x)|\\
a_{-+} &= |N^-(u) \cap N^+(x)|\\
a_{--} &= |N^-(u) \cap N^-(x)|
\end{align*}
Since $T$ is regular and $u \to x$, then $a_{++} + a_{+-} = \frac{k-3}{2}$, $a_{-+} + a_{--} = \frac{k-1}{2}$, $a_{+-} + a_{--} = \frac{k-3}{2}$, and $a_{++} + a_{-+} = \frac{k-1}{2}$.  Similarly, define 
\begin{align*}
b_{++} &= |N^+(v) \cap N^+(y)|\\
b_{+-} &= |N^+(v) \cap N^-(y)|\\
b_{-+} &= |N^-(v) \cap N^+(y)|\\
b_{--} &= |N^-(v) \cap N^-(y)|
\end{align*}
Since $y \to v$, then $b_{++} + b_{+-} = \frac{k-1}{2}$, $b_{-+} + b_{--} = \frac{k-3}{2}$, $b_{++} + b_{-+} = \frac{k-3}{2}$, and $b_{+-} + b_{--} = \frac{k-1}{2}$.

If there are any $i, j \in \{+, -\}$ with $a_{ij} > 0$ and $b_{ij} > 0$, then there are vertices $w, z$ (not necessarily distinct) so that the direction of $(u, w)$ matches $(v, z)$ and the direction of $(x, w)$ matches $(y, z)$.  Thus, $(u, v), (w, z), (x, y)$ is a path in 
$T^{\otimes 2}$.
To see that such pair always exists, suppose that one does not.  If $a_{++} \ne 0$, then $b_{++} = 0$ and so $b_{+-} = \frac{k-1}{2}$ and $b_{-+} = \frac{k-3}{2}$ which implies that $b_{--} = 0$.  The fact that $b_{+-}, b_{-+} > 0$ would imply that $a_{+-} = a_{-+} = 0$.  However, this is impossible since $a_{-+} + a_{--} = \frac{k-1}{2} \ne \frac{k-3}{2} = a_{+-} + a_{--}$.

Next, consider the case of vertices $(u, u)$ and $(u, x)$.  Note that the vertices $\{(v, v) \mid v \in V(T)\}$ form a clique in $T^{\otimes 2}$.  Let $z \in V(T^{\otimes 2})\setminus \{u, x\}$.  Then, $(u, u)$ is adjacent to $(z, z)$ and by the previous case, there is a path between $(z, z)$ and $(u,x)$.

For the final case, consider vertices $(u, v), (u, x)$ with $v \ne x, v \ne u$, and $x \ne u$.  By the previous case, there is a path between $(u, v)$ and $(u, u)$ and a path between $(u, x)$ and $(u, u)$.  Thus, $(u, v)$ and $(u, x)$ are in the same component.

Since the graph $T^{\otimes 2}$ is isomorphic to itself under the operation of interchanging the order of pairs, this covers all cases and shows that every pair of vertices in $T^{\otimes 2}$ are joined by a path.
\end{proof}

\begin{remark} \label{rem:k=3}
Note that the result does not hold for $k = 3$ since the unique regular tournament on $3$ vertices is the directed $3$-cycle, whose Kroenecker square is a disjoint union of three copies of $K_3$.

Throughout the paper, it is regularly assumed that $k > 3$ and part of the reason for this is that when the product graph $G$ is a disjoint union of cliques, the optimization result from Proposition~\ref{prop:generalized_AN} holds without any assumptions on $d$ in terms of $k$.  Indeed, suppose that $G$ is a graph on vertex set with $[k]^2$ and $G \cong k K_k$, a disjoint union of $k$ cliques with the property that the vertices of each clique form a transversal in $[k]\times [k]$.  Fix $s_1, s_2, \ldots, s_k$ with $\sum_{i = 1}^k s_i = k$ and restrict to doubly stochastic matrices with the property that the sum of entries on vertices, $V_i$, corresponding to the $i$-th copy of $K_k$ is exactly $s_i$.  Among such matrices,
\begin{align*}
\mathcal{L}_G(\mbf A) &= \sum_{i = 1}^k \sum_{v, w \in V_i} a_v a_w\\
	&= \frac{1}{2}\sum_{i = 1}^k \left(\left(\sum_{v \in V_i} a_v\right)^2 -  \sum_{v \in V_i} a_v^2 \right)\\
	&\le \frac{1}{2} \sum_{i = 1}^k \left(s_i^2 - \frac{s_i^2}{k}\right)	&&\text{(by convexity)}\\
	&= \frac{(k-1)}{2k}\sum_{i = 1}^k s_i^2,
\end{align*}
Then, for any $d$,
\[
\gamma_d(\mbf A) \le -\frac{1}{k}\sum_{v \in V(G)} a_v \log a_v + \frac{d}{2}\log\left( \frac{(k-1)}{2k}\sum_{i = 1}^k s_i^2\right).
\]
Because both functions $-x \log x$ and $\log x^2$ are concave down, the upper bound for $\gamma_d(\mbf A)$ is maximized when $s_1 = s_1 = \cdots = s_k =1$, which gives precisely the value of $\gamma_d(\frac{1}{k}J_k)$.

In particular, this means that the optimization result holds without condition on $d$ for oriented $3$-colourings with the directed $3$-cycle.  However, for the expected number of such $3$-colourings to be bounded away from $0$, $d \le 2$ is required and this case is covered in Section~\ref{ssec:vsparse_vdense} using structural results of underlying graphs and their possible colourings.
\end{remark}

\begin{proposition}\label{prop:tourn_prod_eigenvectors}
Let $T$ be a doubly regular tournament on $k$ vertices with signed adjacency
matrix $M$. Then the adjacency matrix for the graph $T^{\otimes 2}$ is
\[
\frac{1}{2}\left(M \otimes M + (J_k - I_k) \otimes (J_k - I_k)\right)
\]
with eigenvalues and eigenspaces
\begin{itemize}
	\item $\frac{(k-1)^2}{2}$ with multiplicity $1$ and eigenspace spanned by $\mathbf{1} \otimes \mathbf{1}$,
	\item $\frac{-(k-1)}{2}$ with multiplicity $\frac{(k-1)(k+3)}{2}$ and eigenspace consisting of
	\begin{itemize}
		\item a subspace of dimension $2(k-1)$ spanned by $\{\mathbf{u} \otimes \mathbf{1}, \mathbf{1} \otimes \mathbf{u} \mid \mathbf{u} \ne \mathbf{0},\ \mathbf{u} \bullet \mathbf{1} = 0\}$, 
		\item a subspace of dimension $\frac{(k-1)^2}{4}$ spanned by 
		\[
		\{\mathbf{u} \otimes \mathbf{v} \mid \mathbf{u}, \mathbf{v} \text{ are eigenvectors for $M$ for the eigenvalue $i \sqrt{k}$}\},
		\]
		 and
		\item a subspace of dimension $\frac{(k-1)^2}{4}$ spanned by 
		\[
		\{\mathbf{u} \otimes \mathbf{v} \mid \mathbf{u}, \mathbf{v} \text{ are eigenvectors for $M$ for the eigenvalue $-i \sqrt{k}$}\}.
		\]
	\end{itemize}
	\item $\frac{(k+1)}{2}$ with multiplicity $\frac{(k-1)^2}{2}$ spanned by
	\begin{multline*}
	\{\mathbf{u} \otimes \mathbf{v}, \mathbf{v} \otimes \mathbf{u} \mid \mathbf{u} \text{ an eigenvector for $M$ for the eigenvalue $i\sqrt{k}$ and } \\
	\mathbf{v} \text{ an eigenvector for $M$ for the eigenvalue $-i \sqrt{k}$}\}.
	\end{multline*}
\end{itemize}
\end{proposition}

\begin{proof}
To see that the adjacency matrix for $T^{\otimes 2}$ is $\frac{1}{2} \left(M
\otimes M +  (J_k - I_k) \otimes (J_k - I_k)\right)$, note that the entry of
$M \otimes M$ at $((u, v), (x, y))$ is
\begin{itemize}
	\item $0$ if either $u = x$ or $v = y$,
	\item $+1$ if either $u \to x$ and $v \to y$ in $T$ or else $x \to u$ and $y \to v$,
	\item $-1$ otherwise.
\end{itemize}
Since the $((u, v), (x, y))$ entry of $(J_k - I_k) \otimes (J_k - I_k)$ is $1$ if{f} $u \ne x$ and $v \ne y$, and $0$ otherwise, then the $((u, v), (x, y))$ entry of $\frac{1}{2}\left(M \otimes M + (J_k - I_k) \otimes (J_k - I_k)\right)$ is $1$ if{f} either $u \to x$ and $v \to y$ or else $x \to u$ and $y \to v$ in $T$, and $0$ otherwise, exactly matching the adjacency conditions for $T^{\otimes 2}$.

The rest of the Proposition follows from Lemma~\ref{lem:eval-dr-tourn} and the fact that $(J_k - I_k)$ is diagonalizable with eigenvalues $(k-1)$ (with multiplicity $1$ and eigenspace spanned by $\mathbf{1}$) and $-1$ (with multiplicity $(k-1)$). Thus, there is an orthogonal basis for $\mathbb{C}^k$ $\{\mathbf{v}_1 = \mathbf{1}, \mathbf{v}_2, \mathbf{v}_3, \ldots, \mathbf{v}_k\}$ whose elements are all eigenvectors for both $M$ (corresponding to eigenvalues $\lambda_1, \lambda_2, \ldots, \lambda_k$) and $(J_k - I_k)$ (corresponding to eigenvalues $\alpha_1, \alpha_2, \ldots, \alpha_k$).  Therefore, $\{\mathbf{v}_i \otimes \mathbf{v}_j \mid 1 \le i, j \le k\}$ is a basis for $\mathbb{C}^k \otimes \mathbb{C}^k$ consisting of eigenvectors for $\frac{1}{2}\left(M \otimes M + (J_k - I_k) \otimes (J_k - I_k)\right)$.
\end{proof}

Note that is it also possible to use Lemma~\ref{lem:T_props} to show that if $T$ is a doubly regular tournament on $k$ vertices, then $T^{\otimes 2}$ is a strongly regular graph with the property that every vertex has $\frac{(k-1)^2}{2}$ neighbours, every pair of adjacent vertices has $\frac{(k-1)(k-3)}{4} + 1$ common neighbours, and every non-adjacent pair of vertices have $\frac{(k-1)(k-3)}{4}$ common neighbours.  From this, the eigenvalues of $T^{\otimes 2}$ and their multiplicities follow from standard results in algebraic graph theory on strongly regular graphs.  In Section~\ref{ssec:second_moment}, the precise structure of the eigenspaces of the adjacency matrix of $T^{\otimes 2}$ are used and so the more precise result in Proposition~\ref{prop:tourn_prod_eigenvectors} is given.

It is conjectured that for every $n$, there are skew-Hadamard matrices of order $4n$ (see~\cite{Wa71, KS08}). By results of Brown and Reid~\cite{RB72}, this would correspond to the existence of doubly regular tournaments of order $k$, for every $k \equiv 3 \pmod{4}$.  According to Koukouvinos and Stylianou~\cite{KS08}, the smallest unresolved case is the existence of a skew-Hadamard matrix of order 276, corresponding to the existence of a doubly regular tournament of order $275$.  Brown and Reid note that if all of the odd prime factors of $k+1$ are $3\pmod{4}$ and occur to an odd power, then there exist doubly regular tournaments of order $k$.

Using results about the existence of skew-Hadamard matrices due to Williamson~\cite{Wi44}, Brown and Reid~\cite{RB72} note that for any $t_0, t_1, \ldots, t_r$ and primes $p_1, p_2, \ldots, p_r$ with the property that for every $i$, $p_i^{t_i} \equiv 3 \pmod{4}$, there exist doubly regular tournaments of order $n$ when $n = 2^{t_0}(p_1^{t_1} + 1)(p_2^{t_2} + 1) \cdots (p_r^{t_r} + 1) - 1$. 

\begin{corollary}\label{cor:opt_doubly_reg_tour}
Let $k >  3$ be such that there exist doubly regular tournaments of order $k$, fix such a tournament $T$ and set $G = T^{\otimes 2}$.  For every $d \le \frac{2 (k-1)^3}{k(k+1)(k-2)}\log(k-1)$, the function $\gamma_d$ defined on doubly stochastic matrices by
\[
\gamma_d(\mbf A) = -\frac{1}{k} \sum_{v \in V(G)} a_v \log a_v + \frac{d}{2}\log\left(\frac{2\mathcal{L}_G(\mbf A)}{k^2}\right)
\]
has a unique maximum value at $\mbf A = \frac{1}{k} J_k$ of $\gamma_d(\frac{1}{k} J_k) = \log k + \frac{d}{2} \log\left(\frac{(k-1)^2}{2k^2}\right)$.
\end{corollary}
The proof of Corollary~\ref{cor:opt_doubly_reg_tour} follows immediately from Propositions~\ref{prop:generalized_AN} and \ref{prop:tourn_prod_eigenvectors}.

In order to give some explicit bounds on $\chi_o(\vec{\mathcal{G}}(n,p=\frac{d}{n}))$ and $\chi_o(\vec{\mathcal{G}}(n,d))$ in terms of $d$ in Corollary~\ref{cor:main_cor}, it is useful to have some quantitative bounds on the gaps between orders of doubly regular tournaments. In particular, there is a doubly regular tournament of every order that is prime and equivalent to $3 \pmod{4}$. Erd\H{o}s~\cite{Er35} showed that for $n$ sufficiently large, there are always primes equivalent to each of $1 \pmod{4}$ and $3 \pmod{4}$ in the interval $[n, 2n]$, extending Bertrand's postulate. The precise bounds given in the paper by Ramar\'{e} and Rumely~\cite{RR96} can be used to show that this property holds for all $n \ge 2$. Since Paley tournaments are doubly regular, this gives the following result.

\begin{lemma}\label{lem:bertrand_postulate_dr_tourn}
For every $n \ge 2$, there exists $k$ with $n \le k \le 2n$ so that there exists a doubly regular tournament of order $k$.
\end{lemma}

Very precise results about the number of primes in a given arithmetic progression are given by Bennett, Martin, O'Bryant, and Rechnitzer~\cite{BMOBR18} and give a smaller interval in which one can guarantee a prime of a given residue class.  

\section{Lower Bounds} \label{sec:lb}

In this section we use first moment arguments to identify values of $k$ for which there are a.a.s.~zero proper oriented $k$-colourings of $\vec{\mathcal{M}}(n,m)$ and $\vec{\mathcal{C}}(n,d)$.

\begin{proposition}\label{prop:gnm_lb}
Let $c > 0$ be a real number and $k$ be an integer such that $c > \frac{1}{2}u_k$. Then a.a.s.~there are no proper oriented $k$-colourings of $\vec{\mathcal{M}}(n,m=cn)$.
\end{proposition}

\begin{proof}
Let $X$ be a random variable counting the number of proper oriented $k$-colourings of $\vec{\mathcal{M}}(n,m=cn)$. Fix a colouring of $n$ vertices and a tournament on $k$ vertices. Let $a_i$ be the proportion of vertices that received colour $i$. Then the probability that a directed edge connects two vertices of colour $i$ is $a_i^2$, so the probability that it connects different colours is $1 - \sum_{i=1}^k a_i^2$ and the probability it connects two different colours and its orientation matches the tournament is $\frac{1}{2}(1-\sum_{i=1}^k a_i^2)$. By convexity, since the colour proportions must add to one, this probability is at most $\frac{1}{2}(1-\frac{1}{k})$. As each edge is added independently, the probability that the resulting multigraph is properly coloured is at most $(\frac{1}{2}(1-\frac{1}{k}))^m$. Therefore
\[ \ex X \le 2^{\binom{k}{2}} k^n\left(\frac{1}{2}\left(1-\frac{1}{k}\right)\right)^m = 2^{\binom{k}{2}} \left(k\left(\frac{1}{2}\left(1-\frac{1}{k}\right)\right)^c\right)^n. \]
Noting
\[ c > \frac{1}{2}u_k \implies c\log\left(\frac{2k}{k-1}\right) > \log k \implies 1 > k\left(\frac{1}{2}\left(1-\frac{1}{k}\right)\right)^c, \]
we see $\ex X = o(1)$ and thus a.a.s. there are no proper $k$-oriented-colourings of $\vec{\mathcal{M}}(n,m=cn)$.
\end{proof}

\begin{proposition}\label{prop:gnd_lb}
Let $d \ge 2$ and $k$ be integers such that $d > u_k$. Then a.a.s.~there are no proper oriented $k$-colourings of $\vec{\mathcal{C}}(n,d)$.
\end{proposition}

\begin{proof}

Let $X$ be the number of oriented $k$-colourings of a random $d$-regular oriented multigraph $\vec{\mathcal{C}}(n,d)$. We calculate $\ex X$ by summing over pairs $(\vec{C},\chio)$, where $\vec{C} \in \vec{\mathcal{C}}(n,d)$ is a configuration with edges given orientations independently and uniformly at random and $\chio$ is an oriented colouring that respects $\vec{C}$, then dividing by the total number of oriented configurations.

Starting with the denominator, each of the $|\vec{\mathcal{C}}(n,d)|$ many graphs
of the configuration model has $\frac{dn}{2}$ edges, each with one of two
orientations, giving
\[|\mathcal{C}(n,d)| \cdot 2^{\frac{1}{2}dn} = \frac{(dn)!}{(\frac{1}{2}dn)!2^{\frac{1}{2}dn}} \cdot 2^{\frac{1}{2}dn}= \frac{(dn)!}{(\frac{1}{2}dn)!}\]
total oriented configurations.

Now we turn to the numerator. For each $i \in [k]$, let $a_i \in
\frac{1}{n}\mathbb{Z}$ be the proportion of the vertices of $\vec{C}$ that
$\chio$ stipulates receive colour $i$. We have
\begin{equation}\label{eq:acons}
a_i \ge 0 \ \forall i \in [k] \AND \sum_{i \in [k]} a_i = 1.
\end{equation}
Set $\ba = (a_i)_{i \in [k]}$ and note that $\ba \in \frac{1}{n}\mathbb{Z}^k$.

To be a proper oriented $k$-colouring, all edges between colour classes
must have the same orientation which may be chosen arbitarily. Fix a $k$-colouring of
$K_k$ and choose any of the $2^{\binom{k}{2}}$ possible orientations uniformly
at random. Then orient any edges between vertices of colour $i$ and $i'$
according to that orientation.

Having fixed the direction of all the edges, we now need only specify their
endpoints. To generate edges in the configuration model, we replace each vertex
with a collection of $d$ vertices and select a perfect matching on the resulting
$dn$ vertices. The larger collection of vertices inherits a colouring from the
colouring of the base vertices. For each pair of distinct colours $i,i'$, let
$b_{i,i'}\in\frac{1}{n}\mathbb{Z}$ denote the proportion of the vertices that
$\chio$ colours with colour $i$ and are paired with a vertex of colour $i'$. (As $\chio$ is a
proper colouring there are no edges from a vertex of colour $i$ to another
vertex of colour $i$.) Let $K$ be the set $\{(i,i') \in [k]^2 \mid i \ne i'\}$. We have
\begin{align}
b_{i,i'} \ge 0 \quad &\quad  \forall (i,i') \in K \notag\\
b_{i,i'} = b_{i',i} \quad &\quad  \forall (i,i') \in K \notag\\
\sum_{i \ne i'} b_{i,i'} = a_{i'} \quad &\quad \forall i \in [k] \notag\\
\sum_{i' \ne i} b_{i,i'} = a_i \quad & \quad \forall i' \in [k] \label{eq:bcons}
\end{align}
Set $\bb = (b_{i,i'})_{(i,i') \in K}$ and note that $\bb \in \frac{1}{n}\mathbb{Z}^{k(k-1)}$.

To generate a pair $(\vec{C},\chio)$, we first fix $\ba$ and $\bb$. We then colour the $n$
vertices in accordance with $\ba$, which can be done in
\[ \frac{n!}{\prod\limits_{i\in[k]} (a_in)!} \]
many ways. To add edges in accordance with $\bb$ we first select the
appropriate numbers of vertices in each colour class. The configuration model
blows up the $a_in$ vertices of colour $i$ to $a_idn$ vertices, so this
selection can be made in
\[ \prod_{i \in [k]} \frac{(a_idn)!}{\prod\limits_{i' \ne i} (b_{i,i'}dn)!}\]
many ways. Note that each pair of distinct colours $\{i,i'\}$ contributes two terms to the
denominator of this product. We then select a perfect matching between the
vertices of each pair of colours. For a given pair of coulours $(i,i')$, we only
chose one matching, so there are
\[ \prod_{i < i'} (b_{i,i'}dn)! \]
many such selections. Putting these steps together, we see there are
\[ \frac{n!}{\prod_{i \in k} (a_in)!} \prod_{i\in k} \frac{(a_i dn)!}{\prod_{i'
\ne i} (b_{i,i'}dn)!} \prod_{i < i'} (b_{i,i'}dn)! = n!\prod_{i
\in k} \frac{(a_idn)!}{(a_in)!} \prod_{i < i'} \frac{1}{(b_{i,i'}dn)!} \]
many colourings that satisfy $\ba$ and $\bb$. Finally, colour the vertices of $K_k$ with $[k]$ and choose a random orientation of the edges, then orient the edges of $\vec{P}$ in accordance with that colouring; there are $2^{\binom{k}{2}}$ such orientations of $K_k$ from which to choose.

Altogether, keeping in mind (\ref{eq:acons}) and (\ref{eq:bcons}),
\begin{equation} \label{eq:sum_half_bs}
\sum_{i < i'} b_{i,i'} = \frac{1}{2} \sum_{i < i'} 2b_{i,i'} = \frac{1}{2}
\sum_{i < i'} (b_{i,i'} + b_{i',i}) = \frac{1}{2} \sum_{i \in [k]} \sum_{i' \ne i}
b_{i,i'} = \frac{1}{2} \sum_{i \in [k]} a_i = \frac{1}{2},
\end{equation}
and applying Stirling's approximation $x! \sim \sqrt{2\pi x}(x/e)^x$, we get
\begin{align*}
\ex X &= \frac{\sum_{\ba, \bb} 2^{\binom{k}{2}} \cdot n! \prod_{i \in k} \frac{(a_idn)!}{(a_in)!} \prod_{i < i'}
\frac{1}{(b_{i,i'}dn)!}}{\frac{(dn)}{(\frac{dn}{2})!}}\\
	&\sim \frac{\sum_{\ba, \bb} 2^{\binom{k}{2}} \cdot \sqrt{\pi n} (n/e)^n \prod_{i \in k} \frac{\sqrt{d}(a_idn/e)^{a_idn}}{(a_in/e)^{a_in}} \prod_{i < i'} \frac{1}{\sqrt{2\pi b_{i,i'}dn}(b_{i,i'}dn/e)^{b_{i,i'}dn}}}{\frac{(dn/e)^{dn}}{(dn/(2e))^{dn/2}}}\\
	&=\frac{\sum_{\ba, \bb} 2^{\binom{k}{2}} \cdot \sqrt{\pi n} \prod_{i \in k} \sqrt{d}(a_i)^{a_i(d-1)n} \prod_{i < i'} \frac{1}{\sqrt{2\pi b_{i,i'}dn}(b_{i,i'})^{b_{i,i'}dn}}}{2^{dn/2}}\\
	&= \sum_{\ba, \bb} \frac{2^{\binom{k}{2}} \cdot \sqrt{\pi n} \cdot d^{k/2}}{\prod_{i < i'} \sqrt{2\pi b_{i,i'}dn}} \left(\frac{1}{2}\right)^{dn/2}\prod_{i \in [k]}a_i^{a_i(d-1)n} \prod_{i < i'} \frac{1}{b_{i,i'}^{b_{i,i'}dn}}\\
	&= \sum_{\ba, \bb} \poly(n) \prod_{i < i'} \left(\frac{(a_ia_{i'})^{d-1}}{(2b_{i,i'})^d}\right)^{b_{i,i'}n}\\
	&= \sum_{\ba, \bb} \poly(n) e^{nf(\ba, \bb)}
\end{align*}
where
\begin{equation}\label{eq:fab}
f(\ba,\bb) = \sum_{i < i'} b_{i,i'}
\log\left(\frac{(a_ia_{i'})^{d-1}}{(2b_{i,i'})^d}\right)
\end{equation}
and $\poly(n)$ is a function bounded above by a function that is polynomial in $n$.

Let $\bah = (\frac{1}{k})_{i \in [k]}$ be a vector representing an equal
distribution of colours. Define $\bbh = (\frac{1}{k(k-1)})_{(i,i') \in K}$.

As reproved in Shi and Wormald~\cite[Theorem 1.3]{SW07} and originally presented in Molloy and Reed~\cite{Mo92}, the number of configurations respecting a fixed $\ba$ is maximized when $\ba =
\bah$. Furthermore, we claim that if $\ba = \bah$, then $f(\bah,\bb)$ is maximized at $\bb = \bbh$.

Note that
\[ f(\bah, \bb) = \sum_{i < i'} -d b_{i,i'} \log(2b_{i,i'}) -(d-1)\log k. \]
Since the sequence $\{2b_{i, i'}\}_{i < i'}$ forms a probability distribution on a set of size $\binom{k}{2}$, then by the maximum entropy principle (see e.g. \cite[Lemma 15.7.1(i)]{AS16}),
\begin{align*}
    f(\bah, \bb) 
    &= \frac{-d}{2}\sum_{i < i'} (2b_{i, i'})\log(2b_{i, i'}) - (d-1)\log k\\
    &\le \frac{d}{2}\log \binom{k}{2} - (d-1)\log k\\
    &=\frac{d}{2}\log\left(\frac{k-1}{2k}\right) + \log k\\
    &=\log \left(k \left(\frac{1}{2}\left(1 - \frac{1}{k} \right) \right)^{d/2} \right) = f(\bah, \bbh).
\end{align*}

Note that each of the $\frac{dn}{2}$ edges of $G$ is adjacent to two vertices of
different colours and determining those colours for each edge specifies $\bb$.
Thus there are no more than
\[ \left(\frac{dn}{2}\right)^{k(k-1)} = O\left(n^{k(k-1)}\right) \]
possible assignments for $\bb$, so
\[ \sum_{\bb \text{ s.t.} (\ref{eq:bcons})} \exp(nf(\bah,\bb)) =
O\left(n^{k(k-1)}\right)\left(k\left(\frac{1}{2}\left(1-\frac{1}{k}\right)\right)^{\frac{d}{2}}\right)^n. \]
Similarly, there are at most $n^k$ assignments of colours to vertices, each of
which determines an assignment to $\ba$, so
\[ \sum_{\ba \text{ s.t. } (\ref{eq:acons})} \sum_{\bb \text{ s.t. }
(\ref{eq:bcons})} \exp(nf(\ba, \bb)) = O(n^k) \sum_{\bb \text{ s.t. }
(\ref{eq:bcons})} \exp(nf(\bah, \bb)) =
O\left(n^{k^2}\right)\left(k\left(\frac{1}{2}\left(1-\frac{1}{k}\right)\right)^{\frac{d}{2}}\right)^n. \]
Now recalling that $\poly(n) = O(n^m)$ for some $m$,
\[ \ex X = \sum_{\ba \text{ s. t. } (\ref{eq:acons})} \sum_{\bb \text{ s.t. } (\ref{eq:bcons})} \poly(n)e^{nf(\ba,\bb)} = O(n^{k^2+m})\left(k\left(\frac{1}{2}\left(1-\frac{1}{k}\right)\right)^{\frac{d}{2}}\right)^n. \]
Note that
\[ d > u_k \implies d\log\left(\frac{2k}{k-1}\right) > 2\log k \implies 1 >
k\left(\frac{1}{2}\left(1 - \frac{1}{k}\right)\right)^{d/2}. \]
Therefore when $d > u_k$, $\ex X = o(1)$ and there are a.a.s. no $k$-colourings of $\vec{\mathcal{C}}(n,d)$.
\end{proof}

\section{Upper Bounds} \label{sec:ub}

In this section we use the second moment method to find values of $k$ for which oriented colourings of $\vec{\mathcal{M}}(n,m)$ and $\vec{\mathcal{C}}(n,d)$ exist with positive probability. To simplify the
calculations, we narrow our consideration to particularly well-behaved oriented
colourings. If $T$ is a tournament on $k$ vertices, then a $T$-colouring of
$\vec{C}$ is a colouring with a homomorphism onto $T$. We call a $k$-colouring
(oriented or otherwise) \textit{equitable} if there are $\frac{n}{k}$ vertices
of each colour. Note that this condition requires that $n$ is divisible by $k$.
We make that assumption throughout this section and discuss how to modify these
arguments for $n$ not a multiple of $k$ in Section~\ref{ssec:divisibility}.

\subsection{$\vec{\mathcal{M}}(n,m)$}

Let $c > 0 $ be real and $k \ge 3$ be an integer such that there exists a doubly regular tournament of order $k$, denoted $T_k$. Let $Y$ be the number of equitable $T_k$ oriented colourings of $\vec{\mathcal{M}}(n,m=cn)$. The first moment is straightforward to calculate.

\begin{lemma} \label{lem:gnm_first_mom}
\[ \ex Y \sim k^{k/2}(2\pi n)^{-(k-1)/2}\left(k\left(\frac{1}{2}\left(1-\frac{1}{k}\right)\right)^c\right)^n. \]
\end{lemma}

\begin{proof}
As we are only counting equitable colourings, the probability an arc chosen at random is correct is precisely $\frac{1}{2}(1-\frac{1}{k})$. Therefore the probability a colouring is correct is $(\frac{1}{2}(1-\frac{1}{k}))^m$ and
\begin{align*}
\ex Y &= \frac{n!}{((n/k)!)^k}\left(\frac{1}{2}\left(1-\frac{1}{k}\right)\right)^m\\
	&\sim \frac{\sqrt{2 \pi n}}{(\sqrt{2 \pi n/k})^k} \cdot \frac{(n/e)^n}{((n/ek)^{n/k})^k}\left(\left(\frac{1}{2}\left(1-\frac{1}{k}\right)\right)^c\right)^n\\
	&= k^{k/2}(2\pi n)^{-(k-1)/2}\left(k\left(\frac{1}{2}\left(1-\frac{1}{k}\right)\right)^c\right)^n.
\end{align*}
\end{proof}

The second moment will require summing over the possible overlap matrices of
pairs of colourings. To accurately estimate the sum of this lattice, we use the
following Laplace summation technique of P\'erez and the second author~\cite{NP21} based on
work by Greenhill, Janson, and Ruc\'inski~\cite{GJR10}:

\begin{theorem}[{\cite[Proposition 3.4]{NP21}}]\label{thm:laplace_summation_gamma}
Suppose the following:
\begin{enumerate}[(i)]
\item $\Gamma=(V_\Gamma, E_\Gamma)$ is a non-empty bipartite multigraph with at least one cycle.
\item $D$ is the unsigned incidence matrix of $\Gamma$ .
\item $\tau(\Gamma)$ is the number of maximal forests in $\Gamma$.
\item $\mathbb V = \Ker(D) \subseteq \mathbb{R}^{|E_\Gamma|}$ is a vector space of dimension $r$.
\item $\mbf y\in\mathbb{R}^{|V_\Gamma|}$ such that
\begin{equation}
D\mbf x = \mbf y
\label{eq:Dxy}
\end{equation}
is a consistent linear system.
\item $K \subset \mathbb{R}^{|E_\Gamma|}$ is a compact convex set with non-empty interior $K^\circ$.
\item $\phi : K \to \mathbb{R}$ is a continuous function and the maximum of $\phi$ in $K$ subject to~\eqref{eq:Dxy} is attained at a unique maximizer $\mbf{\hat x} \in K^\circ$.
\item $\phi$ is twice continuously differentiable in a neighbourhood of $\mbf{\hat x}$ and $H$ is its Hessian matrix at $\mbf{\hat x}$.
\item $\psi : K_1 \to \mathbb{R}$ is a continuous function on some neighbourhood $K_1 \subseteq K$ of $\mbf{\hat x}$ with $\psi(\mbf{\hat x}) > 0$.
\item For each positive integer $n$,
\[
\mathbb X_n = \left\{ \mbf x\in K \cap \frac{1}{n} \mathbb{Z}^{|E_\Gamma|} : D\mbf x = \mbf y \right\}
\]
is non-empty, and there is a positive real number $b_n$ and a function $T_n : \mathbb X_n \to \mathbb{R}$ such that, as $n \to \infty$,
\[ T_n(\mbf x) = O( b_n e^{n\phi(\mbf x)+o(n)}),\quad \mbf x \in \mathbb X_n \]
and
\[ T_n(\mbf x) = b_n (\psi(\mbf x) + o(1)) e^{n\phi(\mbf x)}, \quad \mbf x \in \mathbb X_n \cap K_1,\]
uniformly for $\mbf x$ in the indicated sets.
\end{enumerate}
Then provided $\det(-H|_{\mathbb V} )\ne 0$, as $n \to \infty$,
\[
\sum_{\mbf x \in \mathbb X_n} T_n(\mbf x) \sim \frac{\psi(\mbf{\hat x})}{\tau(\Gamma)^{1/2} \det(-H|_{\mathbb V})^{1/2}} (2\pi n)^{r/2} b_n e^{n\phi(\mbf{\hat x})}.
\]
\end{theorem}

Using Theorem~\ref{thm:laplace_summation_gamma} and
Proposition~\ref{prop:generalized_AN}, we give the following asymptotic value for the
expected value of the square of the number of oriented colourings.  

\begin{lemma}\label{lem:gnm_second_mom}
For any integer $k > 3$ such that there exists a doubly regular tournament $T_k$ of order $k$ and $0 < c < \frac{1}{2}\ell_k$,
\[ \ex Y^2 \sim \left(\frac{(k-1)^4}{((k-1)^2-2c)^2-4c^2k^2}\right)^{(k-1)^2/4} (\ex Y)^2.\]
\end{lemma}

\begin{proof}
For any oriented graph $\vec{G}$ with $n$ vertices, let $h_1, h_2: V(\vec{G})
\to V(T_k)$ be any two equitable oriented colourings with the tournament $T_k$.
For every $i, j \in [k]$, let $na_{i, j}$ be the number of vertices of $\vec{G}$
that receive colour $i$ with colouring $h_1$ and that receive colour $j$ with
colouring $h_2$. Let $\mbf A=(a_{ij})_{i,j=1}^k$ be the \emph{overlap matrix}
for the pair of colourings. As the two colourings are both equitable,
\begin{align}
\forall (i,j) \in [k]^2,  \qquad  a_{ij} &\ge 0 \notag\\
\forall i \in [k],  \qquad  \sum_{j=1}^k a_{ij} &= \frac{1}{k} \notag\\
\forall j \in [k],  \qquad  \sum_{i=1}^k a_{ij} &= \frac{1}{k} \label{eq:gnm_Acons}
\end{align}
Note that the overlap matrix can be prescribed first and then the arcs in the random oriented graph chosen.

Under the two colourings $h_1, h_2$, every vertex of $\vec{G}$ receives a pair
of colours in $[k]^2$. Let $x, y \in V(\vec{G})$ be two vertices and set $u =
(h_1(x), h_2(x))$, $v = (h_1(y), h_2(y))$. There can be an edge in $\vec{G}$
between $x$ and $y$ only in the case that there are arcs in $T_k$ between
$h_1(x)$ and $h_1(y)$ and between $h_2(x)$ and $h_2(y)$ and that the direction
of the edges matches: either $h_1(x) \to h_1(y)$ and $h_2(x) \to h_2(y)$ or else
$h_1(y) \to h_1(x)$ and $h_2(y) \to h_2(x)$. Thus, in terms of the Kroenecker
product of tournaments (Definition~\ref{def:kroenecker_tournaments}), arcs in
$\vec{G}$ between $x$ and $y$ can occur exactly when $u v \in E(T_k^{\otimes
2})$ and the direction of the arc is prescribed by the tournament $T_k$. 

Thus, given an overlap matrix $\mbf A = (a_{ij})_{i, j=1}^k$, the probability an edge in a random graph $\vec{G}$  is coloured correctly is $\sum\limits_{uv \in E(T^{\otimes 2})}a_ua_v$, so
\begin{align}
\ex Y^2 &= \sum_{\mbf A \text{ s.t. } \eqref{eq:gnm_Acons}} \frac{n!}{\prod_{v \in V(T^{\otimes 2})} (na_v)!} \left(\sum_{uv \in E(T^{\otimes 2})}a_u a_v\right)^m \label{eq:mnm_count}\\
	&=\sum_{\mbf A \text{ s.t. } \eqref{eq:gnm_Acons}} \frac{\xi(n) (n/e)^n}{\prod_{v \in V(T^{\otimes 2})} \xi(na_v) (na_v/e)^{na_v}} \left(\sum_{uv \in E(T^{\otimes 2})}a_u a_v\right)^{cn} \notag\\
	&=\sum_{\mbf A \text{ s.t. } \eqref{eq:gnm_Acons}} \xi(n) \left(\prod_{v \in V(T^{\otimes 2})} \xi(na_v) \right)^{-1} \left(\prod_{v \in V(T^{\otimes 2})}a_v^{-a_v} \left(\sum_{uv \in E(T^{\otimes 2})}a_u a_v\right)^c \right)^n \notag\\
	&=\sum_{\mbf A \text{ s.t. } \eqref{eq:gnm_Acons}} p(n,\mbf A) e^{nf(\mbf A)} \label{eq:exp_sum}
\end{align}
where
\begin{equation}
p(n,\mbf A) = \xi(n) \left(\prod_{v \in V} \xi(na_v) \right)^{-1}
\end{equation}
and
\begin{equation}
f(\mbf A) = -\sum_{v \in V(T^{\otimes 2})} a_v \log a_v + c \log\left(\sum_{uv \in E(T^{\otimes 2})}a_u a_v\right)
\end{equation}

We seek to apply Theorem~\ref{thm:laplace_summation_gamma} to the function $f$. Assign to each constraint of Equation~\eqref{eq:gnm_Acons} of the form $\sum_{j=1}^k a_{ij} = \frac{1}{k}$ a vertex $w_{1,i}$ and to each constraint of the form $\sum_{i=1}^k a_{ij} = \frac{1}{k}$ a vertex $w_{2,j}$. Let $V_\Gamma = \{w_{1,i}, w_{2,j}\}_{i,j=1}^k$, $E_\Gamma = \{\{w_{1,i}, w_{2,j}\} \mid i,j \in [k]\}$, and $\Gamma = (V_\Gamma, E_\Gamma)$. Then $\Gamma \cong K_{k,k}$, and as $k > 3$, $\Gamma$ contains a cycle. The equality constraints in Equation~\eqref{eq:gnm_Acons} are equivalent to $D\mbf x = \frac{1}{k} \mbf{1}_{2k}$ where $D$ is the unsigned incidence matrix of $\Gamma$. Onodera~\cite{On73} proved that the number of spanning trees of $K_{m,n}$ is $m^{n-1}n^{m-1}$, so $\tau(\Gamma) = k^{2k-2}$. As $\Gamma$ is bipartite and connected, standard results (see e.g. \cite[Theorem 8.2.1]{GR01}) give $\rank D = |V_\Gamma|-1 = 2k-1$ and therefore
\[ r = \dim \mathbb{V} = k^2 - \rank D = k^2-2k+1 = (k-1)^2. \]
Setting $\mbf y = \frac{1}{k} \mbf{1}_{2k}$ we see $D \mbf x = \mbf y$ is a consistent linear system; $\mbf x = \frac{1}{k^2} \mbf{1}_{k^2}$ is a solution. Set
\[ K = \{ \mbf x \in \mathbb{R}^{k^2} \mid 0 \le a_v \le \tfrac{1}{k} \} \quad \text{and} \quad K_1 = \{ \mbf x \in \mathbb{R}^{k^2} \mid \tfrac{0.9}{k^2} \le a_v \le \tfrac{1.1}{k^2}\}. \]
Then $K$ is a compact convex set with nonempty interior, $K_1 \subseteq K$, and $K_1$ contains the point $\mbf{\hat{x}} = \frac{1}{k^2}J_k$.  

In order to apply Corollary~\ref{cor:opt_doubly_reg_tour} to show that $\mbf{\hat{x}}$ is the unique maximizer, we perform a change of variables to consider functions on doubly-stochastic matrices.  For any doubly-stochastic $k \times k$ matrix $\mathbf{B} = (\alpha_v)_{v \in V}$, define
\begin{align*}
\gamma(\mathbf{B}) &= f\left(\frac{1}{k} \mbf{B} \right)\\
    &=-\frac{1}{k}\sum_{v \in V} \alpha_v \log \alpha_v + c\log\left(\frac{2\sum_{uv \in E} \alpha_u \alpha_v}{k^2}\right)+\log k - c \log 2.
\end{align*}
Since $\log k - c \log 2$ is constant, by Corollary~\ref{cor:opt_doubly_reg_tour}, for $2c \le \frac{2(k-1)^3}{k (k+1)(k-2)} \log(k-1) = \ell_k$, $\gamma$ is uniquely maximized at $\mathbf{B} = \frac{1}{k}J_k = k \mbf{\hat{x}}$.  Thus, the function $f$, which is continuous on $K^\circ$, is uniquely maximized at $\mbf{\hat{x}} \in K_1$. We have
\[ \frac{\partial f}{\partial a_u} = -\log a_u - 1 +c \cdot \frac{\sum_{w \sim u} a_w}{\sum_{xy \in E} a_xa_y} \]
and therefore
\[ 
\frac{\partial^2 f}{\partial a_u\partial a_v} = 
\begin{cases} 
\displaystyle -\frac{1}{a_u} - c\left(\frac{\sum_{w \sim u}a_w}{\sum_{xy \in E} a_xa_y} \right)^2 & v = u\\ 
\displaystyle -c\cdot \frac{(\sum_{w \sim u}a_w)(\sum_{w \sim v}a_w)}{\left(\sum_{xy \in E}a_xa_y\right)^2} & v \not\sim u \\ 
\displaystyle c\left(\frac{1}{\sum_{xy \in E}a_xa_y} - \frac{(\sum_{w \sim u}a_w)(\sum_{z \sim v}a_z)}{\left(\sum_{xy \in E}a_xa_y\right)^2}\right) & v \sim u 
\end{cases}. 
\]
Let $B$ denote the adjacency matrix for the graph $T^{\otimes 2}$. This gives
\begin{equation}\label{eq:Mnm-hessian}
 H|_{\mbf{\hat{x}}} = -\left(k^2I_{k^2}+4cJ_{k^2}-\frac{4ck^2}{(k-1)^2}B\right).
 \end{equation}
 
Then for any matrix $U$ whose columns form a basis for $\ker(D)$,
\begin{equation}\label{eq:Hessian-det}
 \det(-H|_{\mathbb{V}}) = \frac{\det(U^T(-H)U)}{\det(U^TU)}.
 \end{equation}

We give an explicit expression for the matrix $U$ and the corresponding determinant of the Hessian over the vector space $\ker(D)$. Letting $e_i$ denote the $i$-th unit vector in $\mathbb{R}^k$ and using the natural correspondence between $\mathbb{R}^{k^2}$ and $\mathbb{R}^k \otimes \mathbb{R}^k$, the following is a basis for $\ker(D)$:
\[
\mathcal{B} = \left\{ (e_1 - e_i) \otimes (e_1 - e_j) :\ 2 \le i, j \le k \right\}.
\]
Since the set  $\{(e_1 - e_i) :\ 2 \le i \le k \}$ is linearly independent over $\mathbb{R}^k$, then $\mathcal{B}$ is a set of $(k-1)^2$ linearly independent vectors in $\mathbb{R}^k \otimes \mathbb{R}^k$.  With the appropriate indexing, these are naturally elements in $\ker(D)$. Recall that $D$ is the incidence matrix of the graph $K_{k, k}$. Label the edges of $K_{k, k}$ in a natural way with the elements of the set $[k]^2$. Then, the vector $(e_1 - e_i) \otimes (e_1 - e_j)$ corresponds to the vector in $\mathbb{R}^{k^2}$ whose entries corresponding to the edges $(1, 1)$ and $(i, j)$ are $1$, entries corresponding to the edges $(i, 1)$ and $(1, j)$ are $-1$ and all other entries are $0$.

Define the $k \times (k-1)$ matrix
\[
V = \left[\begin{array}{c}
	\mbf{1}_{k-1}^T \\ \hline
	-I_{k-1}
\end{array}\right].
\]
Then the matrix $U = V \otimes V$ has columns that form a basis for $\ker(D)$.  Note that 
\begin{equation}\label{eq:vtransposev}
V^T V = \mbf{1}_{k-1} \mbf{1}_{k-1}^T + (-1)^2 I_{k-1} = J_{k-1} + I_{k-1}
\end{equation}
 and since every column of $V$ has entries summing to $0$, $J_k V = 0_{k \times (k-1)}$.  

With this in mind, consider the Hessian matrix in Equation~\eqref{eq:Mnm-hessian}. Recall that $M$ is the signed adjacency matrix of the doubly regular tournament $T_k$.
\begin{align}
-H
	&=k^2I_{k^2}+4cJ_{k^2}-\frac{4ck^2}{(k-1)^2}B \notag \\
	&=k^2I_{k^2}+4cJ_{k^2} - \frac{4ck^2}{(k-1)^2} \cdot \frac{1}{2}\left(M \otimes M + (J_k - I_k) \otimes (J_k - I_k)\right)  \notag\\
	&=k^2 \left(1 - \frac{2c}{(k-1)^2}\right) I_k \otimes I_k + 4cJ_k \otimes J_k \notag \\
	& \qquad - \frac{2ck^2}{(k-1)^2}\left(M \otimes M - J_k \otimes I_k  - I_k \otimes J_k + J_k \otimes J_k\right) \notag
\end{align}
Thus, considering the matrix expression in the numerator of Equation~\eqref{eq:Hessian-det}, and using the fact that $J_k V = 0_{k \times (k-1)}$ and hence $V^T J_k = 0_{(k-1) \times k}$,
\begin{align*}
U^T(-H)U
	&=(V^T \otimes V)(-H) (V \otimes V)\\
	&= k^2 \left(1 - \frac{2c}{(k-1)^2}\right) (V^T V \otimes V^T V) - \frac{2ck^2}{(k-1)^2}\left((V^T \otimes V^T) (M \otimes M) (V \otimes V)\right)\\
    &=\left(k^2 \left(1 - \frac{2c}{(k-1)^2}\right) (V^T V \otimes V^T V)\right) \cdot\\
    & \qquad \left(I_{k^2} -  \frac{2c/(k-1)^2}{1 - 2c/(k-1)^2} \cdot  (V^T \otimes V^T )(M \otimes M) (V \otimes V)  (V^TV \otimes V^T V)^{-1}\right)
\end{align*}

By a determinant theorem sometimes attributed to Sylvester (see~\cite[Lemma 8.2.4]{GR01} or \cite[Section 0.8.5]{HJ13}), if $A$ and $B$ are matrices of size $m\times n$ and $n \times m$, respectively, then
$\det(I_m+AB) = \det(I_n+BA).$
\begin{multline}\label{eq:Sylvester-det-swap}
\det(U^T(-H)U)
	=\det\left(k^2 \left(1 - \frac{2c}{(k-1)^2}\right) (V^T V \otimes V^T V)\right)\cdot\\
	 \det\left(I_{k^2} - \frac{2c/(k-1)^2}{1 - 2c/(k-1)^2} \cdot (M \otimes M) (V \otimes V) (V^TV \otimes V^T V)^{-1} (V^T \otimes V^T)\right)
\end{multline}
Note that $V^TV = J_{k-1} + I_{k-1}$ and hence $(V^T V)^{-1} = \frac{1}{k}\left(k I_{k-1} - J_{k-1}\right)$.  Thus,
\[
(V^TV \otimes V^T V)^{-1} = (V^TV)^{-1} \otimes (V^TV)^{-1} = \frac{1}{k^2}\left(k I_{k-1} - J_{k-1}\right) \otimes \left(k I_{k-1} - J_{k-1}\right).
\]
Since
\begin{align*}
MV(k I_{k-1} - J_{k-1}) V^T
	&=M \left(k V V^T - VJ_{k-1}V^T\right)\\
	&=M\left(k \begin{bmatrix} \begin{tabular}{c|c} $(k-1)$	&$-1_{k-1}^T$\\ \hline $-1_{k-1}$	&$I_{k-1}$ \end{tabular}\end{bmatrix}  - \begin{bmatrix} \begin{tabular}{c|c} $(k-1)^2$	&$-(k-1)1_{k-1}^T$\\ \hline $-(k-1)1_{k-1}$	&$J_{k-1}$ \end{tabular}\end{bmatrix}\right)\\
	&=M(kI_k - J_k)\\
	&=kM,
\end{align*}
then
\[
(M \otimes M) (V \otimes V) (V^TV \otimes V^T V)^{-1} (V^T \otimes V^T) = \frac{1}{k^2} (kM) \otimes (kM) = M \otimes M.
\]
Therefore, using Equation~\eqref{eq:Sylvester-det-swap} and the eigenvalues for $M$ given in Lemma~\ref{lem:eval-dr-tourn},
\begin{align*}
\det(U^T(-H)U)
	&=\det\left(k^2 \left(1 - \frac{2c}{(k-1)^2}\right) (V^T V \otimes V^T V)\right) \det\left(I_{k^2} - \frac{2c/(k-1)^2}{1 - 2c/(k-1)^2} \cdot (M \otimes M)\right)\\
	&=\det\left(k^2 \left(\frac{(k-1)^2 - 2c}{(k-1)^2}\right) (V^T V \otimes V^T V)\right) \det\left(I_{k^2} - \frac{2c}{(k-1)^2 - 2c} \cdot (M \otimes M)\right)\\
	&=\left(k^2 \left( \frac{(k-1)^2- 2c}{(k-1)^2}\right)\right)^{(k-1)^2} \det(U^TU) \cdot (1 - 0)^{2k-1} \cdot\\
	&\qquad \left(1 - k \frac{2c}{(k-1)^2 - 2c}\right)^{(k-1)^2/2} \cdot \left(1 + k \frac{2c}{(k-1)^2 - 2c}\right)^{(k-1)^2/2}\\
	&=\left(\frac{k^2}{(k-1)^2}\right)^{(k-1)^2} \det(U^TU) \cdot \left(((k-1)^2 - 2c)^2 - 4k^2c^2\right)^{(k-1)^2/2}
\end{align*}

Thus, from Equation~\eqref{eq:Hessian-det},
\begin{align*}
 \det(-H|_{\mathbb{V}}) 
  &= \left(\frac{k^2}{(k-1)^2}\right)^{(k-1)^2} \det(U^TU) \cdot \left(((k-1)^2 - 2c)^2 - 4k^2c^2\right)^{(k-1)^2/2}/ \det(U^TU)\\
  &=  \left(\frac{k^2}{(k-1)^2}\right)^{(k-1)^2}  \cdot \left(((k-1)^2 - 2c)^2 - 4k^2c^2\right)^{(k-1)^2/2}.
\end{align*}
Note that for $k > 1$, $\det(-H|_{\mathbb{V}})$ has zeros at $k = (c+1) \pm \sqrt{c^2+4c}$, so for $c < \frac{1}{2}\ell_k = o(k), \det(-H|_{\mathbb{V}}) > 0$.

Let $\psi = \prod_{v\in V} a_v^{-1/2}$. Then $\psi$ is continuous on $K_1$ and $\psi(\mbf{\hat{x}}) = k^{k^2} > 0$. Set $b_n = (2 \pi n)^{-(k^2-1)/2}$ and $T_n(\mbf x) = p(n,\mbf{x})e^{nf(\mbf{x})}$. Then as $1 \le \xi(n) \sim \sqrt{2\pi n}$,
\[ p(n,\mbf{x}) = \xi(n)\left(\prod_{v \in V} \xi(na_v)\right)^{-1} = O(\sqrt{n}) \]
so
\[ p(n, \mbf{x})/b_n = O(\sqrt{n}/b_n) = e^{o(n)} \]
and
\[ T_n(\mbf x) = b_n(p(n,\mbf{x})/b_n)e^{nf(\mbf x)} = O(b_ne^{n\phi(\mbf{\hat{x}})+o(n)}). \]
Furthermore, for $\mbf{x} \in K_1$, 
\[ p(n, \mbf x) = b_n(\psi(\mbf x)+o(1)). \]
Therefore Theorem~\ref{thm:laplace_summation_gamma} gives
\begin{align*}
\ex Y^2 = \sum_{\mbf x \in \mathbb{X}_n} T_n(\mbf x) &\sim \frac{\psi(\mbf{\hat{x}})}{\tau(\Gamma)^{1/2}\det(-H|_{\mathbb{V}})^{1/2}}(2\pi n)^{r/2}c_ne^{n\phi(\mbf{\hat{x}})}\\
	&=\frac{k^{k^2}(2\pi n)^{(k-1)^2/2}(2\pi n)^{-(k^2-1)/2}}{k^{(k-1)}(\frac{k^2}{(k-1)^2})^{(k-1)^2/2}(((k-1)^2-2c)^2-4c^2k^2)^{(k-1)^2/4}}e^{2n\log(k(\frac{1}{2}(1-\frac{1}{k}))^c)}\\
	&= \left(\frac{(k-1)^4}{((k-1)^2-2c)^2-4c^2k^2}\right)^{(k-1)^2/4} k^k(2\pi n)^{-(k-1)}\left(k\left(\frac{1}{2}\left(1-\frac{1}{k}\right)\right)^c\right)^{2n}\\
	&\sim \left(\frac{(k-1)^4}{((k-1)^2-2c)^2-4c^2k^2}\right)^{(k-1)^2/4} (\ex Y)^2
\end{align*}
\end{proof}

\begin{proposition} \label{prop:gnm_ub}
For $k > 3$ an integer such that there exists a doubly regular tournament of order $k$ and $0 < c < \frac{1}{2}\ell_k$, the number of $k$-oriented colourings of $\vec{M} \sim \vec{\mathcal{M}}(n,m=cn))$ is positive with positive probability.
\end{proposition}

\begin{proof}
By the Paley-Zygmund inequality and Lemmas~\ref{lem:gnm_first_mom} and \ref{lem:gnm_second_mom},
\[ \pr[Y > 0] \ge \frac{(\ex Y)^2}{\ex(Y^2)} \sim \left(\frac{((k-1)^2-2c)^2-4c^2k^2}{(k-1)^4}\right)^{(k-1)^2/4} > 0. \]
\end{proof}

Note that we do not claim that the ratio $\frac{(\ex Y)^2}{\ex(Y^2)}$ tends to $1$, which would show that a.a.s.~$\vec{M} \sim \vec{\mathcal{M}}(n,m=cn)$ could be $k$-oriented-coloured. Indeed, we will see that $\vec{M}$ is not even colourable with probability tending to $1$. But this
ratio does not tend towards $1$ in many investigations of even the undirected chromatic number of random graphs. When considering the undirected question
$\chi(\mathcal{G}(n,m=cn))$, Achlioptas and Naor~\cite{AN05} circumvent this
issue using a sharp threshold result of Achlioptas and Friedgut~\cite{AF99}.
There is no reason to believe an equivalent result holds in the directed case.
As we will see in Section~\ref{ssec:vsparse_vdense}, there are values for $d$ for
which $\vec{\mathcal{G}}(n,p=\frac{d}{n})$ and $\vec{\mathcal{G}}(n,d)$ are
concentrated in more than one value. Instead, in Section~\ref{sec:window} we
adapt an argument of {\L}uczak~\cite{Lu91b}, using the fact that $k$ is an upper
bound with positive probability to find a weaker bound that holds with high
probability.

\subsection{$\vec{\mathcal{C}}(n,d)$}

In this section we use a second moment argument to find values of $k$ for which $\vec{C} \sim \vec{\mathcal{C}}(n,d)$ are $k$-oriented colourable with positive probability. Fix a doubly regular tournament on $k$ vertices $T_k$ and let $Y$ count the number of equitable doubly regular $T_k$-colourings of $\vec{\mathcal{C}}(n,d)$. 

We start with the first moment by adapting the following result of Kemkes et al.~\cite{KPW10}:
\begin{proposition}[{\cite[Proposition 2(a)]{KPW10}}]
Fix integers $d,k \ge 3$. Let $Z$ be the number of equitable $k$-colourings of a random $d$-regular multigraph $\mathcal{C}(n,d)$ (where $n$ is restricted to the set of multiples of $k$). Then
\[ \ex Z \sim k^{k/2} \left(\frac{k-1}{2\pi(k-2)}\right)^{(k-1)/2}n^{-(k-1)/2}k^n\left(1-\frac{1}{k}\right)^{dn/2}. \]
\end{proposition}

\begin{corollary}\label{cor:first_mom}
Fix integers $d,k \ge 3$ and a tournament $T_k$ on $k$ vertices (not necessarily doubly regular). Then
\[ \ex Y \sim k^{k/2} \left(\frac{k-1}{2\pi(k-2)}\right)^{(k-1)/2}n^{-(k-1)/2}k^n\left(\frac{1}{2}\left(1-\frac{1}{k}\right)\right)^{dn/2}. \]
\end{corollary}

\begin{proof}
Each proper (unoriented) $k$-colouring of $\mathcal{C}(n,d)$ has exactly one orientation that adheres to $T_k$, so there are the same number of proper $k$-oriented-colourings of $\vec{\mathcal{C}}(n,d)$ that are coloured by $T_k$ as there are unoriented colourings. However, each $C \in \mathcal{C}(n,d)$ has $2^{dn/2}$ orientations. Thus we have
\begin{align*}
|\vec{\mathcal{C}}(n,d)| Y &= |\{(\vec{C}, \chi_o) : \vec{C} \sim \vec{\mathcal{C}}(n,d) \text{ and $\chi_o$ properly $T_k$-colours $\vec{C}$}\}|\\
	&= |\{(C, \chi) : C \sim \mathcal{C}(n,d) \text{ and $\chi$ properly $k$-colours $C$}\}|\\
	&= |\mathcal{C}(n,d)| Z
\end{align*}
so
\[ \ex Y = \frac{|\mathcal{C}(n,d)|}{|\vec{\mathcal{C}}(n,d)|} \ex Z = \frac{1}{2^{dn/2}} \ex Z. \]
\end{proof}

In order to prove Proposition~\ref{prop:gnd_ub}, we also need an estimate for the second moment $\ex Y^2$. This will require a substantial amount more effort. First, in Section~\ref{ssec:counting} we give a counting argument similar to that in Section~\ref{sec:lb}. Then we use Proposition~\ref{prop:generalized_AN} to optimize the exponential part of resulting equation in Section~\ref{ssec:optimization}. Finally, we complete the second moment argument and prove Proposition~\ref{prop:gnd_ub} in Section~\ref{ssec:second_moment}.

\subsubsection{Counting Argument} \label{ssec:counting}

To calculate $\ex(Y^2)$, we count ordered triples $(\vec{C},h_1,h_2)$ such that $h_1$ and $h_2$ are balanced proper $T_k$-colourings of $\vec{C}\in \vec{\mathcal{C}}(n,d)$. Let $A = (a_{ij})_{\substack{1 \le i \le k \\ 1 \le j \le k}}$ be the overlap matrix of $h_1$ and $h_2$; that is, $a_{ij}$ is the proportion of vertices $x \in \vec{C}$ such that $h_1(x) = i$ and $h_2(x) = j$. Note that $a_{ij} \ge 0$ for each $i,j \in [k]$ and, as $h_1$ and $h_2$ are balanced,
\begin{equation} \label{eq:Acons}
\forall i \in [k] \quad \sum_{j=1}^k a_{ij} = \frac{1}{k} \qquad \text{and} \qquad  \forall j \in [k] \quad \sum_{i=1}^k a_{ij} = \frac{1}{k}.
\end{equation}
We introduce an (undirected) auxiliary graph $\Gamma$. The vertex set of $\Gamma$ is $[k]^2$, so there is $v \in V(\Gamma)$ corresponding to each $a_{ij}$, and we may, when convenient, refer to $a_{ij}$ as $a_v$. The edge set of $\Gamma$ is defined by
\[ E(\Gamma) = \bigl\{ \{u,v\} \mid u = (i,j), v = (i',j'), \ \text{and} \ i \to i', j \to j' \in A(T_k)\bigr\}. \]
Note that if we partition edges in $\vec{C}$ by the pairs of colours their vertices receive in $h_1$ and $h_2$, the set of equivalence classes is exactly the edges of $\Gamma$: in order for an edge to be directed from a vertex receiving colour $i$ in $h_1$ and colour $j$ in $h_2$ towards a vertex receiving colour $i'$ in $h_1$ and $j'$ in $h_2$, we require that $i \to i'$ and $j \to j'$ in $T_k$. Note also that as $T_k$ is doubly regular, $\Gamma$ is regular of degree $\frac{(k-1)^2}{2}$.

Then let $B = (b_{uv})_{u=(i,j), v = (i',j')}$ record the proportion of edges between a vertex $x \in \vec{C}$ such that $h_1(x) = i$ and $h_2(x) = j$ and a vertex $y \in \vec{C}$ such that $h_1(y) = i'$ and $h_2(y) = j'$. Then for any tournament $T_k$, $B$ satisfies
\begin{align}\label{eq:Bcons}
\forall u,v & \quad b_{uv} \ge 0 \notag\\
\forall u,v &\quad \text{$b_{uv} = 0$ unless $i \to i'$ and $j \to j'$ (or $i' \to i$ and $j' \to j$) in $T_k$} \notag\\
\forall u,v & \quad b_{uv} = b_{vu} \notag\\
\forall v = (i,j) & \quad \sum\limits_{u \sim v} b_{uv} = a_{ij}.
\end{align}

To choose a $\vec{C} \in \vec{\mathcal{C}}(n,d)$ that respects $h_1$ and $h_2$, first select vertices to receive each pair of colours, which we do in one of 
\[ \frac{n!}{\prod_{i,j} (na_{ij})!} \]
ways. Then, for each pair of colour classes, select the vertices that will be connected by an edge, which we do in
\[ \prod_{i,j} \frac{(dna_{ij})!}{\prod_{u\sim v=(i,j)} (dnb_{uv})!} \]
many ways. Finally, match the chosen vertices of each edge type in one of
\[ \prod_{uv \in E(\Gamma)} (dnb_{uv})! \]
ways. Noting that
\[ \frac{\prod_{uv \in E(\Gamma)} (dnb_{uv})!}{\prod_{v \in V(\Gamma)} \prod_{u \sim v} (dnb_{uv})!} = \frac{1}{\prod_{uv \in E(\Gamma)} (dnb_{uv})!} ,\]
repeating this process for each choice of $A$ and $B$ that satisfy (\ref{eq:Acons}) and (\ref{eq:Bcons}), and dividing by $|\vec{\mathcal{C}}(n,d)|$ gives
\begin{align}
\ex(Y^2) &= \frac{1}{2^{dn/2}(dn-1)!!} \sum_{A \text{ s.t. } \eqref{eq:Acons}} \frac{n!}{\prod_{ij} (na_{ij})!} \sum_{B \text{ s.t. } \eqref{eq:Bcons}} \prod_{i,j} \frac{(dna_{ij})!}{\prod_{u\sim v=(i,j)} (dnb_{uv})!} \prod_{uv \in E} (dnb_{uv})! \notag\\
	&= \sum_{A \text{ s.t.} \eqref{eq:Acons}, \ B \text{ s.t.} \eqref{eq:Bcons}} \frac{n!}{\prod_{ij} (na_{ij})!} \cdot \frac{\prod_{ij} (dna_{ij})!}{(dn)!} \cdot \frac{(\frac{dn}{2})!}{\prod\limits_{uv \in E(\Gamma)} (dnb_{uv})!} \notag\\
	&= \sum_{A \text{ s.t.} \eqref{eq:Acons}, \ B \text{ s.t.} \eqref{eq:Bcons}} \frac{\xi(n)(\frac{n}{e})^n}{\prod_{ij} \xi(na_{ij})(\frac{na_{ij}}{e})^{na_{ij}}} \cdot \frac{\prod_{ij} \xi(dna_{ij})(\frac{dna_{ij}}{e})^{dna_{ij}}}{\xi(dn)(\frac{dn}{e})^{dn}} \cdot \frac{\xi(\frac{dn}{2})(\frac{dn}{2e})^{\frac{dn}{2}}}{\prod\limits_{uv \in E(\Gamma)} \xi(dnb_{uv})(\frac{dnb_{uv}}{e})^{dnb_{uv}}} \notag\\
	&= \sum_{A \text{ s.t.} \eqref{eq:Acons}, \ B \text{ s.t.} \eqref{eq:Bcons}} p(A,B,n) e^{nf(A,B)} \label{eq:second_moment}
\end{align}
where
\begin{equation} \label{eq:p_def}
p(A,B,n) = \frac{\xi(n)\xi(\frac{dn}{2})\prod\limits_{v \in
V(\Gamma)}\xi(dna_v)}{\xi(dn)\prod\limits_{v \in V(\Gamma)}\xi(na_v) \prod\limits_{uv
\in E(\Gamma)}\xi(dnb_{uv})}
\end{equation}
and
\begin{equation}\label{eq:f1}
f(A,B) = (d-1)\sum_{v \in V(\Gamma)} a_v \log a_v - d\sum_{uv \in E(\Gamma)} b_{uv} \log
(2b_{uv}).
\end{equation}

\subsubsection{Optimization} \label{ssec:optimization}

In order to estimate $\ex Y^2$, we will use a Laplace summation technique. As we are summing over a polynomial number of exponential terms, all but $o(1)$ of the weight of the sum comes from terms with maximum exponential contribution. In this section we use Proposition~\ref{prop:generalized_AN} to determine that maximum value.

\begin{lemma} \label{lem:second_mom_opt}
Fix an integer $k$ such that there exists a doubly regular tournament $T_k$ of
order $k$ and let $d < \ell_k$. Let $A = (a_{ij})_{i,j=1}^k$ be a $k \times k$
real positive matrix satisfying Equation~\eqref{eq:Acons} and let $B = (b_{uv})_{u=(i,j),
v = (i',j')}$ be a $k^2 \times k^2$ real positive matrix satisfying
Equation~\eqref{eq:Bcons}. Define $f(A,B)$ as in Equation~\eqref{eq:f1}. Then $f(A,B)$ is uniquely
maximized at $f(\hat{A},\hat{B}) = \log k^2 + d\log(\frac{1}{2}(1-\frac{1}{k}))$
where $\hat{A} = \frac{1}{k^2}J_{k,k}$ and $\hat{B} = (\hat{b}_{uv})_{u=(i,j), v
= (i',j')}$ where
\[ \hat{b}_{uv} = \begin{cases} \frac{2}{k^2(k-1)^2} & u\sim v \\ 0 & u \not\sim v \end{cases}. \]
\end{lemma}

\begin{proof}
We start by finding a more convenient expression of $f$. Note that using Equation~\eqref{eq:Bcons},
\begin{align}
(d-1) \hspace{-4pt}\sum_{v \in V(\Gamma)} \hspace{-3pt} a_v\log a_v - d \hspace{-6pt} \sum_{uv \in E(\Gamma)} &\hspace{-3pt} b_{uv}\log(2b_{uv}) = - \hspace{-5pt} \sum_{v \in V(\Gamma)} \hspace{-2pt} a_v \log a_v +d\left(\sum_{v \in V(\Gamma)} \hspace{-3pt} a_v \log a_v + \hspace{-5pt} \sum_{uv \in E(\Gamma)} \hspace{-4pt} b_{uv} \log\left(\frac{1}{2b_{uv}}\right) \right) \notag\\
	&= -\sum_{v \in V(\Gamma)} a_v \log a_v +d\left(\sum_{v \in V(\Gamma)} \left(\sum_{u \sim v} b_{uv}\right) \log a_v + \hspace{-3pt} \sum_{uv \in E(\Gamma)} \hspace{-3pt} b_{uv} \log\left(\frac{1}{2b_{uv}}\right) \right) \notag\\
	&= -\sum_{v \in V(\Gamma)} a_v \log a_v +d\sum_{uv \in E(\Gamma)} b_{uv} \left(\log a_u + \log a_v + \log\left(\frac{1}{2b_{uv}}\right) \right) \notag\\
	&= -\sum_{v \in V(\Gamma)} a_v \log a_v +d\sum_{uv \in E(\Gamma)} b_{uv} \log\left(\frac{a_ua_v}{2b_{uv}}\right) \label{eq:f2}
\end{align}
Using an idea from~\cite{NP21}, we now maximize $f$ maintaining the constraints from Equation~\eqref{eq:Acons} but relaxing those in Equation~\eqref{eq:Bcons} to $b_{uv} \ge 0$ and
\begin{equation} \label{eq:Brelaxedcons}
\sum_{u, v \in V} b_{uv} = 1.
\end{equation}
In view of this relaxation, we can regard $B$ as an arbitrary probability distribution. We define another probability distribution given by
\[
b^*_{uv} = \begin{cases} \frac{a_{u}a_{v}}{2\sum_{xy \in E(\Gamma)} a_{x}a_{y}} & u\sim v, \\ 0 & u \not\sim v\end{cases},
\]
and we write $B^*=(b^*_{uv})_{u,v \in V(\Gamma)}$. Then
\[
\sum_{uv \in E(\Gamma)} b_{uv}\log\left( \frac{a_{u}a_{v}}{2b_{uv}}\right) = \frac{1}{2}\log \left(\sum_{xy \in E(\Gamma)} a_xa_y\right) - D_{KL}(B \| B^*),
\]
where $D_{KL}(B \| B^*) = \sum_{uv \in E(\Gamma)} b_{uv}\log\left( \frac{b_{uv}}{b^*_{uv}}\right)$ is the Kullback-Leibler divergence from $B$ to $B^*$. By Gibb's inequality, $D_{KL}(B \| B^*)\ge0$ with equality if and only if $B = B^*$. As a result,
\begin{equation}\label{eq:maxgamma}
\max_{B\text{ s.t.~\eqref{eq:Brelaxedcons}}} f(A, B) = -\sum_{v\in V(\Gamma)} a_v \log a_v + \frac{d}{2}\log\left(\sum_{xy \in E(\Gamma)} a_xa_y\right) =: \tilde{f}_d(A),
\end{equation}
with one unique maximizer at $B = B^*$. Note that if $A=\hat A$ then $B^*=\hat B$.

Next we seek to apply Corollary~\ref{cor:opt_doubly_reg_tour}, but our matrix $A$ is not currently doubly stochastic. We reparameterize, setting $\mbf \alpha = kA$. Then
\[ \tilde{f}_d(A) = \tilde{f}_d(\tfrac{1}{k}\mbf \alpha) =  \log k-\frac{1}{k} \sum_{v \in V(\Gamma)} \alpha_v \log \alpha_v + \frac{d}{2} \log\left(\frac{2\sum_{xy \in E(\Gamma)} \alpha_x\alpha_y}{k^2}\right) - \frac{d}{2}\log2. \]
Then as $\log k-\frac{d}{2}\log2$ is constant with respect to $\mbf \alpha$, Corollary~\ref{cor:opt_doubly_reg_tour} gives
\begin{equation} \label{eq:gamma_ub}
\tilde{f}_d(A) \le \log k +\log k + \frac{d}{2}\log\left(\frac{2|E|}{k^4}\right) -\frac{d}{2}\log2 = \log k^2 + d\log\left(\frac{1}{2}\left(1-\frac{1}{k}\right)\right)
\end{equation}
whenever $d \le \frac{2 (k-1)^3}{k(k+1)(k-2)}\log(k-1)$. Now, using Equations \eqref{eq:maxgamma} and \eqref{eq:gamma_ub}, we have
\begin{align*}
f(A,B) \le \max_{A \text{ s.t.} \eqref{eq:Acons}, \ B \text{ s.t.} \eqref{eq:Bcons}} f(A,B) &\le \max_{A \text{ s.t.} \eqref{eq:Acons}, \ B \text{ s.t.} \eqref{eq:Brelaxedcons}} f(A,B)\\
	&\le \max_{A \text{ s.t.} \eqref{eq:Acons}} \tilde{f}_d(A) = \log k^2 + d\log\left(\frac{1}{2}\left(1-\frac{1}{k}\right)\right) = f(\hat{A},\hat{B})
\end{align*}

\end{proof}

\subsubsection{Second Moment Argument} \label{ssec:second_moment}

We now return to Equation~\eqref{eq:second_moment}. In order to estimate the
summation, we again use a version of Laplace summation. This time, however, the
constraints in Equations \eqref{eq:Acons} and \eqref{eq:Bcons} cannot be expressed as the signed incidence matrix of an auxiliary graph. Instead, we use the following theorem proved by Greenhill et al. in \cite{GJR10}. 

\begin{theorem}[{\cite[Theorem 2.3]{GJR10}}]\label{thm:laplacian_summation}
Suppose the following:
\begin{enumerate}[(i)]
\item $\Lambda \subset \mathbb{R}^N$ is a lattice with rank $1\le r \le N$.
\item $\mathbb V \subseteq \mathbb{R}^N$ is the $r$-dimensional subspace spanned by $\Lambda$.
\item $\mathbb W = \mathbb V + \mbf w$ is an affine subspace parallel to $\mathbb V$, for some $\mbf w \in \mathbb{R}^N$.
\item $K \subset \mathbb{R}^N$ is a compact convex set with non-empty interior $K^\circ$.
\item $\phi : K \to \mathbb{R}$ is a continuous function and the restriction of $\phi$ to $K \cap \mathbb W$ has a unique maximum at some point $\mbf z_0 \in K^\circ \cap \mathbb W$.
\item $\phi$ is twice continuously differentiable in a neighbourhood of $\mbf z_0$ and $H := D^2\phi(\mbf z_0)$ is its Hessian at $\mbf z_0$.
\item $\psi : K_1 \to \mathbb{R}$ is a continuous function on some neighbourhood $K_1 \subseteq K$ of $\mbf z_0$ with $\psi(\mbf z_0) > 0$.
\item For each positive integer $n$ there is a vector $\ell_n \in \mathbb{R}^N$ with $\ell_n/n \in \mathbb W$,
\item For each positive integer $n$ there is a positive real number $b_n$ and a function $a_n : (\Lambda + \ell_n) \cap nK \to \mathbb{R}$ such that, as $n \to \infty$,
\[ a_n(\ell) = O(b_n e^{n\phi(\ell/n)+o(n)}),\quad \ell \in (\Lambda +\ell_n) \cap nK \]
and
\[ a_n(\ell) = b_n(\psi(\ell/n) + o(1)) e^{n\phi(\ell/n)}, \quad \ell \in (\Lambda +\ell_n) \cap nK_1,\]
uniformly for $\ell$ in the indicated sets.
\end{enumerate}
Then provided $\det(-H|_{\mathbb V} )\ne 0$, as $n \to \infty$,
\[ \sum_{\ell\in(\Lambda+\ell_n)\cap nK}a_n(\ell) \sim \frac{(2\pi)^{r/2}\psi(\mbf z_0)}{\det(\Lambda)\det(-H|_{\mathbb V})^{1/2}}b_nn^{r/2}e^{n\phi(\mbf z_0)}. \]
\end{theorem}

We also reference this additional results from that paper:

\begin{lemma}[{\cite[Lemma 6.2]{GJR10}}] \label{lem:det_lattice}
Let $0 \le m \le N$. Let $x_1, \ldots, x_m$ be linearly independent vectors in $\mathbb{Z}^N$. Let $V$ be the subspace of $\mathbb{R}^N$ spanned by $x_1, \ldots, x_m$ and let $V^{\perp}$ be its orthogonal complement; thus
\[ V^{\perp} = \{y \in \mathbb{R}^N : \langle y,x_i \rangle = 0 \ \text{for} \ i=1,\ldots,m\}. \]
Let $\mathcal{L}$ and $\mathcal{L}^{\perp}$ be the lattices $V \cap \mathbb{Z}^N$ and $V^{\perp} \cap \mathbb{Z}^N$, and let $\mathcal{L}_0$ be the lattice spanned by $x_1, \ldots, x_m$ (i.e., the set $\{\sum_{i=1}^m n_ix_i : n_ \in \mathbb{Z}\}$ of integer combinations). Then $\mathcal{L}^{\perp}$ has rank $N - m$ and
\[ \det(\mathcal{L}^{\perp}) = \det(\mathcal{L}) = \det(\mathcal{L}_0)/q, \]
where $q$ is the order of the finite group $\mathcal{L}/\mathcal{L}_0$. Explicitly, $q$ is the number of solutions $(t_1, \ldots, t_m)$ in $(\mathbb{R}/\mathbb{Z})^m$ (or $(\mathbb{Q}/\mathbb{Z})^m$) of the system
\[ \sum_i x_{ij}t_i \equiv 0 (\text{mod }1), \quad j=1, \ldots, N, \]
where $x_i = (x_{ij})_{j=1}^N$ for $i = 1,\ldots, m$.
\end{lemma}

To each restriction in Equation~\eqref{eq:Acons} of the form $\sum_{j = 1}^k a_{ij} = \frac{1}{k}$ we associate the vector $r_i \in \mathbb{R}^{k^2+\frac{1}{4}k^2(k-1)^2}$ with value $1$ at index $j+ki$ if the variable $a_{ij}$ appears in the constraint and $0$ at all other indices. To each restriction in Equation~\eqref{eq:Acons} of the form $\sum_{i=1}^k a_{ij} = \frac{1}{k}$ we associate the vector $c_j \in \mathbb{R}^{k^2+\frac{1}{4}k^2(k-1)^2}$ in the same manner. Finally, to each restriction in Equation~\eqref{eq:Bcons} of the form $\sum_{u \sim v} b_{uv} = a_{ij}$ we associate a vector $e_{ij} \in \mathbb{R}^{k^2+\frac{1}{4}k^2(k-1)^2}$ with entry $-1$ at index $j+ki$ and, given some fixed ordering of the edges, index $1$ at position $k^2+\ell$ if the $\ell$th edge is $\{u,v\}$ for some $b_{uv}$ in the constraint.

Define the $(2k+k^2) \times (k^2+\frac{1}{4}k^2(k-1)^2)$ matrix $D$ by setting the first $k$ rows equal to $r_i$, the next $k$ rows equal to $c_j$, and the last $k^2$ rows equal to $e_{ij}$. Let $\mbf y \in \mathbb{R}^{2k+k^2}$ be the
vector with first $2k$ entries $\frac{1}{k}$ and last $k^2$ entries $0$.

\begin{lemma}\label{lem:cons_space}
If $\mbf x \in \mathbb{R}^{k^2+\frac{1}{4}k^2(k-1)^2}$ satisfies $D\mbf x = \mbf
y$ then assigning (as vectors) $A = (x_i)_{i=1}^{k^2}$ and $B =
(x_i)_{i=k^2+1}^{k^2+\frac{1}{4}k^2(k-1)^2}$ gives an assignment to $A$ and $B$ that
satisfies Equations \eqref{eq:Acons} and \eqref{eq:Bcons}, respectively, with the exception that entries may not be positive.
\end{lemma}

\begin{proof}
By construction,
\[ r_i \cdot \mbf x = \sum_{j=1}^k a_{ij} \quad \text{and} \quad c_j \cdot \mbf x = \sum_{i=1}^k a_{ij} \]
so as the first $2k$ entries of $\mbf y$ are $\frac{1}{k}$, $D\mbf x = \mbf y$ enforces the equality constraints in Equation~\eqref{eq:Acons}. For the latter $k^2$ rows,
\[ e_{ij} \cdot \mbf x = -a_{ij} + \sum_{u \sim v} b_{uv} \]
and thus as $\mbf y$ has $0$s in the corresponding rows, $D\mbf x = \mbf y$ also enforces the last constraints in Equation~\eqref{eq:Bcons}. Note that each edge appears only as a single variable in $D$, and thus we may set $b_{uv} = b_{vu}$ to satisfy the third constraint, and that nonedges do not appear in the constraints, so we may set them to $0$ to satisfy the second constraint.
\end{proof}

\begin{lemma}\label{lem:rankD}
The collection $ \{r_i\}_{i=2}^k \cup \{c_j\}_{j=1}^{k} \cup \{e_{ij}\}_{i,j=1}^k$ is linearly independent, and furthermore
\[ \rank(D) = k^2+2k-1. \]
\end{lemma}

\begin{proof}
As each vector contains a nonzero entry in a unique column, the collection
$\{e_{ij}\}$ is linearly independent. We claim that the collection $R =
\{r_i\}_{i=2}^k \cup \{c_j\}_{j=1}^{k}$ is linearly independent. Consider any
linear combination of $R$ that sums to $\vec{0}$. Note that $a_{1j}$ occurs in
$c_j$ but does not occur in any $r_i$ as $r_1 \notin R$. Thus each $c_j$ must
have a coefficient of $0$. But each $r_i$ has unique entries (there is no
$a_{ij}$ that occurs in two different row constraints because the value of $i$
would differ) so these must also have coefficients $0$. Now $R \cup \{e_{ij}\}$
is also linearly independent: given any linear combination that sums to $0$, the
coefficients of the $\{e_{ij}\}$ must be $0$ because $\{e_{ij}\}$ is linearly
independent and $R$ does not contain any entries past index $k^2$ while each
$\{e_{ij}\}$ does, and then because $R$ is linearly independent its coefficients
must also be zero. Finally, as each $a_{ij}$ occurs in exactly $r_i$ and $c_j$,
\[ r_1 = \sum_{j=1}^k c_j - \sum_{i=2}^{k} r_i \]
so no larger collection of rows of $D$ is linearly independent and $\rank(D) = k^2+2k-1$.
\end{proof}

\begin{lemma} \label{lem:det_l_nought}
Let $\hat{D}$ be the matrix that results from removing the row corresponding to $r_1$ from $D$. Then
\[ \det(\hat{D}^T\hat{D}) = \frac{(k-1)^{2k}(k(k-2))^{2k-2}((k^2-k+4)(k^2-3k+4))^{\frac{1}{2}(k-1)^2}}{2^{k^2-1}}.\]
\end{lemma}

\begin{proof}
We begin by considering the structure of $\hat{D}$. We claim that $\hat{D}$ is a block matrix
\[ \hat{D} = \begin{bmatrix} U_k & 0_{2k-1,|E|} \\ -I_{k^2} & E_\Gamma \end{bmatrix} \]
where $E_\Gamma$ is the $k^2 \times \frac{1}{4}k^2(k-1)^2$ incidence matrix of our auxiliary graph $\Gamma$ and $U_k$ is itself a block matrix, defined by
\[ U_k = \begin{bmatrix} 0_{k-1,k} & \mathbf{1}_k \otimes I_{k-1} \\ I_k & I_k \otimes \mathbf{1}_{k-1} \end{bmatrix}. \]
The $0_{2k-1,|E|}$ block occurs in $\hat{D}$ because the constraints in
Equation~\eqref{eq:Acons} only contain entries in the first $k^2$ columns. The
negative identity block comes from the $-a_{ij}$ terms that arise when
rearranging the constraints in Equation~\eqref{eq:Bcons}. The incidence matrix
of $\Gamma$ arises because the latter $\frac{1}{4}k^2(k-1)^2 = |E(\Gamma)|$
columns of $\hat{D}$ contain exactly two non-zero entries corresponding to the
two constraints in Equation~\eqref{eq:Bcons} in which the corresponding $b_{uv}$
appears. To understand the structure of $U_k$, note that each $r_i$ vector
contains $k$ consecutive 1 entries while each $c_j$ vector contains $k$ entries
with value 1, each separated by $k-1$ consecutive 0 entries, with the first 1 entry occurring at
index $j$. As we removed $r_1$ from $D$ to get $\hat{D}$, the first block of
$\hat{D}$ contains no 1 entries.

Now set $G = \hat{D}^T\hat{D}$ and note that $G$ is also a block matrix with blocks
\[  G = \begin{bmatrix} U_kU_k^T & -U_k \\ -U_k^T & I_{k^2} + E_\Gamma E_\Gamma^T\end{bmatrix} \]
It is a fact of spectral graph theory~\cite[Lemma 8.2.3]{GR01} that $E_\Gamma E_\Gamma^T = A_{\Gamma} + D_\Gamma$ where $A_\Gamma$ is the adjacency matrix of $\Gamma$ and $D_\Gamma$ is the degree matrix, a diagonal matrix with $d_{ii} = d_\Gamma(i)$ and zero entries elsewhere. Straightforward calculation gives
\[ U_kU_k^T = \begin{bmatrix} kI_{k-1} & J_{k-1,k} \\ J_{k,k-1} & kI_k \end{bmatrix}. \]
To determine $\det(G)$, we use a block matrix determinant formula using the Schur complement (see~\cite[Section 0.8.5]{HJ13}):
\[ \det\left(\begin{bmatrix} A & B \\ C & D \end{bmatrix}\right) = \det(A)\det(D-CA^{-1}B). \]
This requires several calculations. First,
\[ \det(U_kU_k^T) = \det(kI_{k-1})\det(kI_{k}-J_{k,k-1}(kI_{k-1})^{-1}J_{k-1,k}) = k^{k-1}\det(kI_{k} - J_{k}). \]
The eigenvalues of $-J_{k}$ are $-k$ with multiplicity 1 and $0$ with multiplicity $k-1$. Shifting these by $k$ gives $\det(kI_{k}-J_{k,k}) = k^{k-1}$ and therefore
\begin{equation}\label{eq:detF}
\det(U_kU_k^T) = k^{2k-2}.
\end{equation}
One can verify that
\[ (U_kU_k^T)^{-1} = \frac{1}{k^2} \begin{bmatrix} kJ_{k-1} + I_{k-1} & -kJ_{k-1,k} \\ -kJ_{k,k-1} & (k-1)J_k + kI_k \end{bmatrix} \]
and then, using standard results on tensor products and an iron will, one can calculate
\[ (-U_k^T)(U_kU_k^T)^{-1}(-U_k) = \frac{1}{k^2}(kI_k \otimes J_{k} + kJ_{k} \otimes I_k - J_{k^2}). \]
Thus
\[ I_{k^2}+E_\Gamma E_\Gamma^T - (-U_k^T)(U_kU_k^T)^{-1}(-U_k) = I_{k^2} + A_\Gamma + D_{\Gamma} - \frac{1}{k^2}(kI_k \otimes J_{k} + kJ_{k} \otimes I_k - J_{k^2}). \]
We now investigate the eigenvalues and eigenspaces of the matrices $A_\Gamma, D_{\Gamma}$, and $kI_k \otimes J_{k} + kJ_{k} \otimes I_k - J_{k^2}$.

As $\Gamma$ is regular, $D_{\Gamma} = \frac{(k-1)^2}{2}I_{k^2}$.

By Proposition~\ref{prop:tourn_prod_eigenvectors}, the eigenvalues and eigenspaces of $A_\Gamma$ are
\[ \begin{array}{c|c}
\text{e-value} & \text{e-space}\\
\hline
\frac{(k-1)^2}{2} & \mathbf{1}_k \otimes \mathbf{1}_k = \mathbf{1}_{k^2}\\
-\frac{(k-1)}{2} & \Span(\{\mathbf{u} \otimes \mathbf{1} \mid \mathbf{u} \in \mathbf{1}_k^\perp\})\\
-\frac{(k-1)}{2} & \Span(\{\mathbf{1} \otimes \mathbf{u} \mid \mathbf{u} \in \mathbf{1}_k^\perp\})\\
-\frac{(k-1)}{2} & \mathbb{S}_1\\
-\frac{(k-1)}{2} & \mathbb{S}_2\\
\frac{(k+1)}{2} & \mathbb{S}_3\\
\end{array}\]
where the $\mathbb{S}_i$ are subspaces of $\Span(\{\mathbf{u} \otimes \mathbf{v} \mid \mathbf{u}, \mathbf{v} \in \mathbf{1}_k^\perp\})$ defined carefully in Proposition~\ref{prop:tourn_prod_eigenvectors}.

Finally, for any constant $\alpha$, $J_{\alpha}$ has two eigenvalues, $\alpha$ with multiplicity one corresponding to the eigenspace spanned by $\mathbf{1}_\alpha$ and $0$ with multiplicity $\alpha-1$ corresponding to the eigenspace $\mathbf{1}_k^\perp$. Therefore the eigenvalues and eigenspaces of $kI_k \otimes J_{k}$ are
\[ \begin{array}{c|c}
\text{e-value} & \text{e-space}\\
\hline
k^2 & \mathbf{1}_k \otimes \mathbf{1}_k = \mathbf{1}_{k^2}\\
k^2 & \Span(\{\mathbf{u} \otimes \mathbf{1} \mid \mathbf{u} \in \mathbf{1}_k^\perp\})\\
0 & \Span(\{\mathbf{1} \otimes \mathbf{u} \mid \mathbf{u} \in \mathbf{1}_k^\perp\})\\
0 & \Span(\{\mathbf{u} \otimes \mathbf{v} \mid \mathbf{u}, \mathbf{v} \in \mathbf{1}_k^\perp\})
\end{array}\]
while the eigenvalues of $kJ_{k} \otimes I_k$ are
\[ \begin{array}{c|c}
\text{e-value} & \text{e-space}\\
\hline
k^2 & \mathbf{1}_k \otimes \mathbf{1}_k = \mathbf{1}_{k^2}\\
0 & \Span(\{\mathbf{u} \otimes \mathbf{1} \mid \mathbf{u} \in \mathbf{1}_k^\perp\})\\
k^2 & \Span(\{\mathbf{1} \otimes \mathbf{u} \mid \mathbf{u} \in \mathbf{1}_k^\perp\})\\
0 & \Span(\{\mathbf{u} \otimes \mathbf{v} \mid \mathbf{u}, \mathbf{v} \in \mathbf{1}_k^\perp\})
\end{array}\]
Therefore the eigenvalues and eigenspaces of $-\frac{1}{k^2}(kI_k \otimes J_{k} + kJ_{k} \otimes I_k - J_{k^2})$ are
\[ \begin{array}{c|c}
\text{e-value} & \text{e-space}\\
\hline
-1 & \mathbf{1}_k \otimes \mathbf{1}_k = \mathbf{1}_{k^2}\\
-1 & \Span(\{\mathbf{u} \otimes \mathbf{1} \mid \mathbf{u} \in \mathbf{1}_k^\perp\})\\
-1 & \Span(\{\mathbf{1} \otimes \mathbf{u} \mid \mathbf{u} \in \mathbf{1}_k^\perp\})\\
0 & \Span(\{\mathbf{u} \otimes \mathbf{v} \mid \mathbf{u}, \mathbf{v} \in \mathbf{1}_k^\perp\})
\end{array}\]

Fortunately, the eigenspaces of $A_\Gamma$ and $kI_k \otimes J_{k} + kJ_{k} \otimes I_k - J_{k^2}$ align, so, shifting by $1+\frac{(k-1)^2}{2}$ to account for $I_{k^2}+D_\Gamma = (1+\frac{(k-1)^2}{2})I_{k^2}$, we see that the eigenvalues of $I_{k^2} + A_\Gamma + D_{\Gamma} - \tfrac{1}{k^2}(kI_k \otimes J_{k} + kJ_{k} \otimes I_k - J_{k^2})$ are

\[ \begin{array}{c|c}
\text{e-value} & \text{multiplicity}\\
\hline
1+\frac{(k-1)^2}{2}+\frac{(k-1)^2}{2}-1 = (k-1)^2& 1\\
1+\frac{(k-1)^2}{2}-\frac{(k-1)}{2}-1 = \frac{1}{2}(k-1)(k-2)& 2(k-1)\\
1+\frac{(k-1)^2}{2}-\frac{(k-1)}{2} = \frac{1}{2}(k^2-3k+4)& \frac{(k-1)^2}{2}\\
1+\frac{(k-1)^2}{2}+\frac{(k+1)}{2} = \frac{1}{2}(k^2-k+4)& \frac{(k-1)^2}{2}
\end{array}\]
and therefore
\begin{align*}
\det\bigl(I_{k^2} + A_\Gamma + D_{\Gamma} - &\tfrac{1}{k^2}(kI_k \otimes J_{k} + kJ_{k} \otimes I_k - J_{k^2})\bigr)\\
	&= \frac{(k-1)^{2k}(k-2)^{2(k-1)}\bigl((k^2-3k+4)(k^2-k+4)\bigr)^{\frac{1}{2}(k-1)^2}}{2^{k^2-1}}
\end{align*}
which, along with Equation~\eqref{eq:detF}, gives
\[ \det(G) = \frac{(k-1)^{2k}(k(k-2))^{2k-2}((k^2-k+4)(k^2-3k+4))^{\frac{1}{2}(k-1)^2}}{2^{k^2-1}}. \]
\end{proof}

\begin{proposition} \label{prop:PZ_for_C}
There is a value $C_{d,k}$ depending on $d$ and $k$ such that
\[ \ex Y^2 \sim C_{d,k} (\ex Y)^2. \]
\end{proposition}

\begin{proof}
We use Theorem~\ref{thm:laplacian_summation}. Let $\Lambda = \ker(D) \cap
\mathbb{Z}^{k^2+\frac{1}{4}k^2(k-1)^2}$ and $\mathbb{V} = \ker(D)$. Then
$\Lambda$ and $\mathbb{V}$ satisfy (i) and (ii).

Let $\mbf w \in \mathbb{R}^{k^2+\frac{1}{4}k^2(k-1)^2}$ be defined by $w_i =
\frac{1}{k^2}$ for $1 \le i \le k^2$ and $w_i = \frac{2}{k^2(k-1)^2}$ for $k^2+1
\le i \le k^2+\frac{1}{4}k^2(k-1)^2$ and define $\mathbb{W} = \mathbb{V} + \mbf
w$, satisfying (iii). Note that for any $\mbf x \in \mathbb{W}$,
\[ D\mbf x = D\mbf x + D(\mbf w - \mbf x) = D \mbf w = \mbf y \]
as $\mbf w - \mbf x \in \ker(D)$.

Define
\[ K = \{ \mbf x \in \mathbb{R}^{k^2+\frac{1}{4}k^2(k-1)^2} : 0 \le x_i \le
\frac{1}{k} \} \]
which satisfies (iv).

Let $\phi = f$. By Lemma~\ref{lem:cons_space}, the restriction of $\phi$ to $K
\cap \mathbb{W}$ ensures $A$ satisfies Equation~\eqref{eq:Acons} and $B$ satisfies
Equation~\eqref{eq:Bcons} and thus Lemma~\ref{lem:second_mom_opt} assures $\phi$ has a unique maximum at $z_0 =
(\hat{A},\hat{B}) \in K^\circ \cap \mathbb{W}$.

Using Equation~\eqref{eq:f1} it is easy to see that $\phi$ is twice differentiable and
furthermore that
\begin{align*}
\frac{d^2\phi}{da_uda_v} &= \begin{cases} \frac{d-1}{a_u} & a_u = a_v \\ 0 & a_u
\ne a_v \end{cases},\\
\frac{d^2\phi}{da_udb_{xy}} &= 0, \qquad \text{and}\\
\frac{d^2\phi}{db_{uv}db_{xy}} &= \begin{cases} -\frac{d}{b_{uv}} & b_{uv} =
b_{xy} \\ 0 & b_{uv} \ne b_{xy} \end{cases}.
\end{align*}
and therefore $H$ is a diagonal matrix with $k^2$ entries $(d-1)k^2$ and
$\frac{1}{4}k^2(k-1)^2$ entries $-\frac{1}{2}dk^2(k-1)^2$. We can write
\[ H = -\tfrac{1}{2}dk^2(k-1)^2I_{k^2+\frac{1}{4}k^2(k-1)^2} +
\left(\bigl((d-1)k^2+\tfrac{1}{2}dk^2(k-1)^2\bigr)I_{k^2} \otimes
E_{\frac{1}{4}(k-1)^2}\right) \]
where $E_{\frac{1}{4}(k-1)^2}$ is a matrix with $e_{11} = 1$ and zeroes
elsewhere. Thus $H$ satisfies (vi).

Define $K_1 \subseteq K$ by those $\mbf x \in
\mathbb{R}^{k^2+\frac{1}{4}k^2(k-1)^2}$ such that
\[ \frac{0.9}{k^2} \le x_i \le \frac{1.1}{k^2}, 1 \le i \le k^2 \]
and
\[ \frac{1.9}{k^2(k-1)^2} \le x_i \le \frac{2.1}{k^2(k-1)^2}, k^2+1 \le i \le
k^2 + \frac{1}{4}k^2(k-1)^2. \]
Let $\psi = \prod_{uv \in E} \frac{1}{\sqrt{b_{uv}}}$. Then $K_1$ and $\psi$
satisfy (vii).

For each $n$, define $\ell_n \in \mathbb{R}^{k^2+\frac{1}{4}k^2(k-1)^2}$ to have
first $k^2$ entries $\frac{n}{k^2}$ and remaining entries
$\frac{2n}{k^2(k-1)^2}$, satisfying (viii).

Finally, set
\[ b_n = \frac{1}{\sqrt{2}} \cdot (\sqrt{2\pi n})^{1-\frac{k^2(k-1)^2}{4}}
(\sqrt{d})^{k^2-\frac{k^2(k-1)^2}{4}} \]
and $a_n(\ell) = p(A,B,n)e^{nf(A,B)}$ where $A$ is the first $k^2$ elements of
$\ell/n$ and $B$ is the remaining elements of $\ell/n$. Then as $p(A,B,n) =
O(n^{\frac{1}{2}(k^2+2)})$, for any $\ell \in (\Lambda + \ell_n) \cap nK$,
\[ a_n(\ell) = p(A,B,n)e^{nf(A,B)} =
b_n\left(\frac{p(A,B,n)}{b_n}\right)e^{n\phi(\ell/n)} =
O(b_ne^{n\phi(\ell/n)+o(n)}) \]
and for any $\ell \in (\Lambda + \ell_n) \cap nK_1$,
\[ a_n(\ell) = p(A,B,n)e^{nf(A,B)} = b_n(\psi(\ell/n)+o(1))e^{n\phi(\ell/n)}, \]
both uniformly.

We need to determine $\rank(\Lambda)$ and $\det(\Lambda)$. We do so using Lemma~\ref{lem:det_lattice}. As we proved in Lemma~\ref{lem:rankD}, the $k^2+2k-1$ rows of $D$ are linearly independent. Our $\mathbb{V} = V^{\perp}$ in Lemma~\ref{lem:det_lattice} and thus
\[ \rank(\Lambda) = \rank(\mathcal{L}^\perp) = N-m = k^2+\frac{1}{4}k^2(k-1)^2 - (k^2+2k-1) = \frac{1}{4}k^2(k-1)^2-2k+1. \]
We claim the order of $\mathcal{L}/\mathcal{L}_0$ is $2$. Consider a solution $(t_1, \ldots, t_m)$. It will be
convenient to rename these vectors $t_{r_2}, \ldots, t_{r_k}$ corresponding to
the rows $r_2, \ldots r_k$, $t_{c_1}, \ldots, t_{c_k}$ corresponding to the rows
$c_1, \ldots, c_k$, and $t_{e_{11}}, \ldots, t_{e_{kk}}$ corresponding to the
rows $e_{11}, \ldots, e_{kk}$. The first $k$ columns of $\hat{D}$ require
$t_{c_j} - t_{e_{1j}} \equiv 0 (\text{mod }1)$. The next $(k-1)k$ columns
require $t_{r_i} + t_{c_j} - t_{e_{ij}} \equiv 0 (\text{mod }1)$. The remaining
$|E|$ columns require $t_{e_u} + t_{e_v} \equiv 0 (\text{mod }1)$ whenever $u
\sim v$. As $\Gamma$ is connected and contains an odd cycle, we may either set
$t_{e_{ij}} = 0$ for each $i,j$ or $t_{e_{ij}} = \frac{1}{2}$. Whichever choice
we make determine $t_{c_j}$ for all $j$, due to the first $k$ columns, which in
turn determine $t_{r_i}$. Therefore once we make the choice to set each
$t_{e_{ij}}$ to $0$ or $\frac{1}{2}$, the other variables are determined,
meaning $q=2$. Thus
\[ \det(\Lambda) = \det(\mathcal{L}^\perp) = \frac{1}{2}\det(\mathcal{L}_0). \]
By Lemma 2.1 in \cite{GJR10},
\[ \det(\mathcal{L}_0) = \sqrt{\det(\hat{D}\hat{D}^T)} \]
and thus by Lemma~\ref{lem:det_l_nought}
\[ \det(\Lambda) = \frac{(k-1)^{k}(k(k-2))^{k-1}((k^2-k+4)(k^2-3k+4))^{\frac{1}{4}(k-1)^2}}{\sqrt{2}^{k^2+1}}. \]

Unfortunately, we were not able to determine $\det(-H|_{\mathbb{V}})$. We
consider the implications of finding an exact value in
Section~\ref{ssec:second_mom_const}. 

For now, we note that since the function $f(A, B)$ achieves a unique global maximum in the lattice, there is some constant, $h_{d, k} > 0$, depending on $d$ and $k$ so that 
\[
 \sum_{\ell \in (\Lambda +\ell_n) \cap nK}a_n(\ell) \le  h_{d, k} (2\pi n)^{r/2} \psi(z_0) b_n e^{n\phi(z_0)}
\]

Noting that
\[\psi(z_0) = \prod_{uv \in E} \frac{1}{\sqrt{\frac{2}{k^2(k-1)^2}}} = \left(\frac{k(k-1)}{\sqrt{2}}\right)^{\frac{1}{4}k^2(k-1)^2}\]
and
\[ \phi(z_0) = -\sum_{v \in V} \frac{1}{k^2} \log \frac{1}{k^2} + d\sum_{uv \in E} \frac{2}{k^2(k-1)^2}
\log\left(\frac{\frac{1}{k^2}\cdot\frac{1}{k^2}}{2\cdot\frac{2}{k^2(k-1)^2}}\right) = \log k^2 + d\log \left(\frac{1}{2}\left(1-\frac{1}{k}\right)\right),\]
we have
\begin{align}
\ex Y^2 &= \sum_{A \text{ s.t.} \eqref{eq:Acons}, \ B \text{ s.t.}
\eqref{eq:Bcons}} p(A,B,n) e^{nf(A,B)} \notag\\
	&= \sum_{\ell \in (\Lambda +\ell_n) \cap nK}a_n(\ell) \notag\\
	&\le  h_{d, k} (2\pi n)^{r/2} \psi(z_0) b_n e^{n\phi(z_0)} \notag \\
	&= C_{d, k} \cdot k^k\left(\frac{k-1}{2\pi(k-2)}\right)^{(k-1)}n^{-(k-1)}
k^{2n}\left(\frac{1}{2}\left(1-\frac{1}{k}\right)\right)^{dn} \notag\\
	&= C_{d,k} (\ex Y)^2 \label{eq:PZ_constant}
\end{align}
where $C_{d, k}$ is some positive constant depending on $d$ and $k$.
\end{proof}

We conclude this section by proving the following proposition.

\begin{proposition} \label{prop:gnd_ub}
For $k \ge 3$ such that a doubly regular tournament exists on $k$ vertices and $d < \ell_k$, the number of $k$-oriented colourings of $\vec{C} \sim \vec{\mathcal{C}}(n,d)$ is positive with positive probability.
\end{proposition}

\begin{proof}
By the Paley-Zygmund inequality and Lemma~\ref{prop:PZ_for_C},
\[ \pr[Y > 0] > \frac{(\ex Y)^2}{\ex(Y^2)} \sim \frac{(\ex Y)^2}{C_{d,k}(\ex Y)^2} = \frac{1}{C_{d,k}} > 0. \]
\end{proof}

\section{Range of Possible Values of the Oriented Chromatic Number} \label{sec:window}

In this section, we adapt an argument of {\L}uczak~\cite{Lu91b} who proved that for every $p=\frac{d}{n}$ with $d < n^{1/6-\varepsilon}$, the chromatic number of $\mathcal{G}(n,p)$ is concentrated in an interval of length two. (We note that our adaptation is closely based on that of Achlioptas and Moore~\cite{AM04}.) We prove a weaker result for $\vec{\mathcal{C}}(n,d)$, showing that there is $k=k(d)$ such that, conditioned on being an orientation of a simple graph, a.a.s.~the oriented chromatic number of $\vec{C} \sim \vec{\mathcal{C}}(n,d)$ is concentrated in an interval of length linear in $k$, as well as an analogous result for $\vec{\mathcal{M}}(n,d)$.

We start with a result regarding the oriented chromatic number of sparse digraphs.

\begin{lemma}[{\cite[Theorem 1]{BKNRS99}}]\label{lem:eleven_colour}
Let $\vec{G}$ be an oriented graph with the property that any subgraph of $\vec{G}$ has average degree strictly smaller than three. Then $\chi_o(\vec{G}) \le 11$.
\end{lemma}

Now we argue that a.a.s.~subgraphs of $\vec{\mathcal{C}}(n,d)$ of sub-linear size are sparse.

\begin{lemma}\label{lem:gnd_subgraph_colourable}
For any $0 < \varepsilon < \frac{1}{2}$ and constant $d$, with high probability every subgraph of $\vec{C} \sim \vec{\mathcal{C}}(n,d)$ induced by $s \le e^{-5}d^{-3}n^{1-\varepsilon}$ vertices contains at most $(\frac{3}{2}-\varepsilon)s$ edges.
\end{lemma}

It should be noted that it is possible to prove a stronger version of Lemma~\ref{lem:gnd_subgraph_colourable} where $d$ is allowed to grow bounded by a function of $n$ rather than being constant; however, this simpler version meets our current needs.

\begin{proof}
Let $\vec{C} \sim \vec{\mathcal{C}}(n,d)$ and let $S \subseteq V(\vec{C})$ such that $|S| = s$. Let $v \in S$ and $(u,v) \in E(\vec{C})$. The probability $u \in S$ is at most $\frac{s}{n}$. Thus
\[ \pr[\text{$S$ has at least $k$ edges}] \le \binom{ds}{k} \left(\frac{s}{n}\right)^k \]
and
\[ \pr[\text{There exists a set $S$ with at least $k$ edges}] \le \binom{n}{s}\binom{ds}{k} \left(\frac{s}{n}\right)^k. \]
Using $k = \frac{3}{2}-\varepsilon, 1 \le s \le e^{-5}d^{-3}n^{1-\varepsilon}$, and $0 < \varepsilon < \frac{1}{2}$ we get
\[ \binom{n}{s}\binom{ds}{k} \left(\frac{s}{n}\right)^k < \frac{de^2}{n^{\varepsilon(\frac{1}{2}-\varepsilon)}} = o(1) \]
as $\varepsilon(\frac{1}{2}-\varepsilon) > 0$ for $0 < \varepsilon <
\frac{1}{2}$. Therefore the probability such a set $S$ exists approaches zero as $n$ approaches infinity.
\end{proof}

\begin{corollary} \label{cor:gnd_subgraph_colourable}
Lemma~\ref{lem:gnd_subgraph_colourable} still holds conditioned on $\vec{C}$ being an orientation of a simple graph. In other words, for any $0 < \varepsilon < \frac{1}{2}$ and constant $d$, conditioned on $\vec{C} \sim \vec{\mathcal{C}}(n,d)$ being an orientation of a simple graph, with high probability every subgraph of $\vec{C}$ induced by $s \le e^{-5}d^{-3}n^{1-\varepsilon}$ vertices contains at most $(\frac{3}{2}-\varepsilon)s$ edges. 
\end{corollary}

\begin{proof}
It is well-known (see, for example, \cite{Bo80}) that the (undirected) configuration model produces a simple graph with probability tending toward $\exp(-\frac{1}{2}(d-1)-\frac{1}{4}(d-1)^2) =: \kappa > 0$, and thus $\vec{C}$ is an orientation of a simple graph with probability tending toward $\kappa$. If the probability that the result of Lemma~\ref{lem:gnd_subgraph_colourable} holding conditioned on $\vec{C}$ being an orientation of a simple graph tended toward $\alpha < 1$ as $n$ approaches infinity, then the unconditioned probability would tend to $(1-\kappa)(1) + \kappa\alpha < 1$, contradicting Lemma~\ref{lem:gnd_subgraph_colourable}.
\end{proof}

Next we appeal to a martingale result of Wormald on undirected configurations.

\begin{theorem}[{\cite[Theorem 2.19]{Wo99}}]\label{thm:original_martingale}
Given an undirected configuration $C$, we define a \emph{switching} in $C$ to be the replacement of two edges $\{v_1, v_2\}$, $\{v_3, v_4\}$ by the edges $\{v_1, v_3\}$ and $\{v_2,v_4\}$ or by the edges $\{v_1,v_4\}$ and $\{v_3,v_2\}$. Let $X_n$ be a random variable defined on $\mathcal{C}(n,d)$ such that for any configurations $C, C'$ that differ by a switching
\[ |X_n(C)-X_n(C')| \le b \]
for some constant $b > 0$. Then for every $t > 0$,
\[ \pr[X_n \le \ex[X_n] - t] \le \exp\left(\frac{-t^2}{dnb^2}\right) \quad \text{and} \quad \pr[X_n \ge \ex[X_n] + t] \le \exp\left(\frac{-t^2}{dnb^2}\right). \]
\end{theorem}

The proof of Theorem~\ref{thm:original_martingale} uses the Azuma-Hoeffding inequality for martingales by canonically ordering edges in a pairing model and defining an edge-exposure martingale.  The same proof gives an extension of Wormald's result to directed configurations as long as the constant $b$ is independent of the orientations of the two edges after switching.

Our next step is to prove that as soon as there is some small probability that $k$ colours are sufficient to colour $\vec{C} \sim \vec{\mathcal{C}}(n,d)$, we can almost surely $k$-colour all but a small number of vertices.

\begin{lemma}[{\cite[Lemma 2]{AM04}}]\label{lem:gnd_mostly_colourable}
For $\vec{C}\sim\vec{\mathcal{C}}(n,d)$, let $W_j(\vec{C})$ be the minimal size of a set of vertices $S$ such that $\vec{C}-S$ is $j$-oriented-colourable. For a given function $\omega(n)$, let $k = k(n, d, \omega(n))$ be the smallest $k$ such that
\[ \pr[W_k = 0] \ge 1/\omega(n). \]
With probability greater than $1-1/\omega(n)$, all but $8\sqrt{nd\log \omega(n)}$ vertices of $\vec{\mathcal{C}}(n,d)$ can be properly oriented-coloured using $k$ colours.
\end{lemma}

\begin{proof}
For any $j$ and $\vec{C}$, switching two arcs of $\vec{C}$ can affect $W_j(\vec{C})$ by at most 4 as, in the worst case, we can add all four vertices to $S$ and still $j$-oriented-colour $\vec{C}-S$. By Theorem~\ref{thm:original_martingale}, with $\mu_j = \ex[W_j(\vec{\mathcal{C}}(n,d))]$,
\[ \pr[W_j \le \mu_j - \lambda \sqrt{n}] < e^{-\frac{\lambda^2}{16d}} \quad \text{and} \quad \pr[W_j \ge \mu_j + \lambda \sqrt{n}] < e^{-\frac{\lambda^2}{16d}}. \]
Choose $\lambda = \lambda(n) = 4\sqrt{d\log \omega(n)}$ so that $e^{-\frac{\lambda^2}{16d}} = \frac{1}{\omega(n)}$. Then, recalling the definition of $k$,
\[ \pr[W_k \le \mu_k - \lambda \sqrt{n}] < \frac{1}{\omega(n)} = \pr[W_k = 0]. \]
If $\mu_k \ge \lambda \sqrt{n}$, then the event $W_k = 0$ implies the event $W_k \le \mu_k-\lambda \sqrt{n}$; since we have $\pr[W_k \le \mu_k-\lambda \sqrt{n}] < \pr[W_k = 0]$, this must not be the case and instead we must have $\mu_k < \lambda\sqrt{n}$. Therefore
\[ \frac{1}{\omega(n)} = e^{-\frac{\lambda^2}{16d}} > \pr[W_k \ge \mu_k + \lambda \sqrt{n}] \ge \pr[W_k \ge 2\lambda\sqrt{n}] \]
which completes the proof.
\end{proof}

Finally, we prove that a.a.s.~$\vec{C} \sim \vec{\mathcal{C}}(n,d)$ either is not an orientation of a simple graph or else $\chio(\vec{C})$ is concentrated in an interval of length linear in the number of colours.

\begin{proposition} \label{prop:gnd_window}
For every integer $d$, there exists an integer $k = k(d)$ such that a.a.s.~either $\vec{C} \sim \vec{\mathcal{C}}(n,d)$ is not an orientation of a simple graph or $\chio(\vec{C}) \in [k,2k+11]$.
\end{proposition}

\begin{proof}
Let $d$ be given and let $k = k(d,n)$ be the smallest integer for which the probability that $\vec{C} \sim\vec{\mathcal{C}}(n,d)$ is $k$-oriented-colourable is at least $1/\log \log n$. Note that such a $k$ exists by Proposition~\ref{prop:gnd_ub} as $C_{d,k}$ is independent of $n$. Fix $\vec{C} \sim \vec{\mathcal{C}}(n,d)$. By Lemma~\ref{lem:gnd_mostly_colourable}, with high probability there exists a set of vertices $S$ such that $\vec{C}-S$ can be $k$-oriented-coloured and $|S| < 8\sqrt{nd \log \log \log n} < \sqrt{nd} \log n \equiv s_0$. From $S$, we will construct an increasing sequence of sets of vertices $\{U_i\}$ as follows. $U_0 = S$; for $i \ge 0, U_{i+1} = U_i \cup W$, where either $W = \{w_1,w_2\}$ for some $w_1,w_2 \notin U_i$ which are adjacent and each of which has some neighbour in $U_i$ or $W = \{w\}$ for some $w \notin U_i$ for which there are $u_1,u_2 \in U_i$ such that $(u_1,w)$ and $(w,u_2)$ are arcs in $\vec{C}$. The construction ends, with $U_t$, when no such pair exists.

Define $N = \{v \in \vec{C} : v \notin U_t \ \text{and} \ \exists u \in U_t \ \text{such that} \ u\sim v\}$
to be the neighbourhood of $U_t$ in the rest of the graph. Note that $N$ is an independent set; if $v_1, v_2 \in N$ and $v_1 \sim v_2$, then we could have continued in the construction with $U_{t+1} = U_t\cup\{v_1,v_2\}$. Furthermore, $N = N_I \cup N_O$ where each edge is directed from $U_t$ into $N_I$ and each edge is directed from $N_O$ out to $U_t$; if some $v \in N$ satisfies neither of these constraints, then there are $u_1, u_2 \in U_t$ such that $(u_1, v)$ and $(v,u_2)$ are arcs of $\vec{C}$ and we could have continued the construction with $U_{t+1} = U_t \cup \{v\}$.

We claim that if $\vec{C}$ is an orientation of a simple graph, then we can properly $11$-oriented-colour $U_t$ with high probability. Set $\varepsilon = 0.1$ and let $n$ be large enough so that
\[ \frac{s_0(1+2\log n)}{s_0(1+2\log n)+2} > 1-\frac{\varepsilon}{3}, \quad \frac{3\log n}{1+2\log n} >  \frac{3}{2}- \frac{\varepsilon}{2}, \]
and that $n^{1/2-\varepsilon} > 3e^{5}d^{7/2}(\log n)^2$. It suffices to show that $|U_t| < s_0(1+2\log n)$: then
\[ s_0(1+2\log n) = \sqrt{nd}\log n(1+2 \log n) \le 3\sqrt{nd}(\log n)^2 \le e^{-5}d^{-3}n^{1-\varepsilon} \]
so Corollary~\ref{cor:gnd_subgraph_colourable} ensures the condition of Lemma~\ref{lem:eleven_colour}.

Assume for contradiction that $|U_t| \ge s_0(1+2\log n)$ and let $i^\ast$ be the first step at which $|U_{i^\ast}| \ge s_0(1+2\log n)$. Because $|U_0| = s_0$, at least $2s_0\log n$ vertices have been added. At step $i^\ast$ we added at most two vertices to $U_{i^\ast-1}$, and therefore $|U_{i^\ast}| \le s_0(1+2\log n)+2$. Note that in steps in which two vertices are added to $U_i$, three edges are added, and in steps in which one vertex is added, we add two edges to $U_i$. Thus each vertex added to $U_{i^\ast}$ also added at least $\frac{3}{2}$ edges to $U_{i^\ast}$, from which we conclude $U_{i^\ast}$ contains at least $\frac{3}{2}(2s_0\log n) = 3s_0\log n$ edges. But then	
\[ e(U_{i^\ast}) = 3s_0\log n = (s_0 (1+2\log n)+2)\cdot \frac{s_0(1+2\log n)}{s_0(1+2\log n)+2} \cdot \frac{3\log n}{1+2\log n} > |U_{i^\ast}|(1-\tfrac{\varepsilon}{3})(\tfrac{3}{2}-\tfrac{\varepsilon}{2}) > |U_{i^\ast}|(\tfrac{3}{2}-\varepsilon). \]
This contradicts Corollary~\ref{cor:gnd_subgraph_colourable} as
\[ |U_{i^\ast}| \le s_0(1+2\log n)+2 = \sqrt{nd}\log n(1+2 \log n)+2 \le 3\sqrt{nd}(\log n)^2 \le e^{-5}d^{-3}n^{1-\varepsilon}. \]

Now recall that $\vec{C}-S$ is $k$-oriented-colourable, say by $\chio:\vec{C}-S \to [k]$. We colour $\vec{C}$ as follows. First, for $v \in \vec{C} - U_t - N_O$, colour $v$ with $(\chio(v),0)$. As $N_O \subseteq \vec{C}-U_t \subseteq \vec{C}-S$, each vertex in $N_O$ is also coloured by $\chio$. Colour each $v \in N_O$ with colour $(\chio(v),1)$. Finally, use the eleven unique colours $(1,2), \ldots, (11,2)$ to colour $U_t$. We claim this is a proper oriented $(2k+11)$-oriented-colouring of $\vec{C}$. First, no adjacent vertices share a colour: $U_t$ is properly coloured by Lemma~\ref{lem:eleven_colour}, $\vec{C}-U_t$ is properly $k$-oriented-coloured as $\vec{C}-S$ was, and the vertices in $U_t$ are coloured with colours distinct from those in $\vec{C}-U_t$. Thus if this colouring is not proper there must be arcs $v_1 \to v_2$ and $v_3\to v_4$ such that $v_1$ and $v_4$ receive the same colour $c_1$ while $v_2$ and $v_3$ receive the same colour $c_2$. By our construction, this requires both pairs $(v_1,v_4)$ and $(v_2,v_3)$ occur in the same part of the partition of $\vec{C}$ into $U_t, N_O$, and $\vec{C}-U_t-N_O$. Both pairs cannot be in the same part of this partition as each part is properly oriented coloured. By symmetry, this leaves three cases to consider: $v_1,v_4 \in U_t$ and $v_2,v_3 \in N_O$, $v_1,v_4 \in U_t$ and $v_2,v_3 \in \vec{C}-U_t-N_O$, or $v_1,v_4 \in N_O$ and $v_2,v_3 \in \vec{C}-U_t-N_O$. In the first case, we see the arc between $v_1$ and $v_2$ is directed toward $v_1$, not $v_1\to v_2$, because every arc between $N_O$ and $U_t$ is oriented towards $U_t$. Similarly, in the second case the only arcs from $U_t$ to $\vec{C}-U_t-N_O$ originate in $U_t$ and are directed towards $N_I$, so we have $v_4 \to v_3$ instead of $v_3\to v_4$. In the last case we have $v_1,v_2,v_3,v_4 \in \vec{C}-S$, so as $\chio$ is a proper $k$-oriented-colouring, we cannot have $v_1\to v_2$ and $v_3 \to v_4$ if $v_1,v_4$ share a colour and $v_2,v_3$ also share a colour. We conclude no such arcs occur and thus we have properly $(2k+11)$-oriented-coloured $\vec{C}$.
\end{proof}

A version of Proposition~\ref{prop:gnd_window} holds for $\vec{\mathcal{M}}(n,m)$ holds for $m=cn, c > 0$ as well.

\begin{proposition} \label{prop:gnm_window}
For every $c > 0$ there exists an integer $k = k(c)$ such that a.a.s.~either $\vec{M} \sim \vec{\mathcal{M}}(n,m=cn)$ is not an orientation of a simple graph or $\chio(\vec{M}) \in [k,2k+11]$.
\end{proposition}

The proof of Proposition~\ref{prop:gnm_window} is nearly identical to that of Proposition~\ref{prop:gnd_window}, so rather than repeat each of the intermediate results, we summarize the small changes in the proofs. In Lemma~\ref{lem:gnd_subgraph_colourable}, the probability that a ``bad'' subset exists is $\binom{n}{s}\binom{m}{k}(\binom{s}{2}/\binom{n}{2})^k$. This probability is still less than $(de^2)/n^{\varepsilon(\frac{1}{2}-\varepsilon)}$, so the corresponding result holds. We show $\vec{M} \sim \vec{\mathcal{M}}(n,m=cn)$ is an orientation of a simple graph with positive probability in Lemma~\ref{lem:gnm_simple} in the next section, so Corollary~\ref{cor:gnd_subgraph_colourable} holds as well. Instead of appealing to Theorem~\ref{thm:original_martingale}, we can use Azuma's inequality directly by exposing the edges of the graph one at a time. This has the effect of changing the size of the uncoloured set in Lemma~\ref{lem:gnd_mostly_colourable} by a constant multiple. That difference disappears in the step of the proof of Proposition~\ref{prop:gnd_window} in which we set $s_0 = \sqrt{nd} \log n$.

\section{Proofs of Main Theorems} \label{sec:main}

\begin{lemma}\label{lem:gnm_simple}
For any constant $c > 0$ and $m \le cn + n^{2/3}$, the probability that $\vec{\mathcal{M}}(n,m)$ is an orientation of a simple graph (without loops or multiple edges) $\exp(-c(c+1) + o(1)) \ge \exp(-2c(c+1)) > 0$.
\end{lemma}

\begin{proof}
Using the inequality $1-x \ge \exp(-x-x^2)$, which holds for $x$ small enough, the probability that $\vec{M} \sim \vec{\mathcal{M}}(n,m)$ is an orientation of a simple graph is

\begin{align*}
\frac{1}{n^{2m}}\prod_{j=0}^{m-1} n(n-1)-2j
	&=\prod_{j=0}^{m-1} \left(1 - \frac{(n+2j)}{n^2}\right)^m\\
	&\ge \prod_{j = 0}^{m-1} \exp\left(- \left[ \frac{(n+2j)}{n^2} + \left( \frac{(n+2j)}{n^2} \right)^2 \right] \right)\\
	&\ge \prod_{j = 0}^{m-1}\exp\left(-\left[\frac{n + 2j}{n^2} + \frac{(2c+3)^2}{n^2}\right] \right)	&\text{(since $n+2m \le (2c+3)n$)}\\
	&\ge \exp\left(-\left[-\frac{m}{n} + \frac{m^2}{n^2} + \frac{(2c+3)^2m}{n^2}\right] \right)\\
	&\ge \exp\left(-(c+c^2 + (2c+1)n^{-1/3} + (2c+3)^2(c+1)n^{-1})\right) \hspace{-10pt}\\
	&= \exp(-c(c+1) + o(1)) \ge \exp\left(-2c(c+1)\right).	&\text{(for $n$ large enough)}
\end{align*}
\end{proof}

\begin{corollary}\label{cor:gnm_mnm_equiv}
For any constant $c > 0$, $n$ sufficiently large, any $m \le cn + n^{2/3}$, and any event $A$ in both $\vec{\mathcal{G}}(n,p)$ and $\vec{\mathcal{M}}(n,m)$,
\[ \pr_{\vec{G} \sim \vec{\mathcal{G}}(n,p)}[A \mid e(\vec{G}) = m] \le e^{2c(c+1)}\pr_{\vec{M}\sim\vec{\mathcal{M}}(n,m)}[A]. \]
\end{corollary}

\begin{proof}
The key idea is that $\vec{G} \sim \vec{\mathcal{G}}(n,p)$ conditioned on having $m$ arcs and $\vec{M} \sim \vec{\mathcal{M}}(n,m)$ conditioned on being simple are both uniform over the collection of oriented graphs of order $n$ with $m$ arcs. Thus
\begin{align*}
\pr_{\vec{G} \sim \vec{\mathcal{G}}(n,p)}[A \mid e(\vec{G}) = m] &= \pr_{\vec{M} \sim \vec{\mathcal{M}}(n,m)}[A \mid \text{$\vec{M}$ is simple}]\\
	&= \frac{\pr_{\vec{M} \sim \vec{\mathcal{M}}(n,m)}[A \text{ and $\vec{M}$ is simple}]}{\pr_{\vec{M} \sim \vec{\mathcal{M}}(n,m)}[\text{$\vec{M}$ is simple}]}\\
	&\le e^{2c(c+1)}\pr_{\vec{M} \sim \vec{\mathcal{M}}(n,m)}[A] & \text{(by Lemma~\ref{lem:gnm_simple})}.
\end{align*}
\end{proof}

\begin{lemma}\label{lem:edge_conc}
For any constant $d > 0$, the number of edges of $\vec{G} \sim \vec{\mathcal{G}}(n,p=\frac{d}{n})$ satisfies
\[
\pr_{\vec{G} \sim \vec{\mathcal{G}}(n,p=\frac{d}{n})}[|e(\vec{G}) - dn/2| \ge n^{2/3}] \le 2\exp\left(-\frac{n^{1/3}}{4d}\right) = o(1).
\]
\end{lemma}

\begin{proof}
Set $c = d/2$.  By standard concentration results, for $\vec{G} \sim \vec{\mathcal{G}}(n,p=\frac{d}{n})$, since $\ex(e(\vec{G})) = \binom{n}{2}p = \frac{d(n-1)}{2} = c(n-1)$,
\begin{align*}
\pr_{\vec{G} \sim \vec{\mathcal{G}}(n,p=\frac{d}{n})}[|e(\vec{G}) - cn| \ge n^{2/3}]
	&\le \pr_{\vec{G} \sim \vec{\mathcal{G}}(n,p=\frac{d}{n})}[|e(\vec{G})) - \ex(e(\vec{G}))| \ge n^{2/3}-c]\\
	&\le 2 \exp\left(-\frac{(n^{2/3}-c)^2}{2(\binom{n}{2}\frac{d}{n} +(n^{2/3}-c)/3)}\right)\\
	&\le 2 \exp\left(- \frac{n^{4/3}/2}{d(n-1) + n^{2/3}}\right)\\
	&\le 2 \exp\left(-\frac{n^{1/3}}{4d}\right)	&&\text{(for $n$ large enough)}\\
	&=o(1).
\end{align*}
\end{proof}

\begin{proof}[Proof of Theorem~\ref{thm:gnp_main}]
Let $d > 1$ be given and let $k_1$ be an integer satisfying $d > u_{(k_1-1)}$. Set $c = \frac{d}{2}$. By Corolloary~\ref{cor:gnm_mnm_equiv}, letting $A(\vec{G})$ be the event that a $(k_1-1)$-oriented-colouring of $\vec{G}$ exists, for any $m \le cn+n^{2/3}$,
\[
\pr_{\vec{G} \sim \vec{\mathcal{G}}(n,p=\frac{d}{n})}[A(\vec{G}) \mid e(\vec{G}) = m] \le e^{2c(c+1)}\pr_{\vec{M} \sim \vec{\mathcal{M}}(n, m)}[A(\vec{M})].
\]
Thus, by Lemma~\ref{lem:edge_conc}, the law of total probability, and noting that the probability that $\vec{M} \sim \vec{\mathcal{M}}_{n, m}$ has a $(k_1-1)$-oriented-colouring as $m$ increases (coupling the models $\vec{\mathcal{M}}_{n, m}$ and $\vec{\mathcal{M}}_{n, m'}$ with $m<m'$ by inclusion), with $\vec{\mathcal{G}} = \vec{\mathcal{G}}(n,p=\frac{d}{n})$,
\begin{align*}
\pr_{\vec{G} \sim \vec{\mathcal{G}}}[A(\vec{G})] &= \pr_{\vec{G} \sim \vec{\mathcal{G}}}[A(\vec{G}) \text{ and } |e(\vec{G}) - cn| \le n^{2/3}] +\pr_{\vec{G} \sim \vec{\mathcal{G}}}[A(\vec{G}) \text{ and } |e(\vec{G}) - cn| > n^{2/3}]\\ 
	&\le \sum_{m=\lceil cn-n^{2/3}\rceil}^{\lfloor cn+n^{2/3}\rfloor}\pr_{\vec{G} \sim \vec{\mathcal{G}}}[A(\vec{G}) \mid e(\vec{G}) = m]\pr_{\vec{G} \sim \vec{\mathcal{G}}}[e(\vec{G}) = m] \quad +\\
	&\qquad\qquad \pr_{\vec{G} \sim \vec{\mathcal{G}}}[|e(\vec{G}) - cn| > n^{2/3}] \\
	&\le e^{2c(c+1)}\pr_{\vec{M} \sim \vec{\mathcal{M}}(n, cn-n^{2/3})}[A(\vec{M})]\sum_{m=\lceil cn-n^{2/3}\rceil}^{\lfloor cn+n^{2/3}\rfloor}\pr_{\vec{G} \sim \vec{\mathcal{G}}}[e(\vec{G}) = m] + o(1)\\
	&\le e^{2c(c+1)}\pr_{\vec{M} \sim \vec{\mathcal{M}}(n, cn-n^{2/3})}[A(\vec{M})] + o(1)
\end{align*}

Pick $c'$ satisfying $c > c' > \frac{1}{2}u_{(k_1-1)}$ and let $n$ be large enough that $cn-n^{2/3} \ge c'n$. Then by Proposition~\ref{prop:gnm_lb}, noting the probability that an oriented $(k_1-1)$-colouring exists is decreasing in the number of edges,
\[ e^{2c(c+1)}\pr_{\vec{M} \sim \vec{\mathcal{M}}(n, cn-n^{2/3})}[A(\vec{M})] \le e^{2c(c+1)}\pr_{\vec{M} \sim \vec{\mathcal{M}}(n,m=c'n)}A(\vec{M})] = o(1) \]
and therefore a.a.s.~$A(\vec{G})$ does not occur, implying $\chio(\vec{\mathcal{G}}(n,p=\frac{d}{n})) > k_1-1$.

Now let $k_2 \ge 3$ such that $d < \ell_{k_2}$ and there exists a doubly regular tournament of order $k_2$. Set $c = d/2$, let $c'$ be such that $c < c' < \frac{1}{2}\ell_{k_2}$ and let $n$ be large enough that $cn + n^{2/3} \le c' n$.  Using Corollary~\ref{cor:gnm_mnm_equiv} where $B(\vec{G})$ is the event that no $(2k_2+11)$-oriented colouring of $\vec{G}$ exists, for any $m \le c'n+n^{2/3}$,
\[
\pr_{\vec{G} \sim \vec{\mathcal{G}}(n,p=\frac{d}{n})}[B(\vec{G}) \mid e(\vec{G}) = m] \le e^{2c'(c'+1)}\pr_{\vec{M} \sim \vec{\mathcal{M}}_{n, m}}[B(\vec{M})].
\]
Thus, using Lemma~\ref{lem:edge_conc}, that the probability the oriented chromatic number is greater than $2k_2+11$ is increasing in the number of edges, and the law of total probability, with $\vec{\mathcal{G}} = \vec{\mathcal{G}}(n,p=\frac{d}{n})$,
\begin{align*}
\pr_{\vec{G} \sim \vec{\mathcal{G}}}[B(\vec{G})] &= \pr_{\vec{G} \sim \vec{\mathcal{G}}}[B(\vec{G}) \text{ and } e(\vec{G}) > c'n] +\sum_{m=0}^{\lfloor c'n\rfloor}\pr_{\vec{G} \sim \vec{\mathcal{G}}}[B(\vec{G}) \text{ and } e(\vec{G}) = m]\\
	&\le \pr_{\vec{G} \sim \vec{\mathcal{G}}}[e(\vec{G}) > c'n \ge cn + n^{2/3}] \quad +\\
	&\qquad \qquad \pr_{\vec{M} \sim \vec{\mathcal{M}}(n,m=c'n)}[B(\vec{M}) \mid \text{$\vec{M}$ is simple}]\sum_{m=0}^{\lfloor c'n\rfloor}\pr_{\vec{G} \sim \vec{\mathcal{G}}}[e(\vec{G}) = m]\\
	&\le o(1) + \pr_{\vec{M} \sim \vec{\mathcal{M}}(n,m=c'n)}[B(\vec{M}) \mid \text{$\vec{M}$ is simple}].
\end{align*}
Then by Proposition~\ref{prop:gnm_ub}, as $1 < c' < \frac{1}{2}\ell_{k_2}$, $\vec{M}\sim\vec{\mathcal{M}}(n,m=c'n)$ has a proper oriented $k_2$-colouring with positive probability. Proposition~\ref{prop:gnm_window} then guarantees $\pr_{\vec{M}\sim\vec{\mathcal{M}}(n,m=c'n)}[B(\vec{M})\mid \text{$\vec{M}$ is simple}] = o(1)$. We conclude a.a.s.
\[ \chio(\vec{\mathcal{G}}(n,p=\tfrac{d}{n})) \in [k_1,2k_2+11]. \]
\end{proof}

\begin{proof}[Proof of Theorem~\ref{thm:gnd_main}]
It is well-known (see, for example, \cite{Bo80}) that the (undirected) configuration model produces a simple graph with probability tending toward $\exp(-\frac{1}{2}(d-1)-\frac{1}{4}(d-1)^2) =: \kappa > 0$. Orienting an undirected simple graph produces a simple oriented graph, so with positive probability $\vec{C} \sim \vec{\mathcal{C}}(n,d)$ is a simple digraph. Furthermore, these simple digraphs are uniformly distributed as each simple graph corresponds to the same number of configurations. Therefore, if $A$ is any event in both $\vec{\mathcal{G}}(n,d)$ and $\vec{\mathcal{C}}(n,d)$, then
\begin{align*}
\pr_{\vec{G}\sim\vec{\mathcal{G}}(n,d)}[A] &= \pr_{\vec{C}\sim\vec{\mathcal{C}}(n,d)}[A \mid \text{$\vec{C}$ is simple}]\\
	&= \frac{\pr_{\vec{C}\sim\vec{\mathcal{C}}(n,d)}[A \text{ and $\vec{C}$ is simple}]}{\pr_{\vec{C}\sim\vec{\mathcal{C}}(n,d)}[\text{$\vec{C}$ is simple}]}\\
	&\le \kappa^{-1}\pr_{\vec{C}\sim\vec{\mathcal{C}}(n,d)}[A]
\end{align*}

For $k_1 \ge 2$ and $d \ge u_{(k_1-1)}$, with $A(\vec{G})$ the event that $G$ has a $(k_1-1)$-oriented-colouring, Proposition~\ref{prop:gnd_lb} assures $\pr_{\vec{C}\sim\vec{\mathcal{C}}(n,d)}[A(\vec{C})] = o(1)$ and therefore 
\[ \pr_{\vec{G}\sim\vec{\mathcal{G}}(n,d)}[A(\vec{G})] \le \kappa^{-1}\cdot o(1) = o(1),\]
or equivalently $\pr_{\vec{G}\sim\vec{\mathcal{G}}(n,d)}[\chio(\vec{G}) \le k_1-1] = o(1)$. On the other hand, if there is a $k_2 \ge 3$ such that $d < \ell_{k_2}$ and there exists a doubly regular tournament of order $k_2$, then Proposition~\ref{prop:gnd_ub} gives $\vec{C} \sim \vec{\mathcal{C}}(n,d))$ has a $k_2$-oriented-colouring with positive probability. Then with the event $B(\vec{G})$ being that $\vec{G}$ has no $(2k_2+11)$-oriented-colouring, Proposition~\ref{prop:gnd_window} guarantees $\pr_{\vec{C} \sim \vec{\mathcal{C}}(n,d)}[B(\vec{C})] = o(1)$, so
\[ \pr_{\vec{G}\sim\vec{\mathcal{G}}(n,d)}[\chio(\vec{G}) > 2k_2+11] = \pr_{\vec{G}\sim\vec{\mathcal{G}}(n,d)}[B(\vec{G})] \le \kappa^{-1} \cdot o(1) = o(1).\]
We conclude a.a.s.
\[ \chio(\vec{\mathcal{G}}(n,d)) \in [k_1,2k_2+11]. \]
\end{proof}

\begin{proof}[Proof of Corollary~\ref{cor:main_cor}]
Let $d > 1$ be given and set $k = 2^{d/2}$. Then
\[ u_{k} < \frac{2}{\log 2} \log(k) = d \]

Setting $k' = k^{1/\log 2}+\frac{2}{\log 2}\log k + 1 = e^{d/2}+d+1$, 
\[ d <  2\left(1-\frac{2}{k'+1}\right)\log(k'-1) < \ell_{k'}. \]
Thus we can apply Theorems~\ref{thm:gnp_main} and \ref{thm:gnd_main} for any $k'' \ge k'$ such that there exists a doubly regular tournament of order $k''$. Lemma~\ref{lem:bertrand_postulate_dr_tourn} guarantees the smallest $k''$ for which this holds is at most $2k'$, and thus
\[ \chio(\vec{\mathcal{G}}(n,p=\tfrac{d}{n})), \chio(\vec{\mathcal{G}}(n,d)) \in (k,2(2k')+11] = (2^{d/2}, 4e^{d/2}+4d+15]. \]
\end{proof}

\section{Further Discussion} \label{sec:further_disc}

\subsection{Very sparse or very dense graphs}\label{ssec:vsparse_vdense}

Throughout the majority of this paper, we have focused on the case of oriented graphs with edge density $\frac{d}{2}$, mainly for $d$ fixed and large. However, in the case that $d \le 2$ it is possible to use the properties of random graphs and constraints on oriented colourings with a small number of colours to say something more precise than what is given by the second moment argument.  What makes these cases more tractable for analysis is that there are, up to isomorphism, only two different tournaments on $3$ vertices ($C_3$ and $T_3$ in Figure~\ref{fig:small-tourn}) and there are characterizations of when an oriented unicyclic graph can be properly coloured by either one.

\begin{figure}[htb]
\begin{center}
\begin{tikzpicture}
	[decoration={markings, mark=at position 0.7 with {\arrow{>}}}] 
	\tikzstyle{vertex}=[circle, fill=black,  minimum size=5pt,inner sep=0pt]
	
	\foreach \x/\l in {1/right, 2/above, 3/left}
	{
		\node[vertex, label = \l:{$\x$}] at ({120*\x-150}:0.75cm) (\x) {};
	}
	\draw[postaction={decorate}] (1) -- (2);
	\draw[postaction={decorate}] (2) -- (3);
	\draw[postaction={decorate}] (3) -- (1);
	
	\node at (0, -1) {$C_3$};
	
\end{tikzpicture}
\hspace*{0.5in}
\begin{tikzpicture}
	[decoration={markings, mark=at position 0.7 with {\arrow{>}}}] 
	\tikzstyle{vertex}=[circle, fill=black,  minimum size=5pt,inner sep=0pt]
	
	\foreach \x/\l in {1/right, 2/above, 3/left}
	{
		\node[vertex, label = \l:{$\x$}] at ({120*\x-150}:0.75cm) (\x) {};
	}
	\draw[postaction={decorate}] (1) -- (2);
	\draw[postaction={decorate}] (2) -- (3);
	\draw[postaction={decorate}] (1) -- (3);
	
	\node at (0, -1) {$T_3$};
	
\end{tikzpicture}
\hspace*{0.5in}
\begin{tikzpicture}
	[decoration={markings, mark=at position 0.7 with {\arrow{>}}}] 
	\tikzstyle{vertex}=[circle, fill=black,  minimum size=5pt,inner sep=0pt]
	
	\node[vertex, label = above:{$1$}] at (1, 1) (1) {};
	\node[vertex, label = above:{$2$}] at (0, 1) (2) {};
	\node[vertex, label = below:{$3$}] at (0, 0) (3) {};
	\node[vertex, label = below:{$4$}] at (1, 0) (4) {};
	
	\draw[postaction={decorate}] (1) -- (2);
	\draw[postaction={decorate}] (2) -- (3);
	\draw[postaction={decorate}] (3) -- (4);
	\draw[postaction={decorate}] (4) -- (1);
	\draw[postaction={decorate}] (3) -- (1);
	\draw[postaction={decorate}] (2) -- (4);
	
	\node at (0.5, -1) {$T_4$};
\end{tikzpicture}
\hspace*{0.5in}
\begin{tikzpicture}
	[decoration={markings, mark=at position 0.7 with {\arrow{>}}}] 
	\tikzstyle{vertex}=[circle, fill=black,  minimum size=5pt,inner sep=0pt]
	
	\foreach \x/\l in {1/right, 2/above, 3/left, 4/below, 5/below}
	{
		\node[vertex, label = \l:{$\x$}] at ({72*\x-54}:1cm)  (\x) {};
	}
	\draw[postaction={decorate}] (1) -- (2);
	\draw[postaction={decorate}] (2) -- (3);
	\draw[postaction={decorate}] (3) -- (4);
	\draw[postaction={decorate}] (4) -- (5);
	\draw[postaction={decorate}] (5) -- (1);
	\draw[postaction={decorate}] (1) -- (3);
	\draw[postaction={decorate}] (3) -- (5);
	\draw[postaction={decorate}] (5) -- (2);
	\draw[postaction={decorate}] (2) -- (4);
	\draw[postaction={decorate}] (4) -- (1);
	
	\node at (0, -1.5) {$T_5$};
\end{tikzpicture}
\end{center}
\caption{Small tournaments for colouring oriented $2$-regular graphs}\label{fig:small-tourn}
\end{figure}
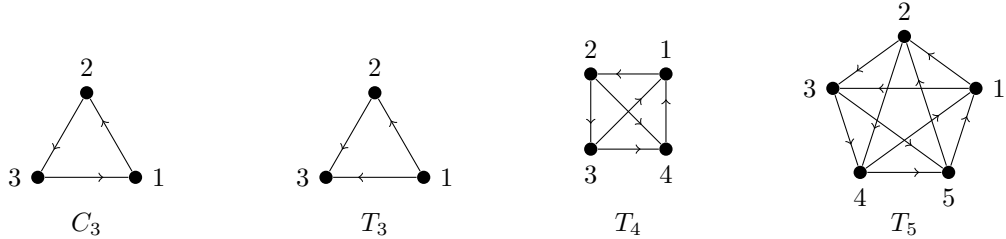

In his survey, Sopena~\cite{So16} gave a complete characterization of the
oriented chromatic number of oriented cycles, showing that for an oriented cycle
$\vec{C}$, the possible values of $\chi_o(\vec{C})$ are $2, 3, 4$, or $5$,
exactly characterizing each case. The proofs can be modified to give the
corresponding results for here for $2$-regular oriented graphs that are not
necessarily connected.  The main ideas of the proofs are that both of the
tournaments $T_4$ and $T_5$ in Figure~\ref{fig:small-tourn} contain directed
$3$-cycles, directed $4$-cycles, and the transitive tournament on $3$ vertices.
This can be used to show that every directed cycle can be properly coloured with
either one, except in the case of the directed $5$-cycle that can only be
properly $T_5$-coloured.  Since both tournaments $T_4$ and $T_5$ have the
property that every vertex has non-zero in-degree and out-degree, the rest
follows by induction on the number of source or sink vertices, with a few
special cases to check.  Since both tournaments $T_4$ and $T_5$ have the
property that every vertex has both in-neighbours and out-neighbours, any proper
colouring of the cycle in an orientation of a unicyclic graph can be extended to
a colouring of the remaining vertices.

\begin{proposition}[see Sopena~{\cite[Proposition 10]{So16}}]\label{prop:4or5-colouring}
If $\vec{G}$ is an orientation of a graph in which each component is a tree or unicyclic, then $\chi_o(\vec{G}) \le 5$ and $\vec{G}$ can be properly $5$-coloured by the tournament $T_5$ in Figure~\ref{fig:small-tourn}.  Moreover, if no component of $\vec{G}$ contains a directed $5$-cycle, then $\chi_o(\vec{G}) \le 4$ and $\vec{G}$ can be properly $4$-coloured by the tournament $T_4$ in Figure~\ref{fig:small-tourn}.
\end{proposition}

Within the proof, Sopena~\cite[Proposition 10]{So16} notes the following straightforward characterization of the $2$-regular oriented graphs that have oriented $3$-colourings.  Up to isomorphism, there are only two different tournaments on $3$ vertices: the directed $3$-cycle ($C_3$) and the transitive tournament on $3$ vertices ($T_3$).  Let $\vec{G}$ be an oriented cycle with vertices labelled around the cycle $u_1, u_2, \ldots, u_k$.  Call an arc a \emph{forward arc} if it is directed $u_i \to u_{i+1}$ (indices computed modulo $k$) and a \emph{backward arc} if it is directed $u_{i+1} \to u_i$.  Then $\vec{G}$ has a proper $C_3$-colouring if{f} the difference between the number of forward arcs and the number of backward arcs is divisible by $3$.  The graph $\vec{G}$ has a proper $T_3$-colouring if{f} there are no 3 consecutive arcs in the same direction.  Neither of these properties depends on the choice of starting vertex for the cycle's labelling, nor the direction of the labelling.  Using these properties it is possible to show that with high probability, some component of a large $2$-regular oriented graph does not satisfy either property. 

To see that a proper $C_3$-colouring is not likely, one can show, by induction, that for any $k \ge 3$, in a random orientation of a $k$-cycle, with probability at most $1/2$, the difference between the number of forward arcs and backward arcs is divisible by $3$.  One can check that in a random $2$-regular graph on $n$ vertices, with high probability, there are at least $\log n$ components.  Thus, the probability that a random orientation of a random $2$-regular graph is $C_3$-colourable is $o(1)$.

We now consider the property of being $T_3$-colourable, which is equivalent to not having three consecutive arcs in the same direction.  The number of orientations of a $k$-cycle with this property is the same as the number of cyclic sequences of $0$s and $1$s of length $k$ that start with a $1$ and have no substring $000$ or $111$.  This is bounded above by the number of binary strings (not cyclic) of length $k$ that start with a $1$ and have no three consecutive $0$s nor three consecutive $1$s.  By checking cases for how such a  string can begin, one can give a linear recurrence for this number and show that, if $\phi = \frac{1+ \sqrt{5}}{2}$, the number of such strings is bounded above by $\phi^k$.  Since the sum of all cycle lengths in a $2$-regular graph on $n$ vertices is $n$, this shows that the probability that $\vec{G} \sim \vec{\mathcal{G}}_{n, 2}$ is properly $T_3$-colourable is at most $\left(\frac{\phi}{2}\right)^n = o(1)$.

Thus, with high probability, for $\vec{G} \sim \vec{\mathcal{G}}(n, 2)$, $\chi_o(\vec{G}) > 3$.

Every oriented $2$-regular graph has a proper oriented $5$ colouring and Proposition~\ref{prop:4or5-colouring} shows that whether $5$ colours are needed depends on the presence of a directed $5$-cycle.  Since the number of $5$-cycles in the undirected graph $G_{n, 2}$ is asymptotically distributed as a Poisson random variable with mean $\frac{1}{2 \cdot 5}$, the number of directed $5$-cycles is a asymptotically distributed as a Poisson random variable with mean $\frac{1}{2 \cdot 5} \cdot \frac{1}{2^4}$.  This gives the following complete result for random $2$-regular oriented graphs.

\begin{theorem}
For $\vec{G} \sim \vec{\mathcal{G}}(n, 2)$, with high probability, $\chi_o(\vec{G}) \in \{4, 5\}$ with
\begin{align*}
\pr(\chi_o(\vec{G}) = 4) &= e^{-1/160} (1+o(1))\\
\pr(\chi_o(\vec{G}) = 5) 	&= \left(1 - e^{-1/160}\right)  (1+o(1)).
\end{align*}
\end{theorem}

A similar analysis can be applied to the random binomial graph $\vec{\mathcal{G}}(n, p)$ in the case that $p < 1/n$ to show that in that case, with high probability, the oriented chromatic number takes at one of $3$ possible values, each with positive probability.

\begin{theorem}\label{thm:gnd_small_d}
For $d$ constant with $0 < d < 1$ and $\vec{G} \sim \vec{\mathcal{G}}(n, p=\frac{d}{n})$, with high probability, $\chi_o(\vec{G}) \in \{3, 4, 5\}$ with 
\begin{align*}
	\pr(\chi_o(\vec{G})) = 3) &> (1-d)^{1/2} \exp\left( \frac{d}{2} + \frac{d^2}{4}\right)(1+o(1)),\\
	\pr(\chi_o(\vec{G})) = 4) & > \frac{d^5}{32}e^{-d^5/10},\\
	\pr(\chi_o(\vec{G})) = 5) & = (1-\exp(-\frac{d^5}{160}))(1+o(1))
\end{align*}
\end{theorem}

\begin{proof}
For $p = d/n$, with $d < 1$, Erd\H{o}s and R\'{e}nyi~\cite{ER60} showed that with high probability, for $G \sim \mathcal{G}(n, p)$, every component of $G$ is either a tree or unicyclic.  By Proposition~\ref{prop:4or5-colouring}, then with high probability $\chi_o(\vec{G}) \le 5$ and $\chi_o(\vec{G}) = 5$ if and only if some component of $\vec{G}$ contains a directed $5$-cycle.  Again, due to results of Erd\H{o}s and R\'{e}nyi~\cite{ER60}, the number of $5$-cycles in $\vec{G}$ is asymptotically a Poisson random variable with mean $\frac{d^5}{10}$.  Thus, the probability that there is at least one directed $5$-cycle in $\vec{G}$ is asymptotically $(1-\exp(-\frac{d^5}{10}\cdot\frac{1}{16}))$.  This is precisely the probability that $\vec{G}$ has oriented chromatic number equal to $5$.

By a similar argument, with probability asymptotically, $\frac{d^5}{10}e^{-d^5/10} \cdot \frac{5}{16}$, there is exactly one $5$-cycle in $\vec{G}$ and it is oriented isomorphically to the cycle in Figure~\ref{fig:5cycle4cols}.  If this is the case, then $\chi_o(\vec{G}) = 4$.

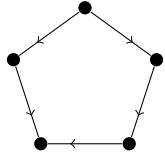
\begin{figure}[htb]
\begin{center}
\begin{tikzpicture}
	[decoration={markings, mark=at position 0.7 with {\arrow{>}}}] 
	\tikzstyle{vertex}=[circle, fill=black,  minimum size=5pt,inner sep=0pt]
	
	\foreach \x in {0, 1, 2, 3, 4}
	{
		\node[vertex] at ({18+72*\x}:1cm) (\x) {};
	}
	\draw[postaction={decorate}] (1) -- (0);
	\draw[postaction={decorate}] (1) -- (2);
	\draw[postaction={decorate}] (2) -- (3);
	\draw[postaction={decorate}] (4) -- (3);
	\draw[postaction={decorate}] (0) -- (4);
\end{tikzpicture}
\end{center}
\caption{Oriented $5$-cycle with $\chi_o(C) = 4$}\label{fig:5cycle4cols}
\end{figure}

On the other hand, based on results of Tak\'{a}cs~\cite{Ta88} or Janson~\cite{Ja87}, the probability that $G \sim \mathcal{G}(n, p=\frac{d}{n})$ is a forest is asymptotically $(1-d)^{1/2} e^{d/2+d^2/4}$.  If $G$ is a forest, then any orientation of $G$ can be properly $3$-coloured by the directed $3$-cycle.  On the other hand, to see that $\chi_o(\vec{G}) > 2$, note that by straightforward second moment arguments, with high probability, $\vec{G}$ contains a directed path of length $2$, which has no proper oriented $2$-colouring.
\end{proof}

It was noted by Bensmail, Duffy, and Sen~\cite{BDS17} that looking at a uniform random oriented graph on $n$ vertices, that is $\vec{G} \sim \vec{\mathcal{G}}(n, 1/2)$, with high probability, every pair of vertices are joined by a directed path of length $2$ (in one direction or the other) and hence none can be coloured with the same colour in any oriented colouring.  Klostermeyer and MacGillivray~\cite{KM04a} called such graphs \emph{oriented cliques} and noted that the condition of having weak diameter $2$ is equivalent to having oriented chromatic number equal to the number of vertices.  Using a standard first moment argument shows that in a large range of values for $p$, $\vec{G} \sim \vec{\mathcal{G}}(n, p)$ is an oriented clique.

\begin{proposition}\label{prop:gnp_large_p}
For any function $\omega(n) \to \infty$, if $p = p(n)$ is such that $p^2 n = 4\log n + \omega(n)$, then with high probability, for $\vec{G} \sim \vec{\mathcal{G}}(n, p)$, $\chi_o(\vec{G}) = n$.
\end{proposition}

\begin{proof}
Since there are two different directed paths of length $2$ between two given vertices, via a third.  Therefore, the expected number of pairs of vertices in $\vec{G}$ not joined by a directed path of length $2$ is
\[
\binom{n}{2}(1 - p^2/2)^{n-2} \le \frac{n^2}{2} \exp\left(-(n-2)\frac{p^2}{2}\right) \le \frac{e}{2}\exp\left(2 \log n - \frac{np^2}{2} \right)
\]
and the result follows.
\end{proof}

\subsection{Precise value of the second moment} \label{ssec:second_mom_const}

To improve the results of Achlioptas and Moore~\cite{AM04} from a collection of
three possible values for $\chi(\mathcal{G}(n,d))$ to two values, Kemkes at
al.~\cite{KPW10} used the small subgraph conditioning method of Robinson and
Wormald (see~\cite[Chapter 9]{JLR00} and \cite{Wo99}). The small subgraph
conditioning method explains that $\ex Y^2$ exceeds $(\ex Y)^2$ because the
presence or absence of small cycles artificially inflates the variance. In cases
where the small subgraph conditioning method is useful, conditioning on the
presence of these cycles ``luckily and yet mysteriously'' accounts for the
difference between $\ex Y^2$ and $(\ex Y)^2$. The method requires
accurate estimates for the first and second moments as well as certain joint
moments.

In Section~\ref{ssec:second_moment}, we were unable to calculate
$\det(-H|_{\mathbb{V}})$ exactly. With an exact value
for $\det(-H|_{\mathbb{V}})$, we could have gotten an exact value for $C_{d,k}$
in Equation~\eqref{eq:PZ_constant}. Though we have not done so carefully, it seems likely
that there is an extension of the small subgraph conditioning method to directed
graphs. With such an extension and an exact value for $C_{d,k}$,
one could use Proposition 2 of Kemkes et al.~\cite{KPW10} to calculate
the required joint factorials and replace ``with positive probability'' in
Proposition~\ref{prop:gnd_ub} with ``almost surely.'' This argument would supersede
Section~\ref{sec:window} and allow for the marginally stronger result
\[ \chio(\mathcal{G}(n,d)) \in (2^{d/2}, 2e^{d/2}+2d+2]. \]
Extending the small subgraph conditioning method to directed graphs would also
open up the possibility of improving the bounds on $\chio(\mathcal{G}(n,m))$.
In addition, one would need to calculate the joint factorial moments of small
cycles in $\mathcal{M}(n,m)$. We chose not to pursue this approach because the
it only offers a scalar improvement, a marginal difference compared to the
exponential gap between the upper and lower bounds.

\subsection{Upper bounds when $k$ does not divide $n$} \label{ssec:divisibility}

In Section~\ref{sec:ub} we calculated values of $k$ for $\vec{\mathcal{M}}(n,m=cn)$ and $\vec{\mathcal{C}}(n,d)$ have $k$-oriented-colourings with positive probability under the assumption that $n$ was divisible by $k$. In this section we discuss how to remove that condition using ideas of~\cite{NP21}.

Suppose $n = qk+r$ where $r \in [0,k-1]$. We extend the concept of equitable colourings by requiring every colour classes 1 through $r$ to contain $q+1$ vertices and colour classes $r+1$ through $k$ to contain $q$ vertices. (Note that for $r=0$ this coincides with our original definition of equitable.) The strategy to adapt Propositions~\ref{prop:gnm_ub} and \ref{prop:gnd_ub} for $r > 0$ is to show that fixing these $r = O(1)$ vertices does not affect the ratio $(\ex Y)^2/(\ex Y^2)$. To get a feeling for the argument, we prove the following proposition.

\begin{proposition} \label{prop:proportional}
Let $k$ be an integer such that there exists a doubly regular tournament of order $k$ and let $c > 0$. Fix a doubly regular tournament $T_k$. For $n = qk+r, r \in [0,k-1]$, let $Y_n$ count the number of equitable oriented $T_k$-colourings of $\vec{\mathcal{M}}(n,m=cn)$. Let $n' = n-r$. Then
\[ \ex Y_n^2 \sim \left(k\left(\frac{1}{2}\left(1-\frac{1}{k}\right)\right)^c\right)^{2r} \ex Y_{n'}^2. \]
\end{proposition}

Before proving Proposition~\ref{prop:proportional}, we introduce a few ideas of P\'erez and the second author. Recall that in the proof of Lemma~\ref{lem:gnm_second_mom}, we defined the set $\mathbb{X}_n$ which contained, for each $n$, the valid overlap matrices on $n$ vertices. We proved that the exponential contribution to each term of the sum was maximized at $\mbf{\hat{x}}$, the vector with every entry $\frac{n}{k^2}$. As seen in Equation~\eqref{eq:exp_sum}, the second moment is a sum containing polynomially many exponential terms, and thus the terms close to this exponential maximum contribute all but a negligible portion of the sum. This fact is captured in the following result:

\begin{proposition}[see {\cite[Proposition 3.5]{NP21}}] \label{prop:central_sum}
Suppose the same set of conditions hold as in Theorem~\ref{thm:laplace_summation_gamma}. In addition, suppose there is $N > 0$ such that $\psi(x) \le n^N$ for each $\mbf x \in K_1$. For any $\gamma > 0$, define $\mathbb{Y}_n(\gamma) \subseteq \mathbb{X}_n$ as
\[ \mathbb{Y}_n(\gamma) = \left\{ \mbf x \in \mathbb{X}_n : ||\mbf{\hat{x}} - \mbf x||_{\infty} < \gamma \frac{\log n}{\sqrt{n}} \right\}. \]
Then
\[ \sum_{\mbf x \in \mathbb{Y}_n(\gamma)} T_n(\mbf x)  \sim \sum_{\mbf x \in \mathbb{X}_n} T_n(\mbf x). \]
\end{proposition}

We now prove Proposition~\ref{prop:proportional}, following the proof of Proposition 7.1 in~\cite{NP21}.

\begin{proof}[Proof of Proposition~\ref{prop:proportional}]
Given $\mbf A \in \mathbb{Y}_n(1)$, define
\[ \mbf{A}^\ast = \left(\frac{1}{n'}(na_{ij}-\delta_{ij})\right)_{i,j \in [k]} \]
where
\[ \delta_{ij} = \begin{cases} 1 & i =j \le r\\ 0 & \text{else} \end{cases}. \]
Then $\mbf{A}^\ast$ defines the overlap matrix of a pair of equitable oriented $T_k$-colourings of $\vec{\mathcal{M}}(m,cn)$ as we have removed the $r$ vertices of prescribed colours 1 through $r$. As $\mbf A \in \mathbb{Y}_n(1)$, we know $\mbf{A}^\ast$ is non-negative.

Recall from Equation~\eqref{eq:mnm_count} that the number of oriented $T_k$-colourings respecting $\mbf A$ is
\[ \frac{n!}{\prod\limits_{v \in V(T^{\otimes 2})} (na_v)!}\left(\sum_{uv \in E(T^{\otimes 2})}a_ua_v \right)^m =: \ell(\mbf A). \]
We have
\[ \frac{\ell(\mbf A)}{\ell(\mbf{A}^\ast)} = \frac{n_{(r)}}{n^r} \cdot \frac{1}{\prod_{i=1}^r a_{ii}} \cdot \left(\frac{n-r}{n}\right)^{2cr}\cdot \frac{\left(\sum\limits_{uv \in E(T^{\otimes 2})}a_ua_v\right)^{cn}}{\left(\sum\limits_{uv \in E(T^{\otimes 2})}(a_ua_v - \frac{1}{n}(a_u\delta_v + a_v\delta_u) + \frac{1}{n^2}\delta_u\delta_v)\right)^{c(n-r)}}. \]
Note that $\delta_u\delta_v = 1$ if and only if $u = (i,i), v = (j,j)$, and $i\ne j \le r$. Therefore
\[ \sum_{uv \in E(T^{\otimes 2})} \delta_u \delta_v = \binom{r}{2}. \]
Furthermore,
\[ 0 \le \sum_{uv \in E(T^{\otimes 2})} (a_u\delta_v + a_v\delta_u) \le \sum_{uv \in E(T^{\otimes 2})} (a_u + a_v) = \frac{(k-1)^2}{2} \sum_{v \in V(T^{\otimes 2})} a_v = \frac{(k-1)^2}{2} \]
as each $v \in T^{\otimes 2}$ has degree $\frac{(k-1)^2}{2}$. Therefore, keeping in mind $\mbf A \in \mathbb{Y}_n(1)$, we have
\[ \frac{\ell(\mbf A)}{\ell(\mbf{A}^\ast)} \ge \frac{n_{(r)}}{n^r}  \cdot \left(\frac{n-r}{n}\right)^{2cr} \cdot \frac{\left(\sum\limits_{uv \in E(T^{\otimes 2})}(\frac{1}{k^2}-\frac{\log n}{\sqrt{n}})^2\right)^{cn}}{(\frac{1}{k^2}+\frac{\log n}{\sqrt{n}})^r\left(\sum\limits_{uv \in E(T^{\otimes 2})}(\frac{1}{k^2}+\frac{\log n}{\sqrt{n}})^2+\frac{1}{n^2}\binom{r}{2}\right)^{c(n-r)}} \to \left(k\left(\frac{1}{2}\left(1-\frac{1}{k}\right)\right)^c\right)^{2r} \]
and
\[ \frac{\ell(\mbf A)}{\ell(\mbf{A}^\ast)} \le \frac{n_{(r)}}{n^r} \cdot \left(\frac{n-r}{n}\right)^{2cr} \cdot \frac{\left(\sum\limits_{uv \in E(T^{\otimes 2})}(\frac{1}{k^2}+\frac{\log n}{\sqrt{n}})^2\right)^{cn}}{(\frac{1}{k^2}-\frac{\log n}{\sqrt{n}})^r\left(\sum\limits_{uv \in E(T^{\otimes 2})}(\frac{1}{k^2}-\frac{\log n}{\sqrt{n}})^2-\frac{(k-1)^2}{2n}\right)^{c(n-r)}} \to \left(k\left(\frac{1}{2}\left(1-\frac{1}{k}\right)\right)^c\right)^{2r} \]
from which we conclude
\[ \frac{\ell(\mbf A)}{\ell(\mbf{A}^\ast)} \to \left(k\left(\frac{1}{2}\left(1-\frac{1}{k}\right)\right)^c\right)^{2r}. \]
Let $\mathcal{A} = \{ \mbf{A}^\ast \mid \mbf A \in \mathbb{Y}_n(1) \}$. Then if there is $\gamma > 0$ such that $\mathbb{Y}_n(\gamma) \subseteq \mathcal{A}$, Proposition~\ref{prop:central_sum} gives
\begin{align*}
\left(k\left(\frac{1}{2}\left(1-\frac{1}{k}\right)\right)^c\right)^{2r} \ex Y_{n'}^2 &\sim \sum_{\mbf A \in \mathbb{Y}_{n'}(\gamma)} \left(k\left(\frac{1}{2}\left(1-\frac{1}{k}\right)\right)^c\right)^{2r} p(n,\mbf A) \exp(n f(\mbf A))\\
	& \le \sum_{\mbf A \in \mathcal{A}} \left(k\left(\frac{1}{2}\left(1-\frac{1}{k}\right)\right)^c\right)^{2r} p(n,\mbf A) \exp(n f(\mbf A)) \sim \ex Y_n^2\\
	&\le \sum_{\mbf A \in \mathbb{X}_{n'}} \left(k\left(\frac{1}{2}\left(1-\frac{1}{k}\right)\right)^c\right)^{2r} p(n,\mbf A) \exp(n f(\mbf A))\\
	&= \left(k\left(\frac{1}{2}\left(1-\frac{1}{k}\right)\right)^c\right)^{2r} \ex Y_{n'}^2,
\end{align*}
completing the proof of the proposition.

Thus it suffices to show for some $\gamma >0$ and all $\mbf{\tilde{A}} \in \mathbb{Y}_{n'}(\gamma)$ there is $\mbf A \in \mathbb{Y}_n(1)$ such that $\mbf{A}^\ast = \mbf{\tilde{A}}$. Each $\tilde{a}_v \in \mbf{\tilde{A}}$ satisfies
\[ \frac{1}{k^2} - \gamma\frac{\log n'}{\sqrt{n'}} \le \tilde{a}_v \le \frac{1}{k^2} + \gamma\frac{\log n'}{\sqrt{n'}} \]
and we want
\[ \frac{1}{n'}(na_v-\delta_v) = \tilde{a}_v \]
where $a_v$ must satisfy
\[ \frac{1}{k^2} - \frac{\log n}{\sqrt{n}} \le a_v \le \frac{1}{k^2} + \frac{\log n}{\sqrt{n}}. \]
Thus it suffices to pick $\gamma$ so that
\[ \frac{1}{k^2}-\frac{\log n}{\sqrt{n}} \le \frac{n'}{n}\left(\frac{1}{k^2}-\gamma\frac{\log n'}{\sqrt{n'}}\right) \]
and
\[ \frac{1}{n}\left(n'\left(\frac{1}{k^2}+\gamma \frac{\log n'}{\sqrt{n'}}\right)+1\right) \le \frac{1}{k^2}+\frac{\log n}{\sqrt{n}}. \]
These are identical conditions to those of~\cite{NP21}, where the authors show $\gamma=0.4$ suffices.
\end{proof}

To complete the argument that $\vec{M} \sim \vec{\mathcal{M}}(n,m=cn)$ has a $k$-oriented-colouring with positive probability, one should prove that
\[ \ex Y_n \sim \left(k\left(\frac{1}{2}\left(1-\frac{1}{k}\right)\right)^c\right)^{r} \ex Y_{n'} \]
in a similar (but easier) manner; then
\[ \pr[Y_n > 0] \ge \frac{(\ex Y_n)^2}{\ex(Y_n^2)} \sim \frac{(\ex Y_{n'})^2}{\ex(Y_{n'}^2)} \sim \frac{(\ex Y)^2}{\ex(Y^2)} \sim \left(\frac{((k-1)^2-2c)^2-4c^2k^2}{(k-1)^4}\right)^{(k-1)^2/4} > 0. \]

In proving the value of $k$ for which $\vec{C} \sim \vec{\mathcal{C}}(n,d)$ has a $k$-oriented-colouring with positive probability, we used the Laplace Summation technique in~\cite{GJR10}, rather than that in~\cite{NP21}; however, the latter is an extension of the former and thus a version of Proposition~\ref{prop:central_sum} holds under the conditions of Theorem~\ref{thm:laplacian_summation} as well. Then one should prove a version of Proposition~\ref{prop:proportional} for the first and second moments of $\vec{\mathcal{C}}(n,d)$. 

\subsection{Colouring with tournaments that are not doubly regular}\label{ssec:col_not_doubly_regular}

The exponential gap between the bounds in Theorem~\ref{thm:gnd_main} comes from the different constants in $u_k$ and $\ell_k$; $u_k$ behaves like $\frac{2}{\log 2} \log k$ while $\ell_k$ grows as $2\log k$. While the argument for the lower bound is straightforward, one might reasonably ask if a better constant for $\ell_k$ could be found by counting equitable colourings of a different tournament, hoping to improve Corollary~\ref{cor:opt_doubly_reg_tour}.  In the calculations of the second moment for the number of equitable oriented colourings in Lemma~\ref{lem:gnm_second_mom} and Section~\ref{ssec:counting}, the initial computation of the second moment would apply equally well to colouring by \emph{any} tournament on $k$ vertices, not just a doubly regular tournament.  Similarly, the function, $\gamma_d$, defined in Proposition~\ref{prop:generalized_AN} can be defined with the Lagrangian part given for any graph $G$.  In the case where a tournament $T$ is not regular, that is, not all vertices have the same in-degree or out-degree, the product graph $G = T^{\otimes 2}$ will not be regular.  

If $G$ is not regular, then subject to the condition that the matrix $\mbf A$ be doubly stochastic, the point at $\mbf A = \frac{1}{k}J_k$ is not a critical point for the function $-\frac{1}{k} \sum_{v \in V(G)} a_v \log a_v + \frac{d}{2}\log\left(\frac{2\sum_{uv \in E} a_u a_v}{k^2}\right)$ and so cannot be the location of the maximum value.  This would mean that the maximum value is strictly greater and so the exponential part of the second moment of the number of equitable $T$-colourings has a larger base that the square of the expected value and a standard second moment method argument would not yield a positive probability of finding an equitable $T$-colouring.     

Among the regular tournaments, one might seek a tournament whose Kronecker square has a smaller second-largest eigenvalue to improve the bound on $d$ obtained from Proposition~\ref{prop:generalized_AN}. However, one can show that the signed adjacency matrix of any regular tournament of order $k$ has an eigenvalue $\lambda$ with $|\lambda| \ge \sqrt{k}$ and use this fact to prove that the second-largest eigenvalue has size at least $\frac{1}{2}(1+\lambda^2) \ge \frac{k+1}{2}$, precisely the value of the second-largest eigenvalue of the Kronecker square of a doubly regular tournament.  The fact that the signed adjacency matrix of a regular tournament of order $k$ has an eigenvalue of modulus at least $\sqrt{k}$ follows from a result by Brauer and Gentry~\cite{BG72} on the unsigned adjacency matrices of regular tournaments, but one can also prove it directly by showing that if $M$ is the signed adjacency matrix of a regular tournament of order $k$ and $\mathbf{v}$ is any column of $M$, then $||M \mathbf{v}||_2/||\mathbf{v}||_2 \ge \sqrt{k}$.  From this, one can show that the largest eigenvalue of $T^{\otimes 2}$ is exactly $\frac{(k-1)^2}{2}$ and the second-largest eigenvalue is at least $\frac{k+1}{2}$.  Thus, for the approach of finding oriented colourings using the second-moment method, doubly regular tournaments will give the best results.

\section*{Acknowledgements}

The authors would like to thank Stephen Kirkland for help with references on the algebraic properties of tournaments and Siddarth Sankaran for assistance with reference for number theoretic results. They also thank Xavier P\'erez for offering his expertise on the history of problems on random graphs of bounded degree and Cris Moore for helpful clarifications about the results for undirected regular graphs.

\bibliographystyle{abbrv}
\nocite{*}
\bibliography{newbib2}

\end{document}